\documentclass[a4paper, 12pt]{article}

\usepackage{amsmath}
\usepackage{amsthm}

\usepackage{amssymb}

\usepackage{geometry}
\usepackage{bm}

\usepackage{mathrsfs}

\usepackage{color}
\usepackage{graphicx}
\usepackage{subfigure}
\usepackage{amsfonts}

\usepackage{epstopdf}
\usepackage[title]{appendix}
\usepackage{latexsym}
\usepackage{psfrag}

\allowdisplaybreaks

\newtheorem{theorem}{Theorem}[section]
\newtheorem{proposition}{Proposition}[section]
\newtheorem{lemma}{Lemma}[section]

\newtheorem{remark}{Remark}[section]
\newtheorem{corollary}{Corollary}[section]

\numberwithin{equation}{section}
\numberwithin{figure}{section}

\title{\Large  \bf Provably Positive High-Order Schemes for Ideal Magnetohydrodynamics: Analysis on General Meshes} 

\author{{\sc Kailiang Wu}\thanks{Department of Mathematics, The Ohio State University, Columbus, OH 43210, USA ({\tt wu.3423@osu.edu}).}{\sc~~~and~~Chi-Wang Shu}\thanks{Division of Applied Mathematics, Brown University, Providence, RI 02912, USA ({\tt shu@dam.brown.edu}). Research is supported in part by ARO grant W911NF-15-1-0226 and NSF grant DMS-1719410. }}

\date{\today} 

\begin{document}
\maketitle

\vspace{-6mm}

\begin{abstract}
	This paper proposes and analyzes 
	arbitrarily high-order discontinuous Galerkin (DG) and finite volume methods which provably preserve the positivity of density and pressure for the ideal  magnetohydrodynamics (MHD) on general meshes.
	Unified auxiliary theories are built for 
	rigorously analyzing 
	the positivity-preserving (PP) property of numerical 
	MHD schemes with a 
	Harten--Lax--van Leer (HLL) 
	type flux on polytopal meshes in any space dimension. 
	The main challenges overcome here  
	include establishing certain relation between 
	the PP property and a discrete divergence of magnetic field on general meshes, and 
	estimating proper wave speeds in the HLL flux to ensure  
	the PP property. 
	In the 1D case, we prove that the standard DG and finite volume methods with the proposed HLL flux are PP, under a condition accessible by a PP limiter. 
	For the multidimensional conservative MHD system, the 
	standard DG methods with a PP limiter are not PP in general, due to the effect of unavoidable divergence error in the magnetic field. 
	We construct provably PP high-order DG and finite volume schemes by  
	proper discretization of 
	 the symmetrizable MHD system, with two 
	 divergence-controlling techniques: 
	the locally divergence-free elements and an important penalty term. 
	The former technique leads to zero divergence within each cell, while the latter controls the divergence error across cell interfaces. 
	Our analysis reveals in theory that a coupling of these two techniques 
	is very important for positivity preservation, as they exactly 
	contribute the discrete divergence terms 
	which are absent in standard multidimensional DG schemes but 
	crucial for ensuring the PP property.
	Several numerical tests further confirm the PP property 
	and the effectiveness of the proposed PP schemes. 
	Unlike the conservative MHD system, the exact smooth solutions 
	of the symmetrizable MHD system are proved to retain the positivity even if the divergence-free condition is not satisfied. 
	Our analysis and findings further the understanding, 
	at both discrete and continuous levels, of the relation 
	between the PP property and the divergence-free constraint.
\end{abstract}



\section{Introduction}

This paper is concerned with 
highly accurate and robust numerical
methods 
for the ideal compressible magnetohydrodynamics (MHD), 
which play an important role in many fields including astrophysics, plasma physics and space physics. 
When viscous, resistive and relativistic effects can be neglected, the governing equations of ideal MHD, which combine
the equations of gas dynamics with the Maxwell equations,
have been widely used to model the dynamics of electrically conducting fluids in the presence of magnetic field.
The ideal MHD system can be written as 
\begin{equation}\label{eq:MHD}
{\bf U}_t + \nabla \cdot {\bf F} ( {\bf U} )  = {\bf 0},
\end{equation}
with an additional divergence-free constraint on the magnetic field
\begin{equation}\label{eq:2D:BxBy0}
\nabla \cdot {\bf B} =0. 
\end{equation}
The conservative vector ${\bf U} = ( \rho,\rho {\bf v},{\bf B},E )^{\top}$; in the $d$-dimensional case,  the divergence operator 
$\nabla \cdot = \sum_{i = 1}^d \frac{\partial } {\partial x_i}$, and the flux ${\bf F} =( {\bf F}_1, \dots, {\bf F}_d )$ with 
\begin{equation*}
{\bf F}_i({\bf U}) = \Big( \rho v_i,~\rho v_i {\bf v}  -  B_i {\bf B}   + p_{\rm tot}  {\bf e}_i,~v_i {\bf B} - B_i {\bf v},~v_i(E+p_{\rm tot} ) -  B_i({\bf v}\cdot {\bf B})
\Big)^{\top}.
\end{equation*}
Here $\rho$ is the density,  ${\bf v}=(v_1,v_2,v_3)$ is 
the fluid velocity, ${\bf B}=(B_1,B_2,B_3)$ denotes the magnetic field,
$p_{\rm tot}=p+\frac{|{\bf B}|^2}2$ is the total pressure which consists of 
the gas pressure $p$ and the magnetic pressure, 
the row vector ${\bf e}_i$ denotes the $i$th row of the unit matrix of size 3, 
$E=  \frac12 \left( \rho |{\bf v}|^2 + |{\bf B}|^2 \right)+\rho e$ is the total energy 
consisting of kinetic, magnetic and thermal energies, and $e$ denotes the specific internal energy.
The system \eqref{eq:MHD} 
is closed with an equation of state (EOS). 
Although the ideal EOS,  
$p = (\gamma-1) \rho e$,  
with a constant adiabatic index $\gamma$, is the most widely used choice, 
there are situations where it is more suitable to use other EOSs.  
A general EOS can be expressed as $p = p(\rho,e)$, which is assumed to satisfy 
(cf.~\cite{zhang2011}): 
\begin{equation}\label{eq:assumpEOS}
\mbox{if}\quad \rho > 0,\quad \mbox{then}\quad e > 0~\Leftrightarrow~p(\rho,e)>0. 
\end{equation}
This is a reasonable condition, and it holds for the ideal EOS with $\gamma>1$.






Although the satisfaction of the divergence-free condition
\eqref{eq:2D:BxBy0} is not explicitly included in 
the 
system \eqref{eq:MHD}, 
the exact solution of \eqref{eq:MHD} always preserves zero divergence in future time if the initial divergence is zero.  
However, due to truncation errors, most 
of the numerical MHD schemes for $d\ge 2$  
lead to a nonzero numerical divergence in the magnetic field, even if the initial data satisfy  \eqref{eq:2D:BxBy0}. 
As it is widely known, 
large divergence error 
can lead to numerical instabilities or nonphysical features in the numerical solutions, cf.~\cite{Brackbill1980,Evans1988,BalsaraSpicer1999,Toth2000,Li2005}. 
In the past several decades, many numerical techniques were proposed to control the divergence error or 
enforce the divergence-free condition in the discrete sense, including but not limited to: 
the projection method \cite{Brackbill1980}, the locally divergence-free methods \cite{Li2005,Yakovlev2013}, 
the hyperbolic divergence cleaning method \cite{Dedner2002}, 
the constrained transport method \cite{Evans1988} and its variants 
(e.g.,  \cite{Ryu1998,BalsaraSpicer1999,Balsara2004,Gardiner2005,Torrilhon2005,Artebrant2008,Li2011,Christlieb2014,Fu2018}), 
and the eight-wave methods (e.g., \cite{Powell1994,Powell1995,Chandrashekar2016,LiuShuZhang2017}).
The eight-wave method was first proposed by Powell \cite{Powell1994,Powell1995}, based on appropriate discretization of 
 the modified MHD equations of Godunov \cite{Godunov1972}: 
\begin{equation}\label{eq:MHD:GP}
{\bf U}_t + \nabla \cdot {\bf F} ({\bf U}) 
= -(\nabla \cdot {\bf B} )~{\bf S} ( {\bf U} ),
\end{equation}
where  
${\bf S} ({\bf U}) = ( 0,~{\bf B},~{\bf v},~{\bf v} \cdot {\bf B} )^\top.$  
In some literature, \eqref{eq:MHD:GP} is also called Powell's MHD system.  
The right-hand side term of \eqref{eq:MHD:GP}, 
termed as the Godunov--Powell source term in the following,  is proportional 
to $\nabla \cdot {\bf B}$. 
This means, at the continuous level, the Godunov form \eqref{eq:MHD:GP} and conservative form \eqref{eq:MHD} are equivalent 
under the condition \eqref{eq:2D:BxBy0}. 
However, the Godunov--Powell source term 
modifies the character of the MHD equations, 
making 
the system \eqref{eq:MHD:GP} 
Galilean invariant (cf.~\cite{Dellar2001}), symmetrizable \cite{Godunov1972} and useful 
for designing entropy stable schemes (see, e.g., \cite{Chandrashekar2016,LiuShuZhang2017}). 
These good properties do not hold anymore if the source term is dropped. 
%
As first demonstrated by Powell \cite{Powell1995}, the 
inclusion of the source term also helps advect   
the divergence away with the flow.   
This renders 
the eight-wave method a stable approach 
to control the divergence error, although 
some drawbacks \cite{Toth2000} 
can be caused due to the 
loss of conservativeness. 

In physics, the density, pressure and internal energy are positive. 
An equivalent mathematical description is that, the conservative vector ${\bf U}$ should stay in the {\em set of physically admissible states} defined by 
\begin{equation}\label{eq:DefG}
{\mathcal G} = \left\{   {\bf U} = (\rho,{\bf m},{\bf B},E)^\top: \rho > 0,~
{\mathcal E}(  {\bf U}  ) := E- \frac12 \left( \frac{|{\bf m}|^2}{\rho} + |{\bf B}|^2 \right) > 0 \right\},
\end{equation}
where the condition \eqref{eq:assumpEOS} has been used, and ${\mathcal E} ({\bf U}) =  \rho e$ denotes the internal energy. 
We are interested in {\em positivity-preserving} (PP) numerical schemes whose solutions always stay in ${\mathcal G}$.  
The motivation comes from that, once the negative density or negative pressure (internal energy) appears in the numerical solution,
the corresponding discrete problem becomes ill-posed because of the loss of hyperbolicity, often causing the breakdown of the simulation codes. 
However, most the existing schemes for the ideal MHD are not PP in general, and may have a large risk of failure in solving MHD problems with low internal energy, low density, small  plasma-beta and/or strong discontinuity. 
A few efforts were made to reduce this risk. Balsara and Spicer \cite{BalsaraSpicer1999a} tried to maintain positive pressure by switching the Riemann solvers for different wave situations. 
Janhunen \cite{Janhunen2000} 
noticed the challenge of designing PP schemes for the conservative system \eqref{eq:MHD}, so he 
proposed a modified MHD system, which is similar to the Godunov form \eqref{eq:MHD:GP} but includes only the source term in 
the induction equation. By the discretization of his modified system, 
Janhunen \cite{Janhunen2000} developed a 1D HLL-type Riemann solver, and numerically demonstrated its PP property which has not been proven yet. 
Bouchut, Klingenberg, and Waagan \cite{Bouchut2007} skillfully derived several approximate multiwave Riemann solvers 
for the 1D ideal MHD,  
and gave the sufficient conditions for the solvers to satisfy the PP property and discrete entropy inequalities. 
Those conditions are met by explicit wave speeds estimated in \cite{Bouchut2010}, where the solvers were also implemented and multidimensional extension was discussed based on Janhunen's modified system.
Waagan \cite{Waagan2009} noticed the importance of proper discretization on Janhunen's modified system, 
and designed a positive second-order MUSCL-Hancock scheme by the approximate Riemann solvers of \cite{Bouchut2007,Bouchut2010} and a new linear reconstruction. 
The robustness of that scheme 
was further demonstrated in \cite{waagan2011robust} by extensive numerical 
tests and comparisons. 
In the last few years, significant advances have been made in developing bound-preserving high-order schemes for hyperbolic systems; see the pioneer works by Zhang and Shu \cite{zhang2010,zhang2010b,zhang2012maximum}, and more recent works, e.g.,  \cite{Hu2013,Xu2014,Liang2014,christlieb2015high,WuTang2015,Wu2017,Xu2017,ZHANG2017301}. 
Balsara \cite{Balsara2012} proposed a self-adjusting PP limiter 
to enforce the positivity of  
the reconstructed solutions in a finite volume method for \eqref{eq:MHD}. 
Cheng et al.~\cite{cheng} extended the PP limiter of 
\cite{zhang2010b,zhang2011} to enforce the positivity of DG solutions for  \eqref{eq:MHD}.  
These PP limiters \cite{Balsara2012,cheng} are based on a presumed proposition that the cell-averaged solutions computed by those schemes always belong to $\mathcal G$. 
Such a proposition has not yet been rigorously proven for the methods in \cite{Balsara2012,cheng},   
although it could be deduced for the 1D schemes in \cite{cheng} 
under some 
assumptions. 
Using the presumed PP property of the Lax--Friedrichs (LF) scheme, Christlieb et al.\ \cite{Christlieb,Christlieb2016} developed PP 
high-order accurate finite difference schemes for \eqref{eq:MHD} by extending the parametrized flux limiters \cite{Xu2014,Xiong2016,seal2016explicit}. 
It was numerically demonstrated that all the PP techniques mentioned above could enhance the robustness of some MHD codes, but 
few theoretical evidences were provided, especially in the multidimensional cases, to 
completely prove the PP property of full discretized schemes. 
In fact, 
finite numerical tests could be insufficient 
to genuinely demonstrate that a scheme is always PP under all circumstances.  
Therefore, it is highly significant  
to develop 
{\em provably} PP schemes and rigorous PP analysis techniques for the ideal MHD. 

Seeking  
provably PP schemes for the ideal MHD 
is quite difficult, largely due to 
the intrinsic complexity of the MHD equations as well as the 
lack of sufficient knowledge about the underlying relation between the PP property and the divergence-free condition \eqref{eq:2D:BxBy0}. 
It can seen from \eqref{eq:DefG} that 
the difficulties mainly 
lie in maintaining the positivity of internal energy, 
whose calculation nonlinearly involves  
all the conservative   
variables. 
In most of the numerical MHD methods, the conservative 
quantities are themselves evolved according to their own conservation laws, which are seemingly
unrelated to and numerically do not necessarily guarantee the positivity of the computed internal energy. In theory, it is indeed a challenge to make an a priori judgment on whether a
scheme is always PP under all circumstances or not.

Recently, 
two 
progresses \cite{Wu2017a,WuShu2018}   
were made to rigorously analyze, understand and design 
provably PP methods for the ideal MHD. 
The first rigorous PP analysis 
was carried out in \cite{Wu2017a} for  
conservative finite volume and 
DG schemes for \eqref{eq:MHD}. 
The analysis revealed in theory that 
a discrete divergence-free (DDF) condition is crucial for designing the PP conservative schemes for \eqref{eq:MHD}. 
This finding is consistent with the relativistic MHD case \cite{WuTangM3AS}. 
It was also proved in \cite{Wu2017a} that if the proposed DDF condition is slightly violated, even the first-order
multidimensional LF scheme for \eqref{eq:MHD} is generally not PP, 
and using very small CFL number or 
many times larger numerical viscosity  
does not help to prevent this effect. 
The DDF condition 
relies on a combination of the information on adjacent cells, and thus   
is not ensured by a locally
divergence-free approach.                   
As a result, in the multidimensional cases, a usual PP limiter 
does
not genuinely guarantee the PP property of the standard DG schemes for  \eqref{eq:MHD}, even if the locally divergence-free DG element \cite{Li2005} is employed. 
Interestingly, on the other hand, at the PDE level 
the positivity preservation and 
the divergence-free constraint \eqref{eq:MHD} are also 
inextricably linked for the ideal MHD equations. 
For the conservative system \eqref{eq:MHD}, 
Janhunen \cite{Janhunen2000} noticed that 
the exact solutions to 1D Riemann
problems  
can have negative pressure if the initial data has 
a jump in the normal component of the magnetic field (i.e., a nonzero divergence). 
Recently in \cite{WuShu2018}, we first observed  
that  
the exact smooth solution of \eqref{eq:MHD} 
may also fail to be PP  
if the divergence-free constraint \eqref{eq:2D:BxBy0} is (slightly) violated.  
Fortunately, in the present paper we find 
that the smooth solutions of 
the modified system \eqref{eq:MHD:GP} 
always retain the desired positivity even if the magnetic field is not divergence-free. 
{\em All these findings 
	motivate us to seek  
	the multidimensional  PP 
	methods via 
	proper discretization of 
	the modified system \eqref{eq:MHD:GP} 
	rather than 
	the conservative system \eqref{eq:MHD}.}  
Using the analysis techniques proposed in \cite{Wu2017a}, 
we first successfully developed in \cite{WuShu2018}  the multidimensional provably PP high-order DG methods for  \eqref{eq:MHD:GP}. 
Note that the study in \cite{Wu2017a,WuShu2018} 
was restricted to 
the schemes with the {\em global LF flux} on {\em uniform Cartesian meshes}. 
It is desirable to construct provably PP high-order schemes with lower dissipative numerical fluxes 
and on more general/unstructured meshes.

The aim of this paper is to 
present the rigorous analysis and a general framework for constructing 
provably PP high-order DG and finite volume methods with the HLL-type flux 
for the ideal MHD on general meshes. 
The contributions and significant innovations of this work 
are outlined as follows:
\begin{enumerate}
	\item We present   
	unified auxiliary theories for PP analysis of  
	schemes with the HLL-type flux on general  
	meshes for the ideal MHD  in any space dimension. 
	These provide a novel way to analytically extract the underlying relation between the PP property and 
	the discrete divergence of magnetic field on an arbitrary polytopal mesh.  Explicit estimates of  
	the wave speeds in the HLL flux are technically derived to guarantee the provably PP property. 
	\item For the 1D MHD system \eqref{eq:MHD}, we prove the PP property of the standard finite volume and DG methods with the proposed HLL flux, under a condition accessible by a simple PP limiter. 
	\item In the multidimensional cases, 
	we construct provably PP high-order DG methods based on the proposed HLL flux, a PP limiter \cite{cheng}, and 
	a proper discretization of the modified MHD system \eqref{eq:MHD:GP} 
	with two divergence-controlling techniques: 
	the locally divergence-free elements and a novel discretization of the Godunov--Powell source term in an upwind manner according to the associated local wave speeds in the HLL flux.  
	The former technique leads to zero divergence within each cell, while the latter controls the divergence error across cell interfaces. 
	Our analysis clearly reveals in theory that a coupling of these two techniques 
	is very important for positivity preservation, as they exactly 
	contribute the discrete divergence terms 
	which are absent in standard multidimensional DG schemes but 
	crucial for ensuring the PP property.
	We also generalize  
	the DDF condition of \cite{Wu2017a} to general meshes 
	and derive sufficient conditions for achieving PP conservative schemes in the multiple dimensions.
	\item We prove that the strong solution  
	to the initial-value problem of the modified MHD system \eqref{eq:MHD:GP} preserves the positivity of density and pressure even if the divergence-free condition \eqref{eq:2D:BxBy0} is not satisfied. This feature, 
	not enjoyed by the conservative system \eqref{eq:MHD} (see \cite{WuShu2018}), can serve as a justification for designing provably PP multidimensional schemes 
	based on the modified system \eqref{eq:MHD:GP}.
\end{enumerate}
The efforts mentioned above are novel and highly nontrivial. 
A key difficulty is to  
analytically 
quantify the relation of the PP property 
to the discrete divergence on general meshes. 
Especially, in the analysis of   
the positivity of ${\mathcal E}({\bf U})$, 
the discrete equations for the conservative variables are 
nonlinearly coupled, 
and the limiting values of the numerical solution at the interfaces of each cell  
are intrinsically connected  
by the discrete divergence. These make the PP analysis in the MHD case very complicated especially in the multidimensional cases, and 
some standard analysis techniques (cf.~\cite{zhang2010b}) are  inapplicable as demonstrated in \cite{Wu2017a}. 
We will skillfully address these challenges by a novel equivalent form of the set $\mathcal G$ and highly technical estimates.  
Note that a LF flux can be considered as a special HLL flux.  
Therefore, all the analyses in the this paper directly 
apply to the local and global LF fluxes. 
It is also worth mentioning 
that many multi-state or multi-wave HLL-type fluxes   
were developed or applied to the ideal MHD in the literature (e.g., \cite{Janhunen2000,Gurski2004,LI2005344,MIYOSHI2005315,Bouchut2007,BALSARA20101970,fuchs2011,BALSARA2014172}), 
but only a few of them (cf.~\cite{Gurski2004,MIYOSHI2005315,Bouchut2007})
were shown to be PP for some 1D schemes.  
Moreover, their PP property 
for higher order schemes, in the multidimensional cases, 
and its relation to the divergence-free condition in the discrete sense have not yet been rigorously proved.


The paper is organized as follows. After  
establishing the auxiliary theories for our PP analysis on general meshes 
in Section \ref{sec:theory}, we present  
the 1D and multidimensional provably PP methods 
in Sections \ref{sec:1D} and \ref{sec:2D}, respectively.  
We conduct numerical tests in Section \ref{app:num} to verify the PP property and the effectiveness of the proposed PP techniques, before concluding the paper in Section \ref{sec:conclusion}.  
The positivity of strong solutions of the 
modified MHD system \eqref{eq:MHD:GP} 
is shown in Appendix \ref{app:1}.

\section{Auxiliary theories}\label{sec:theory}




In this section, we present the auxiliary results for our PP analysis on general meshes.

\subsection{Properties of admissible state set}
The function ${\mathcal E} ({\bf U})$ in \eqref{eq:DefG} is nonlinear with respect to $\bf U$,  complicating the analysis of  
the PP property of a given scheme. 
The following equivalent form of ${\mathcal G}$ was proposed in \cite{Wu2017a}. 
\begin{lemma}
	\label{theo:eqDefG}
	The admissible state set ${\mathcal G}$ is equivalent to
	\begin{equation}\label{eq:newDefG}
	{\mathcal G}_* = \left\{   {\bf U} = (\rho,{\bf m},{\bf B},E)^\top : \rho > 0,~
	{\bf U} \cdot {\bf n}^* + \frac{|{\bf B}^*|^2}{2} > 0,~\forall{\bf v}^*, {\bf B}^* \in {\mathbb {R}}^3 \right\},
	\end{equation}
	where 
	\begin{equation*}
	{\bf n}^* = \bigg( \frac{|{\bf v}^*|^2}2,~- {\bf v}^*,~-{\bf B}^*,~1 \bigg)^\top.
	\end{equation*}
\end{lemma}

The proof of Lemma \ref{theo:eqDefG} can be found in \cite{Wu2017a}. As we can see, the equivalent set ${\mathcal G}_*$ is defined 
{\em with two constraints linear with respect to $\bf U$}, which give it advantages over the 
natural definition \eqref{eq:DefG} in showing the PP property of numerical schemes.  
This novel equivalent form is 
a cornerstone of our PP analysis.




The convexity of admissible state set is desired and useful in bound-preserving analysis,
as it helps simplify the analysis if the scheme can be reformulated 
into certain convex combinations; see e.g., \cite{zhang2010b,zhang2012maximum,WuTang2017ApJS,Wu2017}. We have the convexity of ${\mathcal G}_*$, cf.~\cite{Wu2017a}.

\begin{lemma}
	\label{theo:MHD:convex}
	The set ${\mathcal G}_*$ is convex. 
	Moreover, $\lambda {\bf U}_1 + (1-\lambda) {\bf U}_0 \in {\mathcal G}_*$
	for any ${\bf U}_1 \in {\mathcal G}_*, {\bf U}_0 \in \overline{\mathcal G}_*$ and $\lambda \in (0,1]$, where 
	$\overline{\mathcal G}_*$ is the closure of ${\mathcal G}_*$.
\end{lemma}

\subsection{Technical estimates relative to flux}

\subsubsection{Main estimates}
We summarize our main estimate result in this subsection with the proof 
of it given later.

For the sake of 
convenience, we introduce the following notations, 
which will be frequently used in this paper.  
For any vector 
${\bm \xi}=(\xi_1,\cdots,\xi_d) \in {\mathbb{R}}^d$, we define the inner products
\begin{equation*}
\langle {\bm \xi }, {\bf v} \rangle 
:= \sum_{k=1}^d {\xi }_k v_k,
\qquad
\langle {\bm \xi }, {\bf B} \rangle 
:= \sum_{k=1}^d {\xi }_k B_k,
\qquad \langle {\bm \xi }, {\bf F} \rangle 
:= \sum_{k=1}^d {\xi }_k {\bf F}_k.
\end{equation*}
For any unit vector ${\bm \xi}\in {\mathbb R}^d$,  
define 
\begin{equation*}
{\mathscr{C}} ({\bf U};{\bm \xi}) := \frac{1}{\sqrt{2}}  \left[ {\mathscr{C}}_s^2 + \frac{ |{\bf B}|^2}{\rho} + \sqrt{ \left( {\mathscr{C}}_s^2 + \frac{ |{\bf B}|^2}{\rho} \right)^2 - 4 \frac{ {\mathscr{C}}_s^2 \langle {\bm \xi}, {\bf B} \rangle^2}{\rho}  } \right]^\frac12,
\end{equation*}
where ${\mathscr{C}}_s :=\frac{p}{\rho \sqrt{2e}}$. Note that, for the ideal EOS,  ${\mathscr{C}}_s = \sqrt{\frac{(\gamma-1)p}{2\rho}}$.

Recall that a technical inequality constructed in 
\cite[Lemma 2.6]{Wu2017a} has played a pivotal role in 
the PP analysis on Cartesian meshes in \cite{Wu2017a,WuShu2018}. That inequality involves 
two states, which correspond to the numerical solutions at a couple of symmetric quadrature points 
on cell interfaces. The cells of a general mesh are generally non-symmetric, so that the results in 
\cite{Wu2017a} are inapplicable to the present analysis. 
To carry out PP
analysis on a general mesh, we need to construct a (general) ``multi-state''
inequality, which is derived in the following theorem.

\begin{theorem}\label{lem:main}
	For $ 1\le j \le N$, let $s_j>0$ and the unit vector ${\bm \xi }^{(j)} \in {\mathbb R}^d$ satisfy 
	\begin{equation}\label{eq:sumUpsilon1}
	\sum_{j=1}^N s_{j} {\bm \xi}^{(j)} ={\bf 0}.
	\end{equation}
	Given $N$ admissible states      
	${\bf U}^{(j)}$, 
	$1\le j \le N$, we define 
	\begin{equation}\label{eq:aPP2}
	\begin{split}
	\widehat \alpha_{j} & := 
	\max \left\{ \big\langle {\bm \xi}^{(j)}, {\bf v}^{(j)} \big\rangle , 
	\frac{1}{\sum\limits_{i=1}^Ns_i}
	\sum_{i=1}^N s_i 
	\left \langle
	{\bm \xi}^{(j)} - {\bm \xi}^{(i)}, 
	\frac{ \sqrt{\rho^{(j)}} {\bf v }^{(j)} + 
		\sqrt{\rho^{(i)}} {\bf v }^{(i)} } 
	{ \sqrt{\rho^{(j)}}  +	\sqrt{\rho^{(i)}}  } \right \rangle
	\right\} 
	\\ & \qquad  
	+ {\mathscr C}( {\bf U}^{(j)};{{\bm \xi}^{(j)}} ) + 
	\frac{2}{\sum\limits_{i=1}^Ns_i}
	\sum_{i=1}^N s_i \frac{ |{\bf B}^{(j)}- {\bf B}^{(i)}| }{ \sqrt{\rho^{(j)}}  +	\sqrt{\rho^{(i)}}  }.
	\end{split}
	\end{equation}	
	Then for any $ \alpha_{j} \ge \widehat \alpha_{j} $, the state
	\begin{align}\label{eq:defOverlineU}
	\overline{\bf U}: =  \frac{1}{{\sum\limits_{j = 1}^{N}  { s_j   \alpha_{j}   } }}\sum\limits_{j = 1}^{N} { { { s_j \bigg( \alpha_{j} {{\bf U}^{(j)}  - 
					\Big\langle {\bm \xi}^{(j)}, {\bf F} ({\bf U}^{(j)} ) \Big\rangle
				} \bigg)} } } , 
	\end{align}
	belongs to ${\mathcal G}_\rho :=  \{ {\bf U} = (\rho,{\bf m},{\bf B},E)^\top : \rho >0  \}$, and satisfies   
	\begin{equation}\label{eq:IEQ:multi-states22}
	\overline{\bf U} \cdot {\bf n}^* 
	+ \frac{|{\bf B}^*|^2}{2} \ge - \frac{ {\bf v}^* \cdot {\bf B}^*  }{{\sum\limits_{j = 1}^{N}  { s_j   \alpha_{j}   } }} 
	\sum_{j=1}^{ N } 
	s_j   \big\langle {\bm \xi}^{(j)}, {\bf B}^{(j)} \big\rangle,\quad \forall {\bf v}^*,{\bf B}^* \in {\mathbb{R}}^3.
	\end{equation}
	Furthermore, $\overline{\bf U} \in \overline {\mathcal G}_*$ if  
	\begin{equation}\label{eq:DDFtheorem}
	\sum_{j=1}^{ N } 
	s_j   \big\langle {\bm \xi}^{(j)}, {\bf B}^{(j)} \big\rangle = 0.
	\end{equation}
\end{theorem}

The proof of Theorem \ref{lem:main} and   
the construction of the inequality \eqref{eq:IEQ:multi-states22} are highly nontrivial and  technical.  
For better legibility, we put the proof  
in Section \ref{sec:proofT}. 
Here, we would like to briefly  
explain the result in Theorem \ref{lem:main}, whose meaning will become more clear in 
the PP analysis in Sections \ref{sec:1D} and \ref{sec:2D}.  
Let us consider a cell of the computational mesh, and assume it is a non-self-intersecting $d$-polytope with $N$ edges ($d=2$) or 
faces ($d=3$). The index $j$ 
on the variables in Theorem \ref{lem:main} represents 
the $j$th edge or face of the polytope, and $s_j$ and ${\bm \xi}^{(j)}$ respectively correspond to the $(d-1)$-dimensional Hausdorff measure and the unit 
outward normal vector of the $j$th edge or face. One can verify that 
the condition \eqref{eq:sumUpsilon1} holds naturally. In addition, 
${\bf U}^{(j)}$ stands for the approximate values of  
${\bf U}$ on the $j$th edge or face. 
The condition \eqref{eq:DDFtheorem} is actually  
a DDF condition over the polytope. 

\begin{remark}\label{rem:sumsposi}
	In Theorem \ref{lem:main}, 
	$
	\sum_{j=1}^Ns_j \alpha_{j} 
	$ is always positive, because
	$$
	\sum\limits_{j=1}^Ns_j \widehat \alpha_{j}> 
	\frac{1}{\sum\limits_{i=1}^Ns_i} \sum\limits_{j=1}^Ns_j
	\sum_{i=1}^N s_i 
	\left \langle
	{\bm \xi}^{(j)} - {\bm \xi}^{(i)}, 
	\frac{ \sqrt{\rho^{(j)}} {\bf v }^{(j)} + 
		\sqrt{\rho^{(i)}} {\bf v }^{(i)} } 
	{ \sqrt{\rho^{(j)}}  +	\sqrt{\rho^{(i)}}  } \right \rangle = 0.
	$$
\end{remark}

\begin{remark}
	Theorem \ref{lem:main}, particularly the inequality \eqref{eq:IEQ:multi-states22},  clearly establishes 
	a connection between the PP property and the discrete divergence of magnetic field, i.e., $\sum_{j=1}^{ N } 
	s_j   \langle {\bm \xi}^{(j)}, {\bf B}^{(j)} \rangle$. This will be a key point of our PP analysis.    
	The right-hand side term of \eqref{eq:IEQ:multi-states22} 
	is very important. The construction of this term is highly technical. If it is dropped, the inequality  \eqref{eq:IEQ:multi-states22} would become invalid. 
	As we will see, this term provides 
	a way to take into account the discrete divergence 
	in the PP analysis.
\end{remark}

The following results are immediate corollaries of Theorem \ref{lem:main}, which are useful for  
deriving PP numerical fluxes.  

For any unit vector ${\bm \xi}\in {\mathbb{R}}^{d}$,  
and any pair of admissible states ${\bf U} $ and $\tilde{\bf U}$, we define 
\begin{align}\label{eq:DefAlphaSR}
& \alpha_r ( {\bf U}, \tilde{\bf U}; {\bm \xi} )  := \max \left\{ \langle {\bm \xi}, {\bf v}\rangle,   \frac{\sqrt{\rho} \langle {\bm \xi}, {\bf v} \rangle + \sqrt{\tilde \rho}  \langle {\bm \xi}, \tilde {\bf v} \rangle }{\sqrt{\rho}+\sqrt{\tilde \rho}}   \right\} + {\mathscr C} (  {\bf U};{\bm \xi} ) +  
\frac{|{\bf B}-\tilde {\bf B }|}{\sqrt{\rho}+\sqrt{\tilde \rho}},
\\ \label{eq:DefAlphaSL}
& 
\alpha_l ( {\bf U}, \tilde{\bf U};  {\bm \xi} )  :=  \min \left\{ \langle {\bm \xi}, {\bf v}\rangle,   \frac{\sqrt{\rho} \langle {\bm \xi}, {\bf v} \rangle + \sqrt{\tilde \rho}  \langle {\bm \xi}, \tilde {\bf v} \rangle }{\sqrt{\rho}+\sqrt{\tilde \rho}}   \right\} - {\mathscr C} (  {\bf U};{\bm \xi} ) -  
\frac{|{\bf B}-\tilde {\bf B }|}{\sqrt{\rho}+\sqrt{\tilde \rho}},
\end{align}
and 
\begin{equation}\label{eq:DefAlphaS}
\alpha_\star  ( {\bf U}, \tilde{\bf U}; {\bm \xi} )  := \max \left\{ |\langle {\bm \xi}, {\bf v}\rangle|,   \left|\frac{\sqrt{\rho} \langle {\bm \xi}, {\bf v} \rangle + \sqrt{\tilde \rho}  \langle {\bm \xi}, \tilde {\bf v} \rangle }{\sqrt{\rho}+\sqrt{\tilde \rho}} \right|   \right\} + {\mathscr C} (  {\bf U};{\bm \xi} ) +  
\frac{|{\bf B}-\tilde {\bf B }|}{\sqrt{\rho}+\sqrt{\tilde \rho}}.
\end{equation}

\begin{corollary}\label{main:1D}
	For any $ {\bf U}, \tilde{\bf U} \in {\mathcal G}$, any unit vector $ {\bm \xi}\in  {\mathbb{R}}^{d}$, 
	and 
	$$\forall \alpha \ge \alpha_r  ( {\bf U}, \tilde{\bf U}; {\bm \xi} ),\quad  \forall \tilde \alpha \le  \alpha_l 
	( \tilde{\bf U}, {\bf U}; {\bm \xi} ),$$ 
	the state 
	\begin{equation*}
	\overline{\bf U}:=\frac{ 1 }{ \alpha - \tilde \alpha } 
	\Big(   \alpha {\bf U}
	- \langle {\bm \xi}, {\bf F} (  {\bf U} ) \rangle
	- \tilde \alpha \tilde {\bf U}
	+  \langle {\bm \xi}, {\bf F} ( \tilde {\bf U} ) \rangle
	\Big),
	\end{equation*}
	belongs to ${\mathcal G}_\rho $ and satisfies 
	\begin{equation}\label{eq:KeyIEQ1D}
	\overline{\bf U} \cdot {\bf n}^* 
	+ \frac{|{\bf B}^*|^2}{2} + \frac{{\bf v}^* \cdot {\bf B}^*}{ \alpha - \tilde \alpha} \left( \langle {\bm \xi},  {\bf B} \rangle - \langle {\bm \xi}, \tilde {\bf B} \rangle \right) \ge 0, \quad \forall {\bf v}^*,{\bf B}^* \in \mathbb R^3.
	\end{equation}
	Furthermore, if 
	$
	\langle {\bm \xi},  {\bf B} \rangle - \langle {\bm \xi}, \tilde {\bf B} \rangle=0,
	$ then $\overline{\bf U} \in \overline {\mathcal G}_*$.
\end{corollary}

\begin{proof}
	This directly follows from 
	Theorem \ref{lem:main} with $N=2$, by taking 
\begin{equation*}
	s_1=s_2=1,  \ \, {\bm \xi}^{(1)}=-{\bm \xi}^{(2)}={\bm \xi}, \ \, {\bf U}^{(1)} = {\bf U}, \ \ 
	{\bf U}^{(2)} = \tilde {\bf U}, \ \ \alpha_{1}= \alpha, \ \ 
	\alpha_{2}= - \tilde \alpha. \tag*{\qedhere} 
\end{equation*}
\end{proof}

\begin{corollary}\label{cor:1D}
	Let $ {\bf U}, \tilde{\bf U} \in {\mathcal G}$, unit vector $ {\bm \xi}\in  {\mathbb{R}}^{d}$. For $\forall \alpha \ge \alpha_\star  ( {\bf U}, \tilde{\bf U}; {\bm \xi} )$, $\forall \tilde \alpha \ge  \alpha_\star 
	( \tilde{\bf U}, {\bf U}; {\bm \xi} )$, 
	the state 
	\begin{equation*}
	\overline{\bf U}:=\frac{ 1 }{ \alpha + \tilde \alpha } 
	\Big(   \alpha {\bf U}
	- \langle {\bm \xi}, {\bf F} (  {\bf U} ) \rangle
	+ \tilde \alpha \tilde {\bf U}
	+  \langle {\bm \xi}, {\bf F} ( \tilde {\bf U} ) \rangle
	\Big),
	\end{equation*}
	belongs to ${\mathcal G}_\rho $ and satisfies 
	\begin{equation}\label{eq:KeyIEQ1D2}
	\overline{\bf U} \cdot {\bf n}^* 
	+ \frac{|{\bf B}^*|^2}{2} + \frac{{\bf v}^* \cdot {\bf B}^*}{ \alpha + \tilde \alpha} \left( \langle {\bm \xi},  {\bf B} \rangle - \langle {\bm \xi}, \tilde {\bf B} \rangle \right) \ge 0, \quad \forall {\bf v}^*,{\bf B}^* \in \mathbb R^3.
	\end{equation}
	Furthermore, if 
	$
	\langle {\bm \xi},  {\bf B} \rangle - \langle {\bm \xi}, \tilde {\bf B} \rangle=0,
	$ then $\overline{\bf U} \in \overline {\mathcal G}_*$.
\end{corollary}

\begin{proof}
	This is a direct consequence of Corollary \ref{main:1D}.
\end{proof}

\begin{remark}
	The inequalities \eqref{eq:IEQ:multi-states22}, \eqref{eq:KeyIEQ1D} and \eqref{eq:KeyIEQ1D2}  extend 
	the inequality constructed in \cite[Lemma 2.6]{Wu2017a}. 
	Corollaries \ref{main:1D} and \ref{cor:1D} are useful for estimating the wave speeds to ensure
	the PP property of the 
	HLL flux and local Lax-Friedrichs flux, respectively; see Theorem \ref{thm:HLLflux}. 
\end{remark}

\subsubsection{Proof of Theorem \ref{lem:main}}\label{sec:proofT}
We first establish several technical lemmas as the  
stepping stones on the path to prove Theorem \ref{lem:main}.

For any ${\bf U} \in {\mathcal G}$ and ${\bf v}^*,{\bf B}^* \in {\mathbb R}^3$, we 
define the nonzero vector $\bm \theta \in \mathbb{R}^7$ by 
\begin{equation*}
{\bm \theta} ({\bf U},{\bf v}^*,{\bf B}^*) 
:= \frac{1}{\sqrt 2} 
\Big( {\bf B}-{\bf B}^*,~\sqrt{\rho} ( {\bf v}-{\bf v}^* ),~\sqrt{2 \rho e}  \Big)^\top.  
\end{equation*} 
As a novel point, introducing such a vector will bring 
much convenience 
in the following estimates and analyses. It is easy to verify that 
\begin{equation}\label{eq:IDT1}
{\bf U}\cdot {\bf n}^* + \frac{|{\bf B}^*|^2}{2} = | {\bm \theta} |^2.
\end{equation}

\begin{lemma}\label{lem:densityflux}
	The set 
	\begin{equation*}
	{\mathcal G}_\rho := \big \{ {\bf U} = (\rho,{\bf m},{\bf B},E)^\top : \rho >0 \big \},
	\end{equation*}
	is a convex set. And for any $ {\bf U} \in {\mathcal G}_\rho$,  ${\bm \xi}  \in \mathbb{R}^d $ and  
	$\alpha> \langle {\bm \xi}, {\bf v} \rangle $, 
	it holds
	\begin{equation*}
	\alpha {\bf U}  - 
	\langle {\bm \xi}, {\bf F} ({\bf U} ) \rangle
	\in {\mathcal G}_\rho.
	\end{equation*}
\end{lemma}
\begin{proof}
	The result can be easily verified.
\end{proof}

\begin{lemma}\label{lem:IDs}
	For any $ {\bf U} \in {\mathcal G}$, any $ {\bf v}^*,{\bf B}^* \in {\mathbb R}^3$ and all $ i\in \{1,2,3\}$,  
	we have  
	\begin{equation} \label{eq:wklieq1}
	{\bf F}_i({\bf U}) \cdot {\bf n}^* - B_i ( {\bf v}^* \cdot {\bf B}^* )
	\le v_i \sum_{k=4}^7 \theta_k^2 +  v_i^* \bigg( \frac{ 1 }{2} |{\bf B}|^2 -  {\bf B} \cdot {\bf B}^* \bigg) + {\mathscr C}_i |{\bm \theta}|^2,
	\end{equation}
	where ${\mathscr C}_i :={\mathscr C}({\bf U};{\bf e}_i)$, and the vector 
	${\bf e}_i$ is the $i$-th row of the unit matrix of size 3. 
\end{lemma}

\begin{proof}	
	For any $i\in \{1,2,3\}$, we observe that 	
	\begin{equation}\label{eq:keyEQ}
	{\bf F}_i({\bf U}) \cdot {\bf n}^* - B_i ( {\bf v}^* \cdot {\bf B}^* )
	= v_i \sum_{k=4}^7 \theta_k^2 +  v_i^* \bigg( \frac{ 1 }{2} |{\bf B}|^2 -  {\bf B} \cdot {\bf B}^* \bigg) + 
	\Phi_i, 
	\end{equation}
	where 
	\begin{equation*}
	\Phi_i( {\bf U},{\bf v}^*,{\bf B}^* )  :=  p ( v_i-v_i^* ) + 
	\sum_{ \substack{1\le k \le 3 \\ k \neq i  } }
	\Big( B_k(v_i-v_i^*) - B_i (v_k-v_k^*)  \Big) (B_k - B_k^*).
	\end{equation*}
	Let us show that $\Phi_i$ is bounded by ${\mathscr C}_i |{\bm \theta}|^2$ from above.  
	We further observe that $\Phi_i$ is a quadratic form in the variables $\theta_k$, $1\le k \le 7$, and moreover, 
	the coefficients of the quadratic form do not depend on ${\bf v}^*$ and ${\bf B}^*$. Specifically, for the fixed $i$, we have 
	\begin{align*}
	& p ( v_i-v_i^* ) = 2 {\mathscr{C}}_s  \frac{\sqrt{\rho}}{\sqrt{2}} ( v_i - v_i^* )   \sqrt{\rho e} = 2  {\mathscr{C}}_s \theta_{3+i} \theta_7,
	\\
	& \Big( B_k(v_i-v_i^*) - B_i (v_k-v_k^*)  \Big) (B_k - B_k^*) 
	= 2\frac{B_k}{\sqrt{\rho}} \theta_{3+i} \theta_k 
	- 2\frac{B_i}{\sqrt{\rho}} \theta_{3+k} \theta_k,~~\forall k \neq i. 
	\end{align*} 
	Define $i_1:=i~{\rm mod}~3+1$ and $i_2:=(i+1)~{\rm mod}~3+1$, and  
	\begin{equation*}
	\tilde{\bm \theta}:= 
	\big(\theta_{3+i},~\theta_{3+i_1},~\theta_{3+i_2},~\theta_{i_1},~\theta_{i_2},~\theta_7 \big)^\top,
	\end{equation*}
	then 
	\begin{align*}
	\Phi_i = 2  {\mathscr{C}}_s \theta_{3+i} \theta_7 
	+ 2 \sum_{k \in\{i_1,i_2\}} \bigg(\frac{B_k}{\sqrt{\rho}} \theta_{3+i} \theta_k 
	- \frac{B_i}{\sqrt{\rho}} \theta_{3+k} \theta_k \bigg)
	= \tilde{\bm \theta}^\top {\bf A} \tilde{\bm \theta},
	\end{align*}
	where
	\begin{equation*}
	{\bf A}=
	\begin{pmatrix}
	0 & 0 & 0 & B_{i_1} \rho^{-\frac12} & B_{i_2} \rho^{-\frac12} & {\mathscr{C}}_s  \\
	0 & 0 & 0 & -B_i \rho^{-\frac12} & 0 & 0 \\
	0 & 0 & 0 & 0 & -B_i \rho^{-\frac12} & 0 \\
	B_{i_1} \rho^{-\frac12} & -B_i \rho^{-\frac12} & 0 & 0 & 0 & 0 \\
	B_{i_2} \rho^{-\frac12} & 0 & -B_i \rho^{-\frac12} & 0 & 0 & 0 \\
	{\mathscr{C}}_s  & 0 & 0 & 0 & 0 & 0
	\end{pmatrix}.
	\end{equation*}
	The spectral radius 
	of ${\bf A}$ is $ {\mathscr{C}}_i$. Therefore,
	\begin{equation*}
	|\Phi_i| \le | \tilde{\bm \theta}^\top {\bf A} \tilde{\bm \theta} |
	\le {\mathscr{C}}_i | \tilde{\bm \theta}|^2 
	=  {\mathscr C}_i  (|{\bm \theta}|^2-\theta_i^2) \le {\mathscr C}_i |{\bm \theta}|^2,
	\end{equation*}
	which along with the identity \eqref{eq:keyEQ} imply \eqref{eq:wklieq1}.  	
\end{proof}

For any unit vector ${\bm \xi}\in {\mathbb{R}}^{d}$, we introduce a matrix ${\bf T}_{\bm \xi}:={\rm diag} \big\{1,\widehat{\bf T}_{{\bm \xi}},\widehat{\bf T}_{{\bm \xi}},1 \big\}$,
with the rotational matrix $\widehat{\bf T}_{{\bm \xi}} $  defined as follows:
\\
(i). In $d=1$, ${\bm \xi}=\xi$ is a scalar of value 1 or $-1$, and $\widehat{\bf T}_{{\bm \xi}} $ is defined as ${\rm diag}\{\xi,1,1\}$.
\\
(ii). In $d=2$, let $( \cos \varphi, \sin \varphi )$ be the polar coordinate representation of ${\bm \xi} $, and 
\begin{equation*}
\widehat{\bf T}_{{\bm \xi}} :=  \begin{pmatrix}
\cos \varphi~ & ~\sin \varphi~ & ~0  \\
-\sin\varphi~ & ~\cos \varphi~ & ~0 \\
0~ & ~0~ & ~1
\end{pmatrix}.
\end{equation*}
(iii). In $d=3$, let $( \sin \phi \cos \varphi, \sin \phi \sin \varphi, \cos \phi  )$  be the spherical coordinate representation of ${\bm \xi}$, 
and 
\begin{equation*}
\widehat{\bf T}_{{\bm \xi}} :=  \begin{pmatrix}
\sin \phi \cos \varphi~ & ~\sin \phi \sin \varphi~ & ~\cos  \phi   \\
-\sin\varphi~ & ~\cos \varphi~ & ~0 \\
-\cos \phi \cos \varphi~ & ~-\cos \phi \sin \varphi~ & ~\sin  \phi
\end{pmatrix}.
\end{equation*}
The rotational invariance property of the $d$-dimensional MHD system \eqref{eq:MHD} implies 
\begin{equation}\label{eq:MHD:rotINV}
\langle {\bm \xi}, {\bf F} ({\bf U} ) \rangle  =  {\bf T}_{\bm \xi}^{-1} {\bf F}_1( {\bf T}_{\bm \xi} {\bf U}).
\end{equation}
This helps us extend Lemma \ref{lem:IDs} to the following general case. 



\begin{lemma}\label{lem:IEQ_WKL}
	For any ${\bf U} \in {\mathcal G}$, any ${\bf v}^*,{\bf B}^* \in {\mathbb R}^3$ 
	and any unit vector ${\bm \xi}  \in \mathbb{R}^d $, 
	it holds
	\begin{equation*}
	\langle {\bm \xi}, {\bf F} ( {\bf U} )  \rangle \cdot {\bf n}^*
	- \langle {\bm \xi},{\bf B} \rangle ({\bf v}^* \cdot {\bf B}^*) 
	\le \langle {\bm \xi}, {\bf v} \rangle \sum_{k=4}^7 \theta_k^2 + 
	\langle {\bm \xi}, {\bf v}^* \rangle \Big( \frac{1}{2}  |{\bf B}|^2 -  {\bf B} \cdot {\bf B}^* \Big) + {\mathscr C}({\bf U};{\bm \xi}) |{\bm \theta}|^2.
	\end{equation*}
\end{lemma}

\begin{proof}
	Let $\hat{\bf U}:={\bf T}_{\bm \xi} {\bf U}$, $\hat {\bf v}^* := \widehat{\bf T}_{\bm \xi} {\bf v}^*$, $\hat{\bf B}:=\widehat{\bf T}_{\bm \xi}{\bf B}^*$, $\hat{\bm \theta}:={\bm \theta} ( \hat{\bf U}, \hat {\bf v}^*, \hat{\bf B}^* )$, and 
	\begin{equation*}
	\hat{\bf n}^* := \bigg( \frac{|\hat{\bf v}^*|^2}2,~- \hat{\bf v}^*,~-\hat{\bf B}^*,~1 \bigg)^\top = {\bf T}_{\bm \xi} {\bf n}^*.
	\end{equation*}
	By the definition \eqref{eq:DefG}, one can easily verify  
	$\hat {\bf U} \in {\mathcal G}$, which, together with 
	the orthogonality of  
	${\bf T}_{\bm \xi}^{-1}$ and $\widehat{\bf T}_{\bm \xi}^{-1}$, 
	imply 
	\begin{align*}
	&\langle {\bm \xi}, {\bf F} ( {\bf U} )  \rangle \cdot {\bf n}^*
	- \langle {\bm \xi},{\bf B} \rangle ({\bf v}^* \cdot {\bf B}^*) 
	\\
	&  \overset{\eqref{eq:MHD:rotINV}}{=} 
	\big( {\bf T}_{\bm \xi}^{-1}  {\bf F}_1 ( \hat {\bf U} ) \big) \cdot \big( {\bf T}_{\bm \xi}^{-1} \hat{\bf n}^* \big)
	- \hat B_1 ( \widehat{\bf T}_{\bm \xi}^{-1} \hat{\bf v}^* ) \cdot ( \widehat{\bf T}_{\bm \xi}^{-1} \hat{\bf B}^*) 
	\\
	&  =
	{\bf F}_1 ( \hat {\bf U} )  \cdot  \hat{\bf n}^* 
	- \hat B_1 (  \hat {\bf v}^* \cdot  \hat {\bf B}^*) 
	\\
	&  \overset{\eqref{eq:wklieq1}}{\le}   
	\hat v_1 \sum_{k=4}^7 \hat \theta_k^2 +  \hat v_1^* \bigg( \frac{ 1 }{2} |\hat {\bf B}|^2 -  \hat {\bf B} \cdot \hat{\bf B}^* \bigg) 
	+ {\mathscr{C}}_1 (\hat{\bf U}) |\hat {\bm \theta}|^2
	\\
	&   = 
	\langle {\bm \xi}, {\bf v} \rangle \sum_{k=4}^7 \theta_k^2 + 
	\langle {\bm \xi}, {\bf v}^* \rangle \Big( \frac{1}{2}  |{\bf B}|^2 -  {\bf B} \cdot {\bf B}^* \Big) + {\mathscr C} ({\bf U};{\bm \xi}) |{\bm \theta}|^2.
	\end{align*}
	The proof is completed. 
\end{proof}

\begin{lemma}\label{lem:IEQparing}
	Assume that ${\bf U}=(\rho,\rho{\bf v},{\bf B},E)^\top\in {\mathcal G}$, 
	$\tilde{\bf U}=(\tilde \rho, \tilde \rho \tilde{\bf v}, \tilde {\bf B}, \tilde E)^\top \in {\mathcal G}$. For $\forall {\bf v}^*,{\bf B}^* \in {\mathbb R}^3$,  
	$\forall {\bm \xi}  \in \mathbb{R}^d $ and $\forall \delta \in \mathbb{R}$, 
	it holds
	\begin{equation}\label{eq:IEQparing}
	\begin{aligned}
	& \langle {\bm \xi}, {\bf v}^* \rangle
	\left[
	\bigg( \frac{\big|{\bf B}\big|^2}{2} -  {\bf B} \cdot {\bf B}^*  \bigg) 
	- \bigg(  \frac{\big|\tilde{\bf B}\big|^2}{2} -  \tilde{\bf B} \cdot {\bf B}^*
	\bigg) \right] 
	\\
	& \qquad
	\le \langle {\bm \xi}, \delta {\bf v} + (1-\delta) \tilde {\bf v} \rangle 
	\sum_{k=1}^3  
	\big( \theta_k^2
	- 
	\tilde \theta_k^2 \big) 
	+ |{\bm \xi}|  f( {\bf U}, \tilde {\bf U}; \delta ) 
	\big( |{{\bm \theta}}|^2 + 
	|{\tilde{\bm \theta}}|^2 \big),
	\end{aligned}
	\end{equation}
	where ${\bm \theta}:={\bm \theta} ( {\bf U},  {\bf v}^*,{\bf B}^* )$ and ${\tilde{\bm \theta}}:={\bm \theta} ( {{\tilde{\bf U}}},  {\bf v}^*, {\bf B}^* )$, and $f( {\bf U}, \tilde {\bf U}; \delta)$ is defined by
	\begin{equation}\label{eq:Deff}
	f( {\bf U}, \tilde {\bf U}; \delta) := \frac{ |\tilde{\bf B}-{\bf B}| }{\sqrt{2}} \sqrt{  \frac{\delta^2}{\rho} + \frac{ (1-\delta)^2 }{\tilde \rho}  }.
	\end{equation}
\end{lemma}

\begin{proof}
	With the aid of the Cauchy-Schwarz inequality, we have  
	\begin{align*}
	&\langle {\bm \xi}, {\bf v}^* \rangle
	\left[
	\bigg( \frac{\big|{\bf B}\big|^2}{2} -  {\bf B} \cdot {\bf B}^*  \bigg) 
	- \bigg(  \frac{\big|\tilde{\bf B}\big|^2}{2} -  \tilde{\bf B} \cdot {\bf B}^*
	\bigg) \right] - \langle {\bm \xi}, \delta {\bf v} + (1-\delta) \tilde {\bf v} \rangle 
	\sum_{k=1}^3  
	\big( \theta_k^2- \tilde \theta_k^2 \big)
	\\
	& 
	=\left( \frac{\delta}{2} \langle {\bm \xi}, {\bf v}-{\bf v}^* \rangle
	+ \frac{1-\delta}{2} \langle {\bm \xi}, \tilde {\bf v}-{\bf v}^* \rangle 
	\right) ( {\bf \tilde B}-{\bf B} ) \cdot 
	\big( {\bf B} + \tilde{\bf B} -2{\bf B}^* \big)
	\\[2mm]
	& \le \frac{|{\bm \xi}|}{2} \left( \frac{|\delta|}{\sqrt{\rho}} \sqrt{\rho} | {\bf v}-{\bf v}^*|
	+ \frac{|1-\delta|}{\sqrt{\tilde\rho}} \sqrt{\tilde \rho} | \tilde {\bf v}-{\bf v}^*|
	\right) | {\bf \tilde B}-{\bf B} | 
	\big( | {\bf B} - {\bf B}^* |+ |\tilde{\bf B} -{\bf B}^*| \big)
	\\[2mm]
	& \le 
	\frac{|{\bm \xi}|}{2}\sqrt{ \frac{\delta^2}{\rho} + \frac{ (1-\delta)^2 }{\tilde \rho}  }
	\sqrt{ \rho | {\bf v}-{\bf v}^*|^2 + \tilde \rho | \tilde {\bf v}-{\bf v}^*|^2 } 
	| {\bf \tilde B}-{\bf B} |  \sqrt{2 ( | {\bf B} - {\bf B}^* |^2+ |\tilde{\bf B} -{\bf B}^*|^2 ) }
	\\
	& 
	= 2 |{\bm \xi}| f( {\bf U}, \tilde {\bf U}; \delta)  \sqrt{ \sum_{k=4}^6 \big( \theta_k^2 + \tilde \theta_k^2 \big) } \sqrt{ \sum_{k=1}^3 \big( \theta_k^2 + \tilde \theta_k^2 \big) }
	\\
	& \le |{\bm \xi}| f( {\bf U}, \tilde {\bf U}; \delta) \sum_{k=1}^6 \big( \theta_k^2 + \tilde \theta_k^2 \big) \le |{\bm \xi}|  f( {\bf U}, \tilde {\bf U}; \delta ) 
	\big( |{{\bm \theta}}|^2 + 
	|{\tilde{\bm \theta}}|^2 \big).
	\end{align*}
	The proof is completed.
\end{proof}

We are now ready to prove Theorem \ref{lem:main}. 

\begin{proof}
	Note that $\alpha_j \ge \widehat \alpha_j > \langle {\bm \xi}^{(j)}, {\bf v}^{(j)} \rangle$. 
	It follows from Lemma \ref{lem:densityflux} 
	that 
	$ 
	\alpha_j {\bf U}^{(j)}
	- \langle {\bm \xi}^{(j)}, {\bf F} (  {\bf U}^{(j)} ) \rangle \in {\mathcal G}_\rho $, 
	and furthermore $\overline{\bf U} \in {\mathcal G}_\rho$, 
	by noting that $ \sum_{j=1}^Ns_j \alpha_j>0$ (see Remark \ref{rem:sumsposi}).

	We then focus on proving the inequality
	\eqref{eq:IEQ:multi-states22}, or equivalently,
	\begin{equation}\label{eq:equiProof}
	\sum_{j=1}^{N} s_j \Pi^{(j)} 
	\le \sum_{j=1}^{N}  \alpha_j |{\bm \theta}^{(j)}|^2,
	\end{equation}
	where ${\bm \theta}^{(j)}:={\bm \theta} ( {\bf U}^{(j)},  {\bf v}^*,{\bf B}^* )$, and 
	\begin{align*}
	&	\Pi^{(j)} := 	\langle {\bm \xi}^{(j)}, {\bf F} ( {\bf U}^{(j)} )  \rangle \cdot {\bf n}^*
	- \langle {\bm \xi}^{(j)},{\bf B}^{(j)} \rangle ({\bf v}^* \cdot {\bf B}^*).
	\end{align*}
	Using Lemma \ref{lem:IEQ_WKL} gives 
	\begin{align} \nonumber
	\sum_{j=1}^{N} s_j \Pi^{(j)} 
	& \le 
	\left\{	\sum_{j=1}^{N} s_j  
	\langle {\bm \xi}^{(j)}, {\bf v}^{(j)} \rangle \sum_{k=4}^7 \big|\theta_k^{(j)} \big|^2 \right\}
	+ \left\{	\sum_{j=1}^{N} s_j  {\mathscr C}({\bf U}^{(j)};{\bm \xi}) |{\bm \theta}^{(j)}|^2 \right\} 
	\\
	\nonumber
	& \quad + \left\{ 
	\sum_{j=1}^{N} s_j 
	\langle {\bm \xi}^{(j)}, {\bf v}^* \rangle \Big( \frac{1}{2}  |{\bf B}^{(j)}|^2 -  {\bf B}^{(j)} \cdot {\bf B}^* \Big) \right\} 
	\\ \label{eq:proofwkl11}
	& =: \Pi_1 + \Pi_2 + \Pi_3.
	\end{align} 
	Noting that, for any $1\le i \le N$, the hypothesis \eqref{eq:sumUpsilon1} implies 
	$$
	\sum_{j=1}^{N} s_j 
	\langle {\bm \xi}^{(j)}, {\bf v}^* \rangle  
	=  
	\left\langle \sum_{j=1}^{N} s_j  {\bm \xi}^{(j)}, {\bf v}^* \right \rangle = 0. 
	$$
	Thus we can reformulate $\Pi_3$ as 
	\begin{align*}
	\Pi_3 & = \sum_{j=1}^{N} s_j 
	\langle {\bm \xi}^{(j)}, {\bf v}^* \rangle \Big( \frac{1}{2}  |{\bf B}^{(j)}|^2 -  {\bf B}^{(j)} \cdot {\bf B}^* \Big) 
	- \sum_{j=1}^{N} s_j 
	\langle {\bm \xi}^{(j)}, {\bf v}^* \rangle \Big( \frac{1}{2}  |{\bf B}^{(i)}|^2 -  {\bf B}^{(i)} \cdot {\bf B}^* \Big) 
	\\
	&= \sum_{j=1}^{N} s_j 
	\langle {\bm \xi}^{(j)}, {\bf v}^*  \rangle
	\left[
	\Big( \frac{1}{2}  |{\bf B}^{(j)}|^2 -  {\bf B}^{(j)} \cdot {\bf B}^* \Big) - \Big( \frac{1}{2}  |{\bf B}^{(i)}|^2 -  {\bf B}^{(i)} \cdot {\bf B}^* \Big) 
	\right]=: \sum_{j=1}^{N} s_j \Pi_3^{(ji)},
	\end{align*} 
	for any $1\le i\le N$. 
	For any $\delta \in \mathbb R$, it follows from Lemma \ref{lem:IEQparing} that 
	\begin{equation}\label{eq:paringProof}
	\begin{split}
	\Pi_3^{(ji)} & \le \langle {\bm \xi}^{(j)}, \delta {\bf v}^{(j)} + (1-\delta)  {\bf v}^{(i)} \rangle 
	\sum_{k=1}^3  
	\left( |\theta_k^{(j)}|^2
	-  |\theta_k^{(i)}|^2 \right) 
	\\
	&\quad +   f( {\bf U}^{(j)}, {\bf U}^{(i)}; \delta ) 
	\big( |{\bm \theta}^{(j)}|^2 + 
	|{\bm \theta}^{(i)}|^2 \big).
	\end{split}
	\end{equation}
	In particular, we take the free variable $\delta$ as  ${\sqrt{\rho^{(j)}}}/\big( \sqrt{\rho^{(j)}} + \sqrt{\rho^{(i)}} \big)$, which gives the Roe-type weighted average.  
	Let
	$$
	\bar{\bf v}^{(ji)} := \frac{ \sqrt{\rho^{(j)}} {\bf v }^{(j)} + 
		\sqrt{\rho^{(i)}} {\bf v }^{(i)} } 
	{ \sqrt{\rho^{(j)}}  +	\sqrt{\rho^{(i)}}  },
	$$
	then the inequality \eqref{eq:paringProof} becomes 
	\begin{equation}\label{eq:paringProof2}
	\Pi_3^{(ji)}  \le \left\langle {\bm \xi}^{(j)}, \bar{\bf v}^{(ji)}  \right \rangle 
	\sum_{k=1}^3  
	\left( |\theta_k^{(j)}|^2
	-  |\theta_k^{(i)}|^2 \right) 
	+  \frac{ |{\bf B}^{(j)} - {\bf B}^{(i)} | }{ \sqrt{\rho^{(j)}}  +	\sqrt{\rho^{(i)}}  }
	\big( |{\bm \theta}^{(j)}|^2 + 
	|{\bm \theta}^{(i)}|^2 \big).
	\end{equation}
	It follows that 
	\begin{equation}\label{eq:proofabc}
	\begin{aligned}
	\left(\sum_{i=1}^N s_i \right) \Pi_3 
	& = \sum_{i=1}^N \sum_{j=1}^{N} s_i s_j \Pi_3^{(ji)}
	\\
	& \le \sum_{i=1}^N \sum_{j=1}^{N} s_i s_j  \left\langle {\bm \xi}^{(j)},  \bar{\bf v}^{(ji)}  \right \rangle 
	\sum_{k=1}^3  
	\left( |\theta_k^{(j)}|^2
	-  |\theta_k^{(i)}|^2 \right) 
	\\
	& \quad 
	+ \sum_{i=1}^N \sum_{j=1}^{N} s_i s_j  \frac{ |{\bf B}^{(j)} - {\bf B}^{(i)} | }{ \sqrt{\rho^{(j)}}  +	\sqrt{\rho^{(i)}}  }
	\big( |{\bm \theta}^{(j)}|^2 + 
	|{\bm \theta}^{(i)}|^2 \big).
	\end{aligned} 
	\end{equation}
	By $\bar{\bf v}^{(ji)}=\bar{\bf v}^{(ij)}$ and the technique of exchanging indexes $i$ and $j$, we obtain 
	\begin{align*}
	&	\sum_{i=1}^N \sum_{j=1}^{N} s_i s_j  \left\langle {\bm \xi}^{(j)},  \bar{\bf v}^{(ji)}  \right \rangle 
	\sum_{k=1}^3  
	|\theta_k^{(i)}|^2 = \sum_{i=1}^N \sum_{j=1}^{N} s_i s_j  \left\langle {\bm \xi}^{(i)},  \bar{\bf v}^{(ji)} \right \rangle 
	\sum_{k=1}^3  
	|\theta_k^{(j)}|^2,
	\\
	& \sum_{i=1}^N \sum_{j=1}^{N} s_i s_j  \frac{ |{\bf B}^{(j)} - {\bf B}^{(i)} | }{ \sqrt{\rho^{(j)}}  +	\sqrt{\rho^{(i)}}  }
	|{\bm \theta}^{(i)}|^2 
	=  \sum_{i=1}^N \sum_{j=1}^{N} s_i s_j  \frac{ |{\bf B}^{(j)} - {\bf B}^{(i)} | }{ \sqrt{\rho^{(j)}}  +	\sqrt{\rho^{(i)}}  }
	|{\bm \theta}^{(j)}|^2.
	\end{align*}
	Therefore, the inequality \eqref{eq:proofabc} can be rewritten as 
	\begin{align*}
	\left(\sum_{i=1}^N s_i \right) \Pi_3 
	&  \le 
	\sum_{i=1}^N \sum_{j=1}^{N} s_i s_j  \left\langle {\bm \xi}^{(j)} - {\bm \xi}^{(i)},  \bar{\bf v}^{(ji)}  \right \rangle 
	\sum_{k=1}^3  
	|\theta_k^{(j)}|^2 
	+ 2 \sum_{i=1}^N \sum_{j=1}^{N} s_i s_j  \frac{ |{\bf B}^{(j)} - {\bf B}^{(i)} | }{ \sqrt{\rho^{(j)}}  +	\sqrt{\rho^{(i)}}  }
	|{\bm \theta}^{(j)}|^2, 
	\end{align*}
	which further yields 
	\begin{equation}\label{keyaaaa}
	\begin{split}
	\Pi_3  & \le 
	\sum_{j=1}^{N} s_j 
	\left( 
	\frac{1}{ \sum_{i=1}^N s_i } \sum_{i=1}^Ns_i  \left\langle {\bm \xi}^{(j)} - {\bm \xi}^{(i)},  \bar{\bf v}^{(ji)}  \right \rangle 
	\right) \sum_{k=1}^3  
	|\theta_k^{(j)}|^2 
	\\
	& \quad + \sum_{j=1}^{N} s_j \left(  \frac{2}{ \sum_{i=1}^N s_i } \sum_{i=1}^N s_i   \frac{ |{\bf B}^{(j)} - {\bf B}^{(i)} | }{ \sqrt{\rho^{(j)}}  +	\sqrt{\rho^{(i)}}  }
	\right) |{\bm \theta}^{(j)}|^2.
	\end{split}
	\end{equation}
	Note that 
	\begin{align*}
	&	\Pi_1 + 
	\sum_{j=1}^{N} s_j 
	\left( 
	\frac{1}{ \sum_{i=1}^N s_i } \sum_{i=1}^Ns_i  \left\langle {\bm \xi}^{(j)} - {\bm \xi}^{(i)},  \bar{\bf v}^{(ji)}  \right \rangle 
	\right) \sum_{k=1}^3  
	|\theta_k^{(j)}|^2  
	\\
	& \quad \le 
	\sum_{j=1}^{N} s_j \max \left\{ \big\langle {\bm \xi}^{(j)}, {\bf v}^{(j)} \big\rangle , 
	\frac{1}{\sum_{i=1}^Ns_i}
	\sum_{i=1}^N s_i 
	\left \langle
	{\bm \xi}^{(j)} - {\bm \xi}^{(i)}, 
	\bar{\bf v}^{(ji)} \right \rangle
	\right\} \sum_{k=1}^7  
	|\theta_k^{(j)}|^2,
	\end{align*}
	which along with \eqref{eq:proofwkl11} and \eqref{keyaaaa} 
	imply 
	$$
	\sum_{j=1}^{N} s_j \Pi^{(j)} 
	\le \sum_{j=1}^{N} \widehat \alpha_j |{\bm \theta}^{(j)}|^2 
	\le \sum_{j=1}^{N} \alpha_j |{\bm \theta}^{(j)}|^2.
	$$
	Hence the inequality \eqref{eq:equiProof} holds.
	
	Under the condition \eqref{eq:DDFtheorem}, the inequality \eqref{eq:IEQ:multi-states22} becomes 
	$\overline{\bf U}\cdot {\bf n}^* + \frac{|{\bf B}^*|^2}{2}\ge 0,$ 
	$\forall {\bf v}^*,{\bf B}^*\in {\mathbb R}^3$, 
	which together with $\overline{\bf U} \in {\mathcal G}_\rho$ 
	imply $\overline{\bf U} \in \overline {\mathcal G}_*$. 
	The proof is completed.
\end{proof}

\subsection{Estimates relative to source term}

We also need the following lemma, which was proposed in \cite{WuShu2018}. 

\begin{lemma}
	For any ${\bf U} \in {\mathcal G}$ and any ${\bf v}^*,{\bf B}^* \in {\mathbb R}^3$, we have  
	\begin{align} \label{eq:indentityS}
	&{\bf S}({\bf U}) \cdot {\bf n}^* = ( {\bf v} - {\bf v}^* ) \cdot ( {\bf B} - {\bf B}^* ) - {\bf v}^* \cdot  {\bf B}^*, 
	\\ \label{eq:widelyusedIEQ}
	& |\sqrt{\rho}( {\bf v} - {\bf v}^* ) \cdot ( {\bf B} - {\bf B}^* )| 
	< {\bf U} \cdot {\bf n}^* + \frac{|{\bf B}^*|^2}{2}.
	\end{align}
	Furthermore, for any $b \in {\mathbb R}$, it holds  
	\begin{equation}\label{eq:widelyusedIEQ2}
	-b ( {\bf S}({\bf U}) \cdot {\bf n}^* ) \ge    b ( {\bf v}^* \cdot  {\bf B}^* ) 
	- \frac{|b|}{\sqrt{\rho}} \left( {\bf U} \cdot {\bf n}^* + \frac{|{\bf B}^*|^2}{2} \right).
	\end{equation}
\end{lemma}

\subsection{Properties of the HLL flux}

The Harten--Lax--van Leer (HLL) flux is derived from an 
approximate Riemann solver in the direction normal to each cell interface. Let ${\bm \xi} \in \mathbb R^d$ be the unit normal vector of the interface. 
Then the HLL flux at the interface is given by 
\begin{equation}\label{eq:1DHLL1}
\hat {\bf F} (  {\bf U}^-,  {\bf U}^+; {\bm \xi} )
= \begin{cases}
\langle {\bm \xi},{\bf F} ({\bf U}^-) \rangle,& 0\le \sigma_l < \sigma_r,\\[1mm]
\displaystyle
\frac{\sigma_r \langle {\bm \xi}, {\bf F}({\bf U}^-) \rangle - \sigma_l \langle {\bm \xi}, {\bf F}( {\bf U}^+ ) \rangle 
	+\sigma_l \sigma_r ({\bf U}^+ - {\bf U}^-) }{\sigma_r - \sigma_l},&\sigma_l < 0 < \sigma_r,\\[1mm]
\langle {\bm \xi}, {\bf F} ({\bf U}^+) \rangle,& \sigma_l < \sigma_r \le 0.
\end{cases}	
\end{equation}
Here $\sigma_l ({\bf U}^-,{\bf U}^+;{\bm \xi}) $ and $\sigma_r({\bf U}^-,{\bf U}^+;{\bm \xi})$ are functions of ${\bf U}^-$, ${\bf U}^+$ 
and ${\bm \xi}$,  
denoting the estimates of the leftmost and rightmost wave speeds 
in the (rotated) Riemann problem in the direction of ${\bm \xi}$,  where ${\bf U}^-$ and ${\bf U}^+$ are the left and right  initial states respectively. We require $\sigma_r > \sigma_l$, and 
\begin{equation}\label{eq:HLLconservative}
\sigma_r ({\bf U}^-,{\bf U}^+;{\bm \xi}) = - \sigma_l
({\bf U}^+,{\bf U}^-;-{\bm \xi}),
\end{equation}
which ensures that the numerical flux \eqref{eq:1DHLL1} is conservative, that is,  
$$
\hat {\bf F} (  {\bf U}^-,  {\bf U}^+; {\bm \xi} ) 
+ \hat {\bf F} (  {\bf U}^+,  {\bf U}^-; -{\bm \xi} )=0.
$$ 
Let 
$$
\sigma^+= \max\{ 
\sigma_r, 0
\}, \quad \sigma^-= \min\{ 
\sigma_l, 0
\},
$$
then the flux \eqref{eq:1DHLL1} can be reformulated as 
\begin{equation}\label{eq:1DHLL2}
\hat {\bf F} (  {\bf U}^-,  {\bf U}^+; {\bm \xi} )=
\frac{\sigma^+ \langle {\bm \xi}, {\bf F}({\bf U}^-) \rangle - \sigma^- \langle {\bm \xi}, {\bf F}( {\bf U}^+ ) \rangle 
	+\sigma^- \sigma^+ ({\bf U}^+ - {\bf U}^-) }{\sigma^+ - \sigma^-}.
\end{equation}
Note that the LF flux can be considered as a special HLL flux with $\sigma_r=-\sigma_l=\sigma$, where $\sigma$ is the maximum wave speed. 
Therefore, all the analysis in the present paper also applies to the local LF flux and global LF flux.


The following property is derived for the HLL flux \eqref{eq:1DHLL1} in the ideal MHD case.

\begin{theorem}\label{thm:HLLflux}
	Assume ${\bf U}^-,{\bf U}^+ \in {\mathcal G}$. 
	If the parameters (approximate wave speeds) in the HLL flux \eqref{eq:1DHLL1} satisfy 
	\begin{equation}\label{eq:HLLsigma}
	\sigma_{ r} \ge \alpha_{ r} 
	( {\bf U}^+, {\bf U}^-; {\bm \xi} ),
	\qquad \sigma_l \le  
	\alpha_{ l} 
	( {\bf U}^-, {\bf U}^+; {\bm \xi} ),
	\end{equation}
	then   
	\begin{align}\label{eq:FH1}
	&\hat {\bf F} (  {\bf U}^-,  {\bf U}^+; {\bm \xi} ) 
	= \sigma^- {\bf H} (  {\bf U}^-,  {\bf U}^+; {\bm \xi} )  
	+ \langle {\bm \xi}, {\bf F}({\bf U}^-) \rangle - \sigma^- 
	{\bf U}^-,\\ \label{eq:FH2}
	& 
	\hat {\bf F} (  {\bf U}^-,  {\bf U}^+; {\bm \xi} ) 
	= \sigma^+ {\bf H} (  {\bf U}^-,  {\bf U}^+; {\bm \xi} )  
	+ \langle {\bm \xi}, {\bf F}({\bf U}^+) \rangle - \sigma^+ 
	{\bf U}^+,
	\end{align}
	and the intermediate state 
	\begin{equation}\label{eq:DefH}
	{\bf H} (  {\bf U}^-,  {\bf U}^+; {\bm \xi} ) 
	:= \frac{ 1 }{ \sigma^+ - \sigma^- } 
	\Big(   \sigma^+ {\bf U}^+
	- \langle {\bm \xi}, {\bf F} (  {\bf U}^+ ) \rangle
	- \sigma^-  {\bf U}^-
	+  \langle {\bm \xi}, {\bf F} (  {\bf U}^- ) \rangle
	\Big)
	\end{equation}
	belongs to ${\mathcal G}_\rho$ and satisfies 
	\begin{equation}\label{eq:KeyIEQ1D4}
	{\bf H} \cdot {\bf n}^* 
	+ \frac{|{\bf B}^*|^2}{2} + \frac{{\bf v}^* \cdot {\bf B}^*}{ \sigma^+ - \sigma^-} \big( \langle {\bm \xi},  {\bf B}^+ \rangle - \langle {\bm \xi}, {\bf B}^- \rangle \big) \ge 0, \quad \forall {\bf v}^*,{\bf B}^* \in \mathbb R^3.
	\end{equation}
	Furthermore, if $\langle {\bm \xi},  {\bf B}^+ \rangle = \langle {\bm \xi}, {\bf B}^- \rangle$, then 	
	${\bf H} \in \overline {\mathcal G}_*$. 
\end{theorem}

\begin{proof}
	The identities \eqref{eq:FH1}--\eqref{eq:FH2} can be verified by using \eqref{eq:1DHLL2}. 
	Under the condition \eqref{eq:HLLsigma}, 
	we have 
	$$
	\sigma^+ \ge \sigma_r \ge \alpha_{ r} ( {\bf U}^+, {\bf U}^-; {\bm \xi} ), \quad  
	\sigma^- \le \sigma_l \le 
	\alpha_{ l} ( {\bf U}^-,{\bf U}^+; {\bm \xi} ).
	$$
	It follows from Corollary \ref{main:1D} 
	that ${\bf H} (  {\bf U}^-,  {\bf U}^+; {\bm \xi} ) \in {\mathcal G}_\rho$ and satisfies \eqref{eq:KeyIEQ1D4}. 
\end{proof}

\begin{remark} 
	It is observed from \eqref{eq:KeyIEQ1D4} that the admissibility of the intermediate state  ${\bf H}$ is 
	closely related to
	the jump in the normal magnetic field across the cell interface. 
	If the jump is zero, then ${\bf H} \in \overline {\mathcal G}_*$; otherwise, 
	${\bf H}$ does not always belong to $\overline {\mathcal G}_*$ even if many times larger wave speeds are employed. 
	 However, in the multidimensional cases, a standard finite volume or DG method 
	cannot avoid jumps in normal magnetic field at cell interfaces although such jumps do not exist in the exact solution. 
	This causes some challenges essentially different from 1D case. 
	We will demonstrate that this issue can be overcome by 
	coupling two divergence-controlling techniques: the locally divergence-free  element and properly discretized Godunov--Powell source term. 
	The former technique leads to zero divergence within each cell, while the latter 
	controls the divergence error across cell interfaces. 
\end{remark}

\begin{remark} 
	The proposed condition \eqref{eq:HLLsigma} 
	for the wave speeds $\sigma_l$ and $\sigma_r$ 
	is crucial for the provably PP property of 
	our schemes presented later. 	
	The condition \eqref{eq:HLLsigma} is acceptable, because 
	$\alpha_{ l}$ 
	and $\alpha_{ r} $ are respectively  close to the minimum and maximum signal speeds of the system \eqref{eq:MHD:GP} in the direction of ${\bm \xi}$.  
	Let $\sigma_l^{\rm std}$ and $\sigma_r^{\rm std}$ 
	denote a standard choice of wave speeds in the HLL flux, for example, 
	Davis \cite{davis1988simplified} gave those speeds as
\begin{equation}\label{Davis:HLL}
	\sigma_l^{\rm std} = \min\{ \lambda_1({\bf U}^-;{\bm \xi}), \lambda_1({\bf U}^+;{\bm \xi})  \},\quad  \sigma_r^{\rm std} = \max \{ \lambda_8({\bf U}^-;{\bm \xi}), \lambda_8({\bf U}^+;{\bm \xi})  \},
\end{equation}
	or Einfeldt et al.~\cite{einfeldt1991godunov} suggested to use 
	$$
	\sigma_l^{\rm std} = \min\{ \lambda_1({\bf U}^-;{\bm \xi}), \lambda_1({\bf U}^{\tt Roe};{\bm \xi})  \},\quad  \sigma_r^{\rm std} = \max \{ \lambda_8({\bf U}^+;{\bm \xi}), \lambda_8({\bf U}^{\tt Roe};{\bm \xi})  \},
	$$	
	where $\lambda_1({\bf U};{\bm \xi})$ amd $\lambda_8({\bf U};{\bm \xi})$ 
	are the minimum and maximum eigenvalues of the Jacobi matrix of the system \eqref{eq:MHD:GP} in the direction of ${\bm \xi}$, 
	and $\lambda_i({\bf U}^{\tt Roe};{\bm \xi})$ is the estimate of eigenvalues based on the Roe matrix (cf.~\cite{Powell1995}).	
	These choices may not necessarily give a PP flux in the MHD case and probably not satisfy \eqref{eq:HLLsigma}.  
	In practice, by considering the stability and the PP property, we suggest to use   
	\begin{equation}\label{eq:choiceHLL} 
	\sigma_l =  \min\{
	\alpha_{ l} 
	( {\bf U}^-, {\bf U}^+; {\bm \xi} ), \sigma_l^{\rm std} \},
	\quad 
	\sigma_{ r} = \max\{ \alpha_{ r} 
	( {\bf U}^+, {\bf U}^-; {\bm \xi} ), \sigma_r^{\rm std} \}
	\end{equation} 
	in the HLL flux, and use  
	\begin{equation*} 
	\sigma_{ r} =- \sigma_l =  \max \big\{
	\alpha_\star 
	( {\bf U}^-, {\bf U}^+; {\bm \xi} ), 
	\alpha_\star 
	( {\bf U}^+, {\bf U}^-; {\bm \xi} ), \sigma^{\rm std} \big \},
	\end{equation*} 
	in the local LF flux, where $\sigma^{\rm std}$ denotes a standard  numerical viscosity parameter for the local LF flux.
\end{remark}

\section{Positivity-preserving schemes in one dimension} \label{sec:1D}
In this section, we present provably PP 
finite volume and DG schemes 
with the proposed HLL flux for 1D MHD equations \eqref{eq:MHD}. 
Let $x$ denote the spatial variable. The 
condition \eqref{eq:2D:BxBy0} and the fifth equation of \eqref{eq:MHD} imply $B_1(x,t)\equiv {\rm constant}$ (denoted by ${\tt B}_{\tt const}$) for all $x$ and $t \ge 0$.

Let $I_j=[x_{j-\frac{1}{2}},x_{j+\frac{1}{2}}]$, 
$I=\cup_j I_j$ be a partition of the spatial domain. 
Denote $\Delta x_j = x_{j+\frac{1}{2}} - x_{j-\frac{1}{2}}$. Let $\{t_0=0, t_{n+1}=t_n+\Delta t_{n}, n\geq 0\}$ 
be a partition of the time interval $[0,T]$, where 
the time step-size $\Delta t_{n}$ is determined by some Courant–Friedrichs–Lewy (CFL) condition. Let $\bar {\bf U}_j^n $ denote the numerical approximation to the cell average of the exact solution ${\bf U}( x,t)$ over $I_j$ at $t=t_n$. 
We would like to seek PP schemes with  
$\bar {\bf U}_j^n $ always preserved in the admissible state set $\mathcal G$.

\subsection{First-order scheme}
We consider the 1D first-order scheme  
\begin{equation}\label{eq:1DMHD:HLL}
\bar {\bf U}_j^{n+1} = \bar {\bf U}_j^{n} - \frac{\Delta t_n}{\Delta x_j}
\Big(  \hat {\bf F}_1 ( \bar {\bf U}_j^{n}, \bar {\bf U}_{j+1}^{n} )
- \hat {\bf F}_1 ( \bar {\bf U}_{j-1}^{n}, \bar {\bf U}_j^{n})
\Big),
\end{equation}
where $\hat {\bf F}_1 (  \bar{\bf U}_j^n,  \bar{\bf U}_{j+1}^n )
:= \hat {\bf F} (  \bar{\bf U}_j^n,  \bar{\bf U}_{j+1}^n ;1 )$ 
is taken as the HLL flux in \eqref{eq:1DHLL2}.

\begin{theorem}
	Assume that $\bar {\bf U}_j^0 \in{\mathcal G}$ and $\bar B_{1,j}^0 = {\tt B}_{\tt const}$ for all $j$, and
	the wave speeds in the HLL flux satisfy \eqref{eq:HLLsigma}. 
	Then the solution $\bar {\bf U}_j^n$, computed by the scheme \eqref{eq:1DMHD:HLL}  under the CFL condition
	\begin{equation}\label{eq:1DCFL:HLL}
	\left(  \sigma_{j-\frac12}^{n,+} - \sigma_{j+\frac12}^{n,-} \right) 
	\frac{\Delta t_n}{ \Delta x_j} < 1,\quad \forall j,
	\end{equation}
	belongs to ${\mathcal G}$ and satisfies $\bar B_{1,j}^n = {\tt B}_{\tt const} $ for all $j$
	and $n\in {\mathbb{N}}$, 
	where 
	$$\sigma_{j-\frac12}^{n,+} := \sigma^+( \bar{\bf U}_{j-1}^n, \bar{\bf U}_{j}^n;1 ), \qquad \sigma_{j+\frac12}^{n,-} := \sigma^-( \bar{\bf U}_{j}^n, 
	\bar{\bf U}_{j+1}^n;1).$$
\end{theorem}

\begin{proof}
	Here we use the induction argument for the time level number $n$.
	The conclusion obviously holds for $n=0$ because of the hypothesis on the initial data.
	Let us assume that $\bar {\bf U}_j^n\in {\mathcal G}$ with $\bar B_{1,j}^n = {\tt B}_{\tt const} $ for all $j$,
	and verify the conclusion holds for $n+1$.  	
	Let $\lambda := \Delta t_n /\Delta x_j$, and ${\bf H}_{j+\frac12}^n := {\bf H} ( \bar{ \bf  U }_j^n, \bar{ \bf  U }_{j+1}^n;1 )$; see \eqref{eq:DefH} for the definition of ${\bf H}$. 
	Under the induction hypothesis, we have that 
	${\bf H}_{j+\frac12}^n \in \overline{\mathcal G}_*, \forall j  $ according to 
	Theorem \ref{thm:HLLflux}, and 
	the fifth component of ${\bf H}_{j+\frac12}^n $ is ${\tt B}_{\tt const}$ for all $j$ by noting that the fifth component of ${\bf F}_1$ is zero. 
	Using the identities \eqref{eq:FH1} and \eqref{eq:FH2}, one can rewrite the scheme \eqref{eq:1DMHD:HLL} as 
	\begin{align} \label{eq:HLL1Deq}
	\begin{split}
	\bar{\bf U}_j^{n+1} 
	&=\bar{\bf U}_j^n - \lambda 
	\Big[
	\left(\sigma_{ j+\frac12 }^{n,-} 
	{\bf H}_{j+\frac12}^n + {\bf F}_1 ( \bar{\bf U}_j^n ) - \sigma_{j+\frac12}^{n,-} \bar{\bf U}_{j}^n \right)
	\\  
	& \qquad \qquad 
	- 	\left(\sigma_{ j-\frac12 }^{n,+} 
	{\bf H}_{j-\frac12}^n + {\bf F}_1 ( \bar{\bf U}_j^n ) - \sigma_{j-\frac12}^{n,+} \bar{\bf U}_{j}^n \right)
	\Big]
	\\ 
	& = \Big( 1+ \lambda ( \sigma_{j+\frac12}^{n,-} 
	- \sigma_{j-\frac12}^{n,+} ) \Big) \bar{\bf U}_j^n 
	+  
	\Big(
	- \lambda \sigma_{j+\frac12}^{n,-} \Big)  {\bf H}_{j+\frac12}^n 
	+ \lambda \sigma_{ j-\frac12 }^{n,+} 
	{\bf H}_{j-\frac12}^n.
	\end{split}
	\end{align}
	Under the condition \eqref{eq:1DCFL:HLL}, $\bar{\bf U}_j^{n+1}$ 
	is a convex combination of $\bar{\bf U}_j^n $, 
	${\bf H}_{j+\frac12}^n$ and ${\bf H}_{j-\frac12}^n$. 
	Hence we have $\bar{\bf U}_j^{n+1} \in {\mathcal G}$ by Lemma \ref{theo:MHD:convex}. 
	The fifth equation of \eqref{eq:HLL1Deq} also implies  
	\begin{equation*}
	\bar B_{1,j}^{n+1} = 
	\Big( 1+ \lambda ( \sigma_{j+\frac12}^{n,-} 
	- \sigma_{j-\frac12}^{n,+} ) \Big) {\tt B}_{\tt const}  
	- \lambda\sigma_{j+\frac12}^{n,-} {\tt B}_{\tt const} 
	+ \lambda \sigma_{ j-\frac12 }^{n,+} 
	{\tt B}_{\tt const} 
	= {\tt B}_{\tt const}. 
	\end{equation*}
	Therefore, the conclusion holds for $n+1$. The proof is completed.
\end{proof}

\subsection{High-order schemes}

For convenience,  we first focus on the forward Euler method for time discretization and will discuss the 
high-order time discretization later. 
We consider the high-order finite volume schemes as well as the scheme satisfied by
the cell-averaged solution of a standard DG method for \eqref{eq:MHD}, 
which have the following form
\begin{equation}\label{eq:1DMHD:Hcellaverage}
\bar {\bf U}_j^{n+1} = \bar {\bf U}_j^{n} - \frac{\Delta t_n}{\Delta x_j}
\Big(  \hat {\bf F}_1 ( {\bf U}_{j+ \frac{1}{2}}^-, {\bf U}_{j+ \frac{1}{2}}^+ )
- \hat {\bf F}_1 ( {\bf U}_{j- \frac{1}{2}}^-, {\bf U}_{j-\frac{1}{2}}^+)
\Big) ,
\end{equation}
where $\hat {\bf F}_1 ( {\bf U}_{j+ \frac{1}{2}}^- ,  {\bf U}_{j+ \frac{1}{2}}^+ )
:= \hat {\bf F} (  {\bf U}_{j+ \frac{1}{2}}^- ,  {\bf U}_{j+ \frac{1}{2}}^+ ;1 )$ 
is taken as the HLL flux in \eqref{eq:1DHLL2}. 
The quantities ${\bf U}_{j + \frac{1}{2}}^-$ and ${\bf U}_{j + \frac{1}{2}}^+$ denote the high-order accurate approximations  
of the point values ${\bf U}( x_{j + \frac{1}{2}} ,t_n )$ within the cells $I_j$ and $I_{j+1}$, respectively, 
computed by
\begin{equation}\label{eq:DG1Dvalues}
{\bf U}_{j + \frac{1}{2}}^- = {\bf U}_j^n \big( x_{j + \frac{1}{2}}-0 \big), \quad {\bf U}_{j + \frac{1}{2}}^+ = {\bf U}_{j+1}^n \big( x_{j + \frac{1}{2}}+0 \big).
\end{equation}
Here the function ${\bf U}_j^n(x)$ is a polynomial vector of degree $k$ 
with the cell-averaged value of $\bar {\bf U}_j^n$ over the cell $I_j$. It 
approximates ${\bf U}( x,t_n)$ within $I_j$, and is either reconstructed in the finite volume schemes from $\{\bar {\bf U}_j^n\}$ or directly evolved in the DG schemes. 
The discrete evolution equations for the high-order ``moments'' of ${\bf U}_j^n(x)$ in the DG schemes are omitted. 

If $k=0$, i.e., ${\bf U}_j^n(x)=\bar{\bf U}_j^{n}$, $\forall x \in I_j$, then the scheme \eqref{eq:1DMHD:Hcellaverage} reduces to the 
first-order scheme \eqref{eq:1DMHD:HLL}, which has been proven to be PP under the CFL condition \eqref{eq:1DCFL:HLL}.

When $k\ge 1$, 
the solution $\bar{\bf U}_j^{n+1}$ of 
the high-order scheme \eqref{eq:1DMHD:Hcellaverage} 
does not always belong to $\mathcal G$ even if $\bar {\bf U}_j^{n} \in {\mathcal G}$ for all $j$. 
In the following theorem, we give a satisfiable condition for achieving the provably PP property of the scheme \eqref{eq:1DMHD:Hcellaverage} when $k\ge 1$. 

Let $\{ \widehat x_j^{(\mu)} \}_{\mu=1}^{ {\tt L}}$ be the {\tt L}-point Gauss--Lobatto quadrature points in the interval $I_j$. The associated weights are denoted by $\{\widehat \omega_\mu\}_{\mu=1} ^{\tt L}$ with $\sum_{\mu=1}^{\tt L} \widehat\omega_\mu = 1$.  
Following \cite{zhang2010,zhang2010b}, we 
take ${\tt L}=\lceil \frac{k+3}{2} \rceil$.

\begin{theorem} \label{thm:PP:1DMHD}
	Let the wave speeds in the HLL flux satisfy \eqref{eq:HLLsigma}. 
	If the polynomial vectors $\{{\bf U}^n_j(x)\}$ satisfy 
	\begin{align}\label{eq:1DDG:con1}
	&
	B_{1,j+\frac12}^{\pm} = {\tt B}_{\tt const}  ,\quad \forall j,  
	\\ \label{eq:1DDG:con2}
	& {\bf U}_j^n ( \widehat x_j^{(\mu)} ) \in {\mathcal G}, \quad \forall \mu \in \{1,2,\cdots,{\tt L}\}, ~\forall j,
	\end{align}
	then the high-order scheme \eqref{eq:1DMHD:Hcellaverage} is PP under the CFL condition
	\begin{equation}\label{eq:CFL:1DMHD}
	\frac{ \Delta t_n }{ \Delta x_j} \max\left\{
	\alpha_j^\star 
	+ \sigma_{j-\frac12}^{n,+}, 
	\alpha_j^\star 
	- \sigma_{j+\frac12}^{n,-}
	\right\} \le \widehat \omega_1,\quad \forall j,
	\end{equation}
	where $\sigma_{j+\frac12}^{n,\pm} := \sigma^{\pm}( {\bf U}_{j+\frac12}^-,{\bf U}_{j+\frac12}^+; 1 )$, 
	and 
	$$\alpha_j^\star:=\max\left\{ \alpha_\star ( {\bf U}_{j-\frac12}^+,{\bf U}_{j+\frac12}^-;1 ) , \alpha_\star ( {\bf U}_{j+\frac12}^-, {\bf U}_{j-\frac12}^+ ;1 )   \right\}.$$ 
\end{theorem}

\begin{proof}
	Using \eqref{eq:FH1}--\eqref{eq:FH2}, we can reformulate the numerical fluxes in \eqref{eq:1DMHD:Hcellaverage} as 
	\begin{align}\label{1DHLL1}
	\hat {\bf F}_1 ( {\bf U}_{j+ \frac{1}{2}}^- ,  {\bf U}_{j+ \frac{1}{2}}^+ ) 
	= \sigma_{j+\frac12}^{n,-} {\bf H}_{j+\frac12} 
	+ {\bf F}_1( {\bf U}_{j+\frac12}^- ) - \sigma_{j+\frac12}^{n,-} {\bf U}_{j+\frac12}^-,
	\\ \label{1DHLL2}
	\hat {\bf F}_1 ( {\bf U}_{j- \frac{1}{2}}^- ,  {\bf U}_{j- \frac{1}{2}}^+ ) 
	= \sigma_{j-\frac12}^{n,+} {\bf H}_{j-\frac12} 
	+ {\bf F}_1( {\bf U}_{j-\frac12}^+ ) - \sigma_{j-\frac12}^{n,+} {\bf U}_{j-\frac12}^+,	
	\end{align}
	where ${\bf H}_{j+\frac12}={\bf H} ( {\bf U}_{j+\frac12}^-,{\bf U}_{j+\frac12}^+; 1 )$. 
	Under the conditions \eqref{eq:1DDG:con1}--\eqref{eq:1DDG:con2}, we 
	have ${\bf H}_{j+\frac12} \in \overline{\mathcal G}_* $
	for all $j$ by using Theorem \ref{thm:HLLflux}. 
	The exactness of the $\tt L$-point Gauss--Lobatto quadrature rule for the polynomials of degree $k$ implies
	$$
	\bar{\bf U}_j^n = \frac{1}{\Delta x_j} \int_{I_j} {\bf U}_j^n ({x}) d x = \sum \limits_{\mu=1}^{\tt L} \widehat \omega_\mu {\bf U}_j^n (\widehat { x}_j^{(\mu)} ).
	$$
	Noting $\widehat \omega_1 = \widehat \omega_{\tt L}$ and $\widehat{ x}_j^{1,{\tt L}}={ x}_{j\mp\frac12}$ and using \eqref{1DHLL1}--\eqref{1DHLL2}, we can rewrite the scheme \eqref{eq:1DMHD:Hcellaverage} into the following convex combination form
	\begin{equation}\label{eq:1DMHD:convexsplit}
	\begin{aligned} 
	\bar{\bf U}_j^{n+1} & =
	\sum \limits_{\mu=2}^{{\tt L}-1} \widehat \omega_\mu {\bf U}_j^n ( \widehat {x}_j^{(\mu)} )  	+  \Big(2\widehat \omega_1 + \lambda \sigma_{j+\frac12}^{n,-} 
	- \lambda \sigma_{ j-\frac12 }^{n,+}  \Big) {\bf \Xi}
	\\
	& \quad 
	+ \Big(-\lambda \sigma_{j+\frac12}^{n,-}\Big)  {\bf H}_{j+\frac12}
	+ \lambda \sigma_{ j-\frac12 }^{n,+} 
	{\bf H}_{j-\frac12},
	\end{aligned}
	\end{equation}
	where $\lambda :=  \Delta t_n/\Delta x_j$, 
	and
	\begin{align*}
	& {\bf \Xi} := \frac{
		\left( \lambda^{-1} \widehat \omega_1 + \sigma^{n,-}_{j+\frac12} \right)
		{\bf U}_{j+\frac{1}{2}}^- 
		- {\bf F}_1 \big( {\bf U}_{j+\frac{1}{2}}^- \big) 
		+ \left( \lambda^{-1} \widehat \omega_1 - \sigma^{n,+}_{j-\frac12} \right)
		{\bf U}_{j-\frac{1}{2}}^+ 
		+ {\bf F}_1 \big( {\bf U}_{j-\frac{1}{2}}^+ \big) 
	}{  
		\lambda^{-1} \widehat \omega_1 + \sigma^{n,-}_{j+\frac12} 
		+  \lambda^{-1} \widehat \omega_1 - \sigma^{n,+}_{j-\frac12}  }.
	\end{align*}
	The condition \eqref{eq:CFL:1DMHD} implies 
	$$
	\lambda^{-1} \widehat \omega_1 + \sigma^{n,-}_{j+\frac12} \ge 
	\alpha_j^\star \ge \alpha_\star ( {\bf U}_{j+\frac12}^-, {\bf U}_{j-\frac12}^+ ;1 ),~ \lambda^{-1} \widehat \omega_1 - \sigma^{n,+}_{j-\frac12} \ge \alpha_j^\star \ge \alpha_\star ( {\bf U}_{j-\frac12}^+,{\bf U}_{j+\frac12}^-;1 ),
	$$
	which together with the condition \eqref{eq:1DDG:con1} yield ${\bf \Xi} \in \overline{\mathcal G}_*$
	by Corollary \ref{cor:1D}. 
	We therefore conclude $\bar{\bf U}_j^{n+1} \in {\mathcal G} $ from \eqref{eq:1DMHD:convexsplit}
	according to the convexity of ${\mathcal G}_*$ and Lemma \ref{theo:eqDefG}.
\end{proof} 

\begin{remark}
	In practice, it is easy to ensure 
	the condition \eqref{eq:1DDG:con1}, since the exact solution $B_1( x,t)\equiv {\tt B}_{\tt const}$ and
	the flux for $B_1$ in the $x$-direction is zero. The condition \eqref{eq:1DDG:con2} can also be easily enforced by a simple scaling limiter, which was 
	designed in \cite{cheng} by extending the  
	techniques in \cite{zhang2010,zhang2010b,zhang2011}. For readers' convenience, the PP limiter is briefly reviewed in Appendix \ref{app:PPlimit}.  
\end{remark}

The above PP schemes and analysis are focused on first-order time discretization. 
In fact, the high-order explicit time discretization are also be 
applied by 
using strong stability-preserving (SSP)  methods  (cf.~\cite{Gottlieb2009}). The PP analysis remains valid, because $\mathcal G$ is convex 
and an SSP method is a convex combination of the forward Euler method.

\section{Positivity-preserving schemes in multiple dimensions}\label{sec:2D}
In this section, we develop   
provably PP methods for 
the multidimensional ideal MHD. 
We remark that the design of multidimensional 
PP schemes have challenges essentially different from the 1D case, due to the divergence-free condition \eqref{eq:2D:BxBy0}.  
For the sake of clarity, we shall restrict ourselves 
to the 2D case ($d=2$), 
keeping in mind that our PP methods and 
analyses are extendable to the 3D case. 
We will use ${\bf x}  \in \mathbb R^d$ to denote the spatial coordinate vector.


\begin{figure}[htbp]
	\centering
	\includegraphics[width=0.43\textwidth]{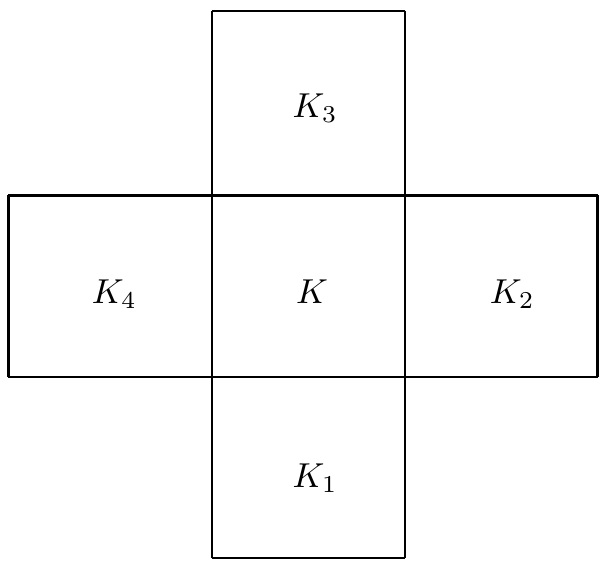}~~~~~~~~
	\includegraphics[width=0.37\textwidth]{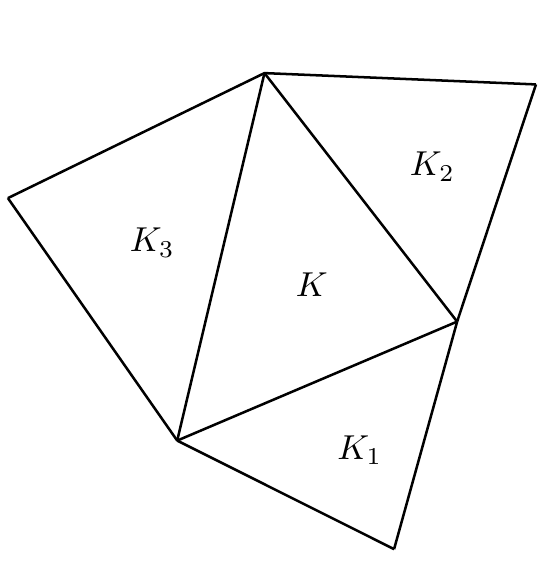}
	\caption{\small
		Illustration of a 
		rectangular mesh (left) and a triangular mesh (right).
	}\label{fig:ILLMesh}
\end{figure}

Assume that the 2D spatial domain is partitioned into a mesh ${\mathcal T}_h$, which can be unstructured 
and consists of 
polygonal cells. An illustration of two special meshes is given in Fig.\ \ref{fig:ILLMesh}. 
Let $K \in \mathcal{T}_h$ be a polygonal cell 
with edges   
$\mathscr{E}_{K}^{j}$, $j=1,\cdots,N_K$, 
and $K_j$ be the adjacent cell which shares the edge $\mathscr{E}_{K}^{j}$ with $K$.  
We denote by ${\bm \xi }_K^{(j)}=\big(\xi_{1,K}^{(j)},\cdots,\xi_{d,K}^{(j)} \big)$ the unit normal vector of $\mathscr{E}_K^j$ pointing from $K$ to $K_j$. 
The notations 
$|K|$ and $|{\mathscr E}_{K}^{j}|$ are used to denote the area of $K$ and the length of ${\mathscr E}_{K}^{j}$, respectively. 
The time interval is also divided into the mesh $\{t_0=0, t_{n+1}=t_n+\Delta t_{n}, n\geq 0\}$
with the time step-size $\Delta t_{n}$ determined by some CFL condition.  


\subsection{First-order schemes}


We consider the 
following first-order scheme for the Godunov form \eqref{eq:MHD:GP} of the ideal MHD equations
\begin{equation}\label{eq:1stschemeGP}
\bar{\bf U}_{K}^{n+1}
= \bar{\bf U}_{K}^{n} - \frac{\Delta t_n}{|K|}   \sum_{j=1}^{N_K}
\left|{\mathscr E}_{K}^{j}\right| \hat{\bf F} \big(  \bar{\bf U}_K^n , \bar{\bf U}_{K_j}^n  ; {\bm \xi}^{(j)}_K \big) 
- \Delta t_n \big( {\rm div} _{K} \bar {\bf B}^n \big) {\bf S}(\bar{\bf U}^n_K),
\end{equation}
where $\bar {\bf U}_{K}^n$ is the numerical approximation 
to the cell average of ${\bf U}({\bf x},t_n)$ over the cell $K$, and the numerical flux $\hat {\bf F}$ is taken as the HLL flux in \eqref{eq:1DHLL2}. 
As a discretization of the Godunov--Powell source, the last term at the right-hand side of \eqref{eq:1stschemeGP} is a penalty term, with $\mbox{\rm div} _{K} \bar {\bf B}^n$ defined by 
\begin{equation}\label{eq:DefDisDivB}
\mbox{\rm div} _{K} \bar {\bf B}^n := \frac{1}{|K|}
\sum_{j=1}^{N_K} \big| {\mathscr E}_{K}^{j} \big|
\left\langle {\bm \xi}_{K}^{(j)},  
\frac{ \sigma_{K,j}^{n,+} \bar {\bf B}_K^{n}  - \sigma_{K,j}^{n,-} \bar{\bf B}_{K_j}^n  }{
	\sigma_{K,j}^{n,+} - \sigma_{K,j}^{n,-} 
} \right\rangle,
\end{equation} 
where $\sigma_{K,j}^{n,\pm}:= 
\sigma^\pm (  \bar{\bf U}_K^n , \bar{\bf U}_{K_j}^n  ; {\bm \xi}^{(j)}_K )
$. 
The quantity  
$\mbox{\rm div} _{K} \bar {\bf B}^n$ can be considered as a discrete divergence of magnetic field, because it is  
a first-order accurate approximation to the left-hand side of 
\begin{equation*}
\frac{1}{|K|} \sum_{j=1}^{N_K} \int_{{\mathscr E}_{K}^{j}}  
\left\langle {\bm \xi}_{K}^{(j)}, {\bf B}({\bf x},t_n) \right\rangle ds
=
\frac{1}{|K|}\int_K \nabla \cdot {\bf B} d {\bf x}=0.
\end{equation*}
In the special case of using the LF type fluxes, $\sigma_{K,j}^{n,+}=-\sigma_{K,j}^{n,-}$, then the discrete divergence becomes 
\begin{equation*}
\mbox{\rm div} _{K} \bar {\bf B}^n = \frac{1}{|K|}
\sum_{j=1}^{N_K} \big| {\mathscr E}_{K}^{j} \big|
\left\langle {\bm \xi}_{K}^{(j)},  
\frac{  \bar {\bf B}_K^{n}  + \bar{\bf B}_{K_j}^n  }{
	2
} \right\rangle,
\end{equation*} 
which is consistent with the one introduced in \cite{Wu2017a,WuShu2018} on the Cartesian meshes.

The PP property of the scheme \eqref{eq:1stschemeGP} is shown as follows.

\begin{theorem}\label{thm:2DLF1st}
	Let the wave speeds in the HLL flux satisfy \eqref{eq:HLLsigma}.  	
	If $\bar{\bf U}_K^{n} \in {\mathcal G}$, $\forall K \in {\mathcal T}_h$, 
	then the solution $\bar{\bf U}^{n+1}_K$ of \eqref{eq:1stschemeGP} 
	belongs to $\mathcal G$ for all $ K \in {\mathcal T}_h$ under the CFL-type condition 
	\begin{equation}\label{eq:2D1stCFLGP}
	\Delta t_n \left( \frac{1}{|K|} 
	\sum_{j=1}^{N_K}   \big| {\mathscr E}_{K}^{j} \big|
	\left( - \sigma_{K,j}^{n,-} \right) + \frac{ \left|{\rm div} _{K} \bar {\bf B}^n\right| }{ \sqrt{\bar \rho_K^n} } \right) < 1,\quad \forall K \in {\mathcal T}_h. 
	\end{equation}
\end{theorem}

\begin{proof}
	Let ${\bf H}_{K,j}^n:={\bf H}( \bar{ \bf  U }_K^n,\bar{ \bf  U }_{K_j}^n;{\bm \xi}_K^{(j)} )$. Then the identity  \eqref{eq:FH1} implies 
	\begin{equation}\label{eq:aa1}
	\hat{\bf F} \big(  \bar{\bf U}_K^n , \bar{\bf U}_{K_j}^n  ; {\bm \xi}^{(j)}_K \big) = \sigma_{K,j}^{n,-} 
	{\bf H}_{K,j}^n 
	+ \left\langle {\bm \xi}_K^{(j)}, {\bf F}( \bar{\bf U}_K^n )
	\right\rangle - \sigma_{K,j}^{n,-} \bar {\bf U}_K^n.
	\end{equation}
	Using \eqref{eq:aa1} and the identity 
	\begin{equation}\label{eq:sumxi0}
	\sum_{j=1}^{N_k} \big| {\mathscr E}_{K}^{j} \big| {\bm \xi}_K^{(j)} ={\bf 0},
	\end{equation}
	one can rewrite the scheme \eqref{eq:1stschemeGP} as  
	\begin{equation}\label{eq:2DLFrewGP}
	\bar { \bf  U }_K^{n+1}  = 
	\frac{\Delta t_n}{|K|} 
	\sum_{j=1}^{N_K}   \big| {\mathscr E}_{K}^{j} \big|
	\left( - \sigma_{K,j}^{n,-} \right) {\bf H}_{K,j}^n 
	+ ( 1 - \lambda_K  ) \bar{ \bf  U }_K^n   - \Delta t_n \big( {\rm div} _{K} \bar {\bf B}^n \big) {\bf S}(\bar{\bf U}^n_K),
	\end{equation} 
	where 
	$
	\lambda_K := \frac{\Delta t_n}{|K|} 
	\sum_{j=1}^{N_K}   \big| {\mathscr E}_{K}^{j} \big|
	\big( - \sigma_{K,j}^{n,-} \big) \in [0,1).
	$ 
	Thanks to Theorem \ref{thm:HLLflux}, we have 
	${\bf H}_{K,j}^n  \in {\mathcal G}_\rho$ and for any  
	${\bf v}^*,{\bf B}^*\in {\mathbb R}^3$, 
	\begin{equation}\label{eq:Hnstar}
	{\bf H}_{K,j}^n \cdot {\bf n}^* 
	+ \frac{|{\bf B}^*|^2}{2}  \ge 
	-\frac{{\bf v}^* \cdot {\bf B}^*}{ \sigma_{K,j}^{n,+} - \sigma_{K,j}^{n,-} } \left \langle {\bm \xi}_K^{(j)},  
	\bar{\bf B}_{K_j}^n - \bar{\bf B}_K^n 
	\right  \rangle.
	\end{equation}
	
	Since ${\bf H}_{K,j}^n  \in {\mathcal G}_\rho$ and 
	the first component of ${\bf S}(\bar{\bf U}^n_K)$ is zero, 
	we have $\bar \rho^{n+1}_K \ge (1-\lambda_K) \bar \rho^n_K > 0$. 
	For any $ {\bf v}^*,{\bf B}^* \in \mathbb R^3$, using 
	\eqref{eq:indentityS} we derive from \eqref{eq:2DLFrewGP} that 
	\begin{equation*}
	\bar {\bf U}_{K}^{n+1} \cdot {\bf n}^* + \frac{ |{\bf B}^*|^2 }{2} = \Pi_1 + \Pi_2,
	\end{equation*}
	where 
	\begin{align*}
	\Pi_1 & :=  
	\frac{\Delta t_n}{|K|} 
	\sum_{j=1}^{N_K}   \big| {\mathscr E}_{K}^{j} \big|
	\left( - \sigma_{K,j}^{n,-} \right) \left( {\bf H}_{K,j}^n 
	\cdot {\bf n}^* + \frac{|{\bf B}^*|^2}{2} \right)  
	+ \Delta t_n \big( \mbox{\rm div} _{K} \bar {\bf B}^n \big) \big( {\bf v}^* \cdot {\bf B}^* \big),\\ 
	\Pi_2 & := 
	(1-\lambda_K) \left( \bar {\bf U}_{K}^n \cdot {\bf n}^* + \frac{|{\bf B}^*|^2}{2} \right)
	- \Delta t_n \big( {\rm div}_{K} \bar{\bf B}^n \big) ( \bar{\bf v}_{K}^n - {\bf v}^*  )
	\cdot  ( \bar{\bf B}_{K}^n - {\bf B}^*  ).
	\end{align*}
	Let us estimate the lower bounds of $\Pi_1$ and $\Pi_2$ respectively. 
	Using \eqref{eq:Hnstar} and \eqref{eq:sumxi0} gives 
	\begin{align*}
	\Pi_1 & 
	\overset { \eqref{eq:Hnstar} }
	{\ge}   
	\frac{\Delta t_n}{|K|} 
	\sum_{j=1}^{N_K}   \big| {\mathscr E}_{K}^{j} \big|
	\sigma_{K,j}^{n,-}   \frac{
		\big \langle {\bm \xi}_K^{(j)},  
		\bar{\bf B}_{K_j}^n -  \bar{\bf B}_K^n  \big  \rangle 
	}{ \sigma_{K,j}^{n,+} - \sigma_{K,j}^{n,-} } 
	\big( {\bf v}^* \cdot {\bf B}^* \big)
	+ \Delta t_n \big( \mbox{\rm div} _{K} \bar {\bf B}^n \big) \big( {\bf v}^* \cdot {\bf B}^* \big)
	\\
	& \overset { \eqref{eq:DefDisDivB} }{=} 
	\frac{\Delta t_n}{|K|} 
	\sum_{j=1}^{N_K}   \big| {\mathscr E}_{K}^{j} \big|
	\left( 
	\sigma_{K,j}^{n,-}\frac{
		\big \langle {\bm \xi}_K^{(j)}, \bar{\bf B}_{K_j}^n -  \bar{\bf B}_K^n 
		\big  \rangle 
	}{ \sigma_{K,j}^{n,+} - \sigma_{K,j}^{n,-} } 
	+ \left\langle {\bm \xi}_{K}^{(j)},  
	\frac{ \sigma_{K,j}^{n,+} \bar {\bf B}_K^{n}  - \sigma_{K,j}^{n,-} \bar{\bf B}_{K_j}^n  }{
		\sigma_{K,j}^{n,+} - \sigma_{K,j}^{n,-} 
	} \right\rangle
	\right)
	\big( {\bf v}^* \cdot {\bf B}^* \big)
	\\
	& = \frac{\Delta t_n}{|K|} 
	\sum_{j=1}^{N_K}   \big| {\mathscr E}_{K}^{j} \big| 
	\left \langle {\bm \xi}_K^{(j)}, 
	\bar{\bf B}_K^n \right \rangle \big( {\bf v}^* \cdot {\bf B}^* \big) \overset { \eqref{eq:sumxi0} } {=} 0.
	\end{align*}
	It follows from \eqref{eq:widelyusedIEQ} that  
	\begin{equation*}
	\begin{aligned}
	\Pi_2
	& \ge (1-\lambda_K) \left( \bar {\bf U}_{K}^n \cdot {\bf n}^* + \frac{|{\bf B}^*|^2}{2} \right)
	- \Delta t_n \frac{\left| {\rm div}_{K} \bar{\bf B}^n \right|}{\sqrt{\bar \rho ^n_K}} \Big| \sqrt{\bar \rho ^n_K}  ( \bar{\bf v}_{K}^n - {\bf v}^*  )
	\cdot  ( \bar{\bf B}_{K}^n - {\bf B}^*  ) \Big|
	\\
	& \overset {\eqref{eq:widelyusedIEQ} }{\ge} \left(1-\lambda_K - \Delta t_n \frac{\left| {\rm div}_{K} \bar{\bf B}^n \right|}{\sqrt{\bar \rho ^n_K}} \right) \left( \bar {\bf U}_{K}^n \cdot {\bf n}^* + \frac{|{\bf B}^*|^2}{2} \right) > 0.
	\end{aligned}
	\end{equation*}
	Therefore, $\bar {\bf U}_{K}^{n+1} \cdot {\bf n}^* + \frac{ |{\bf B}^*|^2 }{2}>0$, $\forall {\bf v}^*,{\bf B}^* \in \mathbb R^3$. 
	
	Hence $\bar {\bf U}_{K}^{n+1}\in \mathcal G$ by Lemma \ref{theo:eqDefG}.  
\end{proof}

It is worth emphasizing that the penalty term 
is crucial for guaranteeing the PP property of the scheme \eqref{eq:1stschemeGP}. 
While the scheme \eqref{eq:1stschemeGP} 
without this term reduces to the 2D HLL scheme for 
the conservative MHD system \eqref{eq:MHD}, specifically, 
\begin{equation}\label{eq:1stscheme}
\bar{\bf U}_{K}^{n+1}
= \bar{\bf U}_{K}^{n} - \frac{\Delta t_n}{|K|}   \sum_{j=1}^{N_K}
\big|{\mathscr E}_{K}^{j}\big| \hat{\bf F} \big(  \bar{\bf U}_K^n , \bar{\bf U}_{K_j}^n  ; {\bm \xi}^{(j)}_K \big). 
\end{equation} 
For the LF flux, 
the analysis in \cite{Wu2017a} on Cartesian meshes 
showed that the scheme \eqref{eq:1stscheme} 
is generally not PP, unless a discrete divergence-free (DDF) condition 
is satisfied. 
We find that, on a general mesh ${\mathcal T}_h$, the corresponding DDF condition is 
\begin{equation}\label{eq:DisDivB}
\mbox{\rm div} _{K} \bar {\bf B}^n 
= 0,\quad \forall K \in {\mathcal T}_h.
\end{equation} 
As a direct consequence of Theorem \ref{thm:2DLF1st}, we immediately have the following  
corollary. 

\begin{corollary}\label{prop:2D1st}
	Let the wave speeds in the HLL flux satisfy \eqref{eq:HLLsigma}.  	 
	If $\bar{\bf U}_K^{n} \in {\mathcal G}$, $\forall K \in {\mathcal T}_h$, 
	and satisfy the DDF condition \eqref{eq:DisDivB}, 
	then under the CFL condition 
	\begin{equation*}
	\frac{\Delta t_n}{|K|} 
	\sum_{j=1}^{N_K}   \big| {\mathscr E}_{K}^{j} \big|
	\left( - \sigma_{K,j}^{n,-} \right) < 1,\quad \forall K \in {\mathcal T}_h, 
	\end{equation*}
	the solution $\bar{\bf U}^{n+1}_K$ of \eqref{eq:1stscheme} 
	belongs to $\mathcal G$ for all $ K \in {\mathcal T}_h$. 
\end{corollary}

If ${\mathcal T}_h$ is a Cartesian mesh and the numerical flux $\hat {\bf F}$ is taken as the global LF flux, then 
the scheme \eqref{eq:1stscheme} preserves the DDF condition \eqref{eq:DisDivB} provided that the DDF condition is satisfied by the initial data \cite{Wu2017a}. 
It was also shown in \cite{Wu2017a} 
that even slightly violating the DDF condition 
can cause the failure of the scheme \eqref{eq:1stscheme} to preserve the positivity of pressure. 
Unfortunately, on general meshes 
the scheme \eqref{eq:1stscheme} does not necessarily preserve 
the DDF condition \eqref{eq:DisDivB}, and it is generally not PP.


\subsection{High-order schemes}
We are now in the position to discuss provably PP high-order schemes for the multidimensional ideal MHD. We mainly focus on 
the PP high-order DG methods, keeping in mind that 
the analysis and framework also apply 
to high-order finite volume schemes.

\subsubsection{Locally divergence-free schemes}\label{sec:LDF}
We first propose locally divergence-free schemes 
for the modified MHD system \eqref{eq:MHD:GP}, as  
they are the base schemes of our PP high-order schemes presented later.  
Towards achieving high-order spatial accuracy, the exact solution 
${\bf U}({\bf x},t_n)$ is approximated with a discontinuous piecewise polynomial function 
${\bf U}_h^n({\bf x})$, which is sought in the locally divergence-free DG 
space \cite{Li2005} 
\begin{equation*}
{\mathbb V}_h^{k} 
= \left\{ {\bf u}=(u_1,\cdots,u_8)^\top~\Big|~ u_\ell \big|_{K} \in {\mathbb P}^{k} (K),
\forall \ell,~   
\sum_{i=1}^d \frac{ \partial u_{4+i}}{\partial {x_i}} 
\bigg|_{K}  = 0,~\forall K \in {\mathcal T}_h   \right\},
\end{equation*}
where ${\mathbb P}^{k} (K)$ is the space of polynomials in $K$ of degree at most $k$.

%
%
%
%
%

Our ${\mathbb P}^{k}$-based locally divergence-free DG method 
is obtained by proper discretization of the Godunov form \eqref{eq:MHD:GP}. Specifically, our DG solution ${\bf U}_h^n \in {\mathbb V}_h^{k} $ is explicitly evolved by 
\begin{equation}\label{eq:2DDGUh}
\begin{split}
& \int_{K} {\bf u} \cdot  \frac{{\bf U}_h^{n+1}- {\bf U}_h^{n}}{\Delta t_n}   d x d y
=  \int_{K}   \nabla {\bf u} \cdot  
{\bf F}  ( {\bf U}_h^{n} )  d {\bf x}  
\\
& \quad  - \sum_{j=1}^{N_K} \int_{ {\mathscr E}_K^j } 
{\bf u}^{{\rm int}(K)} \cdot \bigg\{
\hat{\bf F} \left( {\bf U}_h^{n,{\rm int}(K)}, {\bf U}_h^{n,{\rm ext}(K)}; {\bm \xi}^{(j)}_{K} \right)   
\\ 
& \qquad 
- \left[
\eta_K ({\bf x})
\left \langle {\bm \xi}^{(j)}_K, {\bf B}_h^{n,{\rm ext}(K)} - {\bf B}_h^{n,{\rm int}(K)} \right \rangle 
{\bf S} \big({\bf U}_h^{n,{\rm int}(K)} \big)  \right] \bigg\} ds,\quad \forall {\bf u} \in {\mathbb V}_h^{k},
\end{split} 
\end{equation} 
where the numerical flux $\hat{\bf F}$ is taken as the HLL flux in \eqref{eq:1DHLL2}, and the factor  
$$
\eta_K ({\bf x}) := 
\frac{ \sigma^- \big(  {\bf U}_h^{ n,{\rm int} (K) } , {\bf U}_h^{ n,{\rm ext} (K) }   ; {\bm \xi}^{(j)}_K \big) }{ \sigma^+ \big(  {\bf U}_h^{ n,{\rm int} (K) } , {\bf U}_h^{ n,{\rm ext} (K) }   ; {\bm \xi}^{(j)}_K \big) 
	-  \sigma^- \big(  {\bf U}_h^{ n,{\rm int} (K) } , {\bf U}_h^{ n,{\rm ext} (K) }   ; {\bm \xi}^{(j)}_K \big) }, \quad \forall {\bf x} \in {\mathscr E}_K^j.
$$
Here the superscripts ``${\rm int}(K)$'' and ``${\rm ext}(K)$'' indicate that the associated limits at the interface ${\mathscr E}_K^j$ are taken from the interior and exterior of $K$, respectively. 
The term inside the bracket in \eqref{eq:2DDGUh} is a penalty term discretized from the Godunov--Powell source term. The factor $\eta_K$ is carefully devised in an upwind 
manner according to the local wave speeds in the HLL flux. 
This is motivated from our theoretical analysis, and is very important for achieving the provably PP property, as we will see the proof of Theorem \ref{thm:PP:2DMHD} and Remark \ref{eq:remark3}.  
If the LF flux is employed, i.e., $\sigma^-=-\sigma^+$, 
then $\eta_K({\bf x}) \equiv -\frac12$, 
and the penalty term reduces to 
the one used in \cite{WuShu2018}.

In the practical computations, the boundary and element integrals at the right-hand side of \eqref{eq:2DDGUh} 
are discretized by certain quadratures of sufficiently high order accuracy (specifically, the algebraic degree of accuracy should be at least $2k$). For example, we can 
employ the Gauss quadrature  
with $Q=k+1$ points for the boundary integral: 
\begin{align*}
\begin{split}
& \int_{ {\mathscr E}_K^j } {\bf u}^{{\rm int}(K)} \cdot 
\bigg[
\hat{\bf F} \left( {\bf U}_h^{n,{\rm int}(K)}, {\bf U}_h^{n,{\rm ext}(K)}; {\bm \xi}^{(j)}_{K} \right)  
\\ 
& \qquad 
-\eta_K ({\bf x})  
\left \langle {\bm \xi}^{(j)}_K, {\bf B}_h^{n,{\rm ext}(K)} - {\bf B}_h^{n,{\rm int}(K)} \right \rangle 
{\bf S} \big({\bf U}_h^{n,{\rm int}(K)} \big)  \bigg]  ds
\end{split}
\\
\begin{split}
& \quad  \approx |{\mathscr E}_K^j| \sum_{q=1}^Q \omega_q
{\bf u}^{{\rm int}(K)} ( {\bf x}_K^{(jq)} ) \cdot 
\bigg[
\hat{\bf F} \left( {\bf U}_h^{n,{\rm int}(K)} ( {\bf x}_K^{(jq)} ), {\bf U}_h^{n,{\rm ext}(K)} ( {\bf x}_K^{(jq)} ); {\bm \xi}^{(j)}_{K} \right)  
\\ 
& \qquad 
- \eta_K ( {\bf x}_K^{(jq)} )
\left \langle {\bm \xi}^{(j)}_K, {\bf B}_h^{n,{\rm ext}(K)} ( {\bf x}_K^{(jq)} ) - {\bf B}_h^{n,{\rm int}(K)} ( {\bf x}_K^{(jq)} ) \right \rangle 
{\bf S} \left({\bf U}_h^{n,{\rm int}(K)} ( {\bf x}_K^{(jq)} ) \right)  \bigg],
\end{split}
\end{align*}
where 
$ \{{\bf x}_K^{(jq)}\}_{1\le q \le Q}$ are the quadrature points on the interface ${\mathscr E}_K^j$, 
and $\{\omega_q\}_{1\le q \le Q}$ are the associated weights. 


Let 
\begin{equation*}
{\bf U}_h^n\big|_{K}=: {\bf U}_{K}^n ({\bf x}),
\end{equation*}
and its cell average over $K$ be $\bar{\bf U}_K^n$. Then we can obtain from 
\eqref{eq:2DDGUh} the 
evolution equations for the cell averages $\{\bar {\bf U}_{K}^n\}$ as follows 
\begin{equation}\label{eq:2DMHD:cellaverage}
\bar {\bf U}_{K}^{n+1}  = \bar {\bf U}_{K}^{n} + \Delta t_n {\bf L}_{K} ( {\bf U}_h^n  ),
\end{equation}
where 
\begin{equation*}
\begin{split}
{\bf L}_{K} ( {\bf U}_h^n  ) & := 
- \frac{1}{|K|} 
\sum_{j=1}^{N_K} \sum_{q=1}^Q |{\mathscr E}_K^j|
\omega_q
\bigg[
\hat{\bf F} \left( {\bf U}_h^{n,{\rm int}(K)} ( {\bf x}_K^{(jq)} ), {\bf U}_h^{n,{\rm ext}(K)} ( {\bf x}_K^{(jq)} ); {\bm \xi}^{(j)}_{K} \right)  
\\ 
& 
- \eta_K ( {\bf x}_K^{(jq)} )  
\left \langle {\bm \xi}^{(j)}_K, {\bf B}_h^{n,{\rm ext}(K)} ( {\bf x}_K^{(jq)} ) - {\bf B}_h^{n,{\rm int}(K)} ( {\bf x}_K^{(jq)} ) \right \rangle 
{\bf S} \left({\bf U}_h^{n,{\rm int}(K)} ( {\bf x}_K^{(jq)} ) \right)  \bigg].
\end{split}
\end{equation*}

The scheme \eqref{eq:2DMHD:cellaverage} can also be derived from a finite volume method for \eqref{eq:MHD:GP}, if the approximate function 
${\bf U}_h^n  $ in \eqref{eq:2DMHD:cellaverage} 
is reconstructed from  $\{\bar {\bf U}_{K}^n\}$ by using a locally 
divergence-free reconstruction approach (cf.~\cite{ZhaoTang2017,xu2016divergence}) such that ${\bf U}_h^n \in {\mathbb V}_h^k $. 

If choosing $k=0$, 
the above finite volume and DG schemes reduce to 
the first-order scheme \eqref{eq:1stschemeGP}, whose PP property has been proven in Theorem \ref{thm:2DLF1st}. 
If taking $k\ge 1$, 
the above high-order accurate DG and finite volume schemes are generally not PP. 
However, we find that these locally divergence-free schemes have the weak positivity, that is, they 
can be rendered provably PP by a simple limiting procedure, as demonstrated in the following.  Note that the 
standard multidimensional DG schemes does not have the weak positivity, even if the locally divergence-free element is used.

\subsubsection{Positivity-preserving schemes}\label{sec:2DhstPP}
We first assume that there exists a special 2D quadrature on each cell $K \in {\mathcal T}_h$ satisfying:
\begin{itemize}
	\item The quadrature rule is with positive weights and 
	exact for integrals of polynomials of degree up to $k$ on the cell $K$.
	\item The set of the quadrature points, denoted by ${\mathbb S}_K$, must include all the Gauss quadrature points $ {\bf x}_K^{(jq)}$, $j=1,\dots,N_K$, $q=1,\dots,Q$, 
	on the cell interface. 
\end{itemize}
In other words, we would like to have a special quadrature such that 
\begin{equation}\label{eq:decomposition}
\frac{1}{|K|}\int_K u({\bf x}) d{\bf x}  = 
\sum_{j=1}^{N_K} \sum_{q=1}^Q \varpi_{jq} u ( {\bf x}_K^{(jq)} ) + \sum_{ q=1 }^{\widetilde Q} \widetilde \varpi_q   
u ( \widetilde {\bf x}_K^{(q)} ), \quad \forall 
u \in {\mathbb P}^{k}(K) ,
\end{equation}
where $\{\widetilde {\bf x}_K^{(q)}\}$ are the other (possible) quadrature nodes in $K$, and the quadrature weights 
$\varpi_{jq},\widetilde \varpi_q $ are positive and satisfy 
$
\sum_{j=1}^{N_K} \sum_{q=1}^Q \varpi_{jq} + \sum_{ q=1 }^{\widetilde Q} \widetilde \varpi_q =1.
$
For rectangular cells, such a quadrature was constructed in \cite{zhang2010,zhang2010b} by tensor products of Gauss quadrature and Gauss--Lobatto quadrature. 
For triangular cells, it can be constructed by a Dubinar transform from rectangles to triangles \cite{zhang2012maximum}. 
For more general polygonal cells, one can always decompose 
the polygons into non-overlapping triangles, and then build the above quadrature rule by gathering those on the small triangles; see, for example,   \cite{VILAR2016416,du2018positivity}. 
An illustration of the special quadrature on rectangle and triangle for $k=2$ is shown in Fig.\ \ref{fig:GS}, 
where the (red) solid points 
are $\{{\bf x}_K^{(jq)}\}$ and the (blue) hollow circles 
denote $\{\widetilde {\bf x}_K^{(q)}\}$.
We remark that such a special quadrature is not employed for computing any integral, but only used in the PP limiter 
and theoretical analysis as it decomposes the cell average into a convex combination of the desired point values.

\begin{figure}[htbp]
	\centering
	\includegraphics[width=0.99\textwidth]{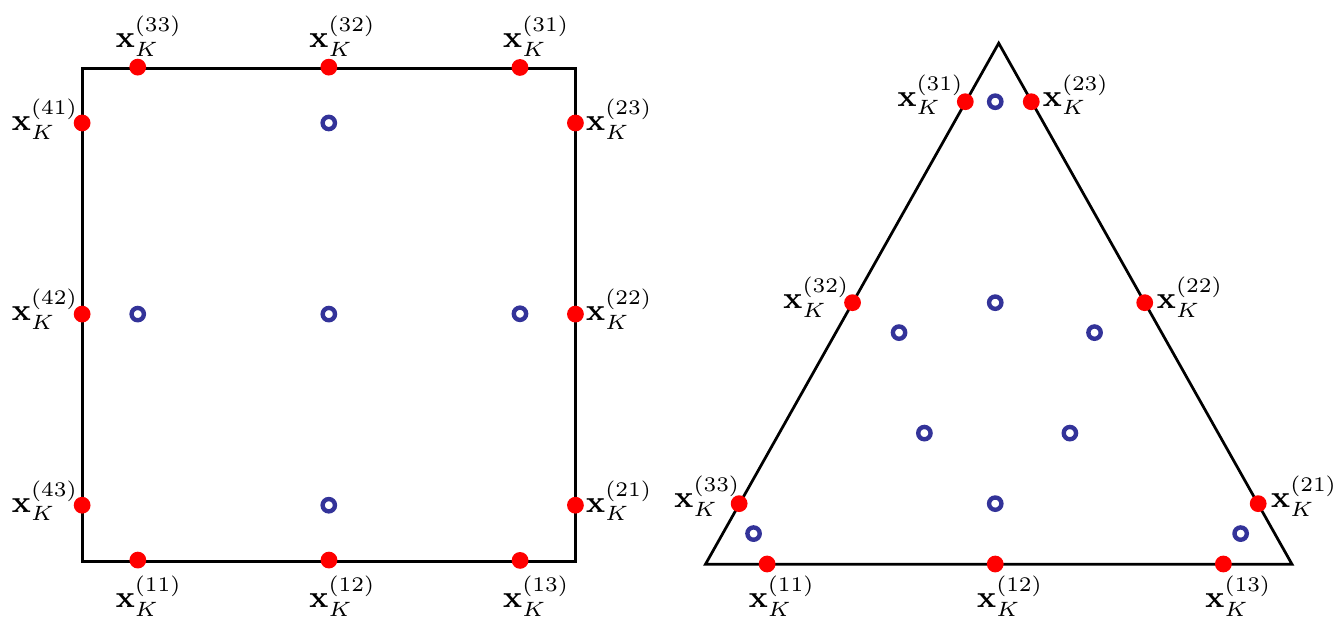}
	\caption{\small
		Illustration of the quadrature  \eqref{eq:decomposition} on a rectangular cell (left) and 
		a triangular cell (right) for $k=2$. The (red) solid points 
		are $\{{\bf x}_K^{(jq)}\}$ and the (blue) hollow circles 
		denote $\{\widetilde {\bf x}_K^{(q)}\}$; all of them  constitute the point set ${\mathbb S}_K$. 
	}\label{fig:GS}
\end{figure}

Based on the above special quadrature and the high-order locally divergence-free schemes in Section \ref{sec:LDF},  
our provably PP high-order DG and finite volume schemes are constructed  
as follows. The rigorous proof of the PP property is very technical and  will be presented later. 

\vspace{2mm}
\noindent
{\bf Step 0.} Initialization. After setting $t=0$ and $n=0$,  
we compute the cell averages $\{\bar {\bf U}_{K}^0\}$ 
and the polynomial functions $\{ {\bf U}_{K}^0 ({\bf x}) \}$
by a local $L^2$-projection of the initial data onto $ {\mathbb V}_h^{k}$, so that 
$\bar {\bf U}_{K}^0 \in {\mathcal G}$ is ensured by the convexity of $\mathcal G$, 
and ${\bf U}_h^0 \in  {\mathbb V}_h^{k} $ is also 
guaranteed.

\vspace{2mm}
\noindent
{\bf Step 1.} Given admissible cell-averaged solution $\big\{\bar {\bf U}_{K}^n\big\}$ 
and ${\bf U}_h^n \in   {\mathbb V}_h^{k} $, perform the PP limiting procedure. 
The PP limiter \cite{cheng} is employed to modify the polynomials $\big\{ {\bf U}_{K}^n ({\bf x})  \big\}$, such that the modified polynomials $\big\{\widetilde {\bf U}_{K}^n ({\bf x})  \big\}$ satisfy
\begin{equation}\label{eq:FVDGsuff}
\widetilde {\bf U}_{K}^n ({\bf x}) 
\in {\mathcal G}, \quad 
\forall  {\bf x}  \in {\mathbb S}_{K} := 
\left\{ \widetilde {\bf x}_K^{(q)} \right\}_{ 1\le q\le \widetilde Q } \bigcup 
\left\{ {\bf x}_K^{(jq)} \right\}_{ 1\le j \le N_K, 1\le q \le Q }.
\end{equation}
For readers' convenience, the PP limiter is briefly reviewed in Appendix \ref{app:PPlimit}. 
Let $\widetilde{\bf U}_h^n ( {\bf x} )$ be the (discontinuous)  piecewise polynomial function
corresponding to $\widetilde{\bf U}_{K}^n ({\bf x})$. 
Note that the limited function 
$\widetilde{\bf U}_h^n  \in {\mathbb V}_h^{k} $,
since the PP limiter only involves certain local 
convex combination of ${\bf U}_{K}^n ({\bf x})$ and its cell average over $K$.

\vspace{2mm}
\noindent
{\bf Step 2.} Update the cell averages by the scheme
\begin{equation}\label{eq:PP2DMHD:cellaverage}
\bar {\bf U}_{K}^{n+1}  = \bar {\bf U}_{K}^{n} + \Delta t_n {\bf L}_{K} ( \widetilde{\bf U}_h^n  ),
\end{equation}
{\em As will be shown in Theorem \ref{thm:PP:2DMHD}, 
	because of the weak positivity of our locally divergence-free schemes, 
	the PP limiting procedure in {Step 1} can ensure 
	$\bar {\bf U}_{K}^{n+1} \in {\mathcal G}$, which meets the requirement of performing the PP limiting procedure
	in the next time-forward step.}

\vspace{2mm}
\noindent
{\bf Step 3.} Build the piecewise polynomial function 
${\bf U}_h^{n+1}$. For our $\mathbb{P}^{k}$-based DG method $({k}\ge 1)$, 
 the high-order ``moments'' of the polynomials $\{ {\bf U}_{K}^{n+1}({\bf x}) \}$ are evolved by \eqref{eq:2DDGUh} with   
${\bf U}_h^n$ replaced with $\widetilde{\bf U}_h^n$. 
For a high-order finite volume scheme, 
 the approximate solution polynomials $\{{\bf U}_{K}^{n+1}({\bf x})\}$ are reconstructed from the
cell averages $\big\{ \bar{\bf U}_{K}^{n+1} \big\}$ by a locally 
divergence-free method such that ${\bf U}_h^{n+1} \in {\mathbb V}_h^{k}  $. We here omit the details, because these does not affect the PP property of the  schemes.

\vspace{2mm}
\noindent
{\bf Step 4.} Set $t_{n + 1}  = t_n  + \Delta t_n$. If $t_{n + 1}  < T$, assign $n \leftarrow n+1$ and go to {Step 1}, where 
$\bar {\bf U}_{K}^{n+1} \in {\mathcal G}$ 
has been guaranteed in Step 2; otherwise, output results.

Now, we give the proof of the PP property of the above schemes,
i.e., prove that 
the cell average $\bar {\bf U}_{K}^{n+1}$ 
computed by the scheme \eqref{eq:PP2DMHD:cellaverage} always stays  $\mathcal G$ under the condition \eqref{eq:FVDGsuff}. 
It is worth emphasizing that 
the locally divergence-free spatial discretization 
and the penalty term in \eqref{eq:2DDGUh}  
are crucial for achieving the provably PP scheme, as will be seen from the proof of Theorem \ref{thm:PP:2DMHD}. 

To shorten the notations, we define 
\begin{equation*}
{\bf U}^{{\rm int}(K)}_{jq}:=
\widetilde{\bf U}_h^{n,{\rm int}(K)} ( {\bf x}_K^{(jq)} ),\qquad {\bf U}^{{\rm ext}(K)}_{jq}:=
\widetilde {\bf U}_h^{n,{\rm ext}(K)} ( {\bf x}_K^{(jq)} ),
\end{equation*}
where the dependence on $n$ is omitted. 
Let 
$$\sigma_{jq}^{K,\pm}:= \sigma^{\pm} 
\big(  { \bf  U }_{jq}^{{\rm int}(K)},{ \bf  U }_{jq}^{{\rm ext}(K)} ; {\bm \xi}^{(j)}_K \big).
$$
For $\forall K \in {\mathcal T}_h$, we define 
\begin{equation*}
\begin{split}
& \widehat \alpha_{jq}^{{\rm int}(K)}  := 
{\mathscr C}( {\bf U}_{jq}^{{\rm int}(K)};{\bm \xi}_K^{(j)} ) + 
\frac{2}{|\partial K|}
\sum_{i=1}^{N_K} |{\mathscr E}_K^i| \frac{ |{\bf B}_{jq}^{{\rm int}(K)}- {\bf B}_{iq}^{{\rm int}(K)}| }{ \sqrt{\rho_{jq}^{{\rm int}(K)}}  +	\sqrt{\rho_{iq}^{{\rm int}(K)}}  } 
\\ &  
+ \max \left\{ \big\langle {\bm \xi}^{(j)}_K, {\bf v}^{{\rm int}(K)}_{jq} \big\rangle , 
\frac{1}{|\partial K| }
\sum_{i=1}^{N_K} |{\mathscr E}_K^i| 
\left \langle
{\bm \xi}^{(j)}_K - {\bm \xi}^{(i)}_K, 
\frac{ \sqrt{\rho_{jq}^{{\rm int}(K)}} {\bf v }_{jq}^{{\rm int}(K)} + 
	\sqrt{\rho_{iq}^{{\rm int}(K)}} {\bf v }_{iq}^{{\rm int}(K)} } 
{ \sqrt{\rho_{jq}^{{\rm int}(K)}}  +	\sqrt{\rho_{iq}^{{\rm int}(K)}}  } \right \rangle \right\},
\end{split}
\end{equation*}	
with $|\partial K|:=\sum_{i=1}^{N_K} |{\mathscr E}_K^i| $ denoting the circumference of the cell $K$.

\begin{theorem} \label{thm:PP:2DMHD}
	Let the wave speeds in the HLL flux satisfy \eqref{eq:HLLsigma}.  	
	If the polynomial vectors $\{\widetilde{\bf U}_{K}^n({\bf x})\}$ are locally divergence-free and satisfy 
	the condition \eqref{eq:FVDGsuff}, 
	then the scheme \eqref{eq:PP2DMHD:cellaverage} preserves 
	$\bar{\bf U}_{K}^{n+1} \in {\mathcal G}$ under the CFL-type condition
	\begin{equation}\label{eq:CFL:2DMHD}
	\Delta t_n \frac{| {\mathscr E}_K^j |}{|K|} 
	\alpha_{jq}^{K} < \frac{ \varpi_{jq}}{\omega_q},\qquad \forall
	K\in{\mathcal T}_h,~1\le j \le N_K,~1\le q \le Q,
	\end{equation}
	with 
	\begin{equation}\label{eq:alphamax}
	\alpha_{jq}^{K}:=
	\widehat \alpha_{jq}^{{\rm int}(K)}  - \sigma_{jq}^{K,-}
	- \eta_K \big( {\bf x}_K^{(jq)} \big) 
	\Big( {\rho_{jq}^{{\rm int}(K)} } \Big)^{-\frac12} \left| 
	\big\langle 
	{\bm \xi}_K^{(j)},  
	{\bf B}_{jq}^{{\rm int}(K)} - {\bf B}_{jq}^{{\rm ext}(K)} 
	\big\rangle
	\right|.
	\end{equation}
\end{theorem}

Note that $\sigma_{jq}^{K,-}\le 0$ and $-1\le \eta_K \big( {\bf x}_K^{(jq)} \big)  \le 0$. The last term in \eqref{eq:alphamax} 
is relatively small compared to the maximum signal speed, and thus 
does not cause strict restriction on the time step-size; see 
the detailed justification and numerical evidence in \cite{WuShu2018}. 

We now present the proof of Theorem \ref{thm:PP:2DMHD}. 
\begin{proof}
	Recalling the identity  \eqref{eq:FH1} and Theorem \ref{thm:HLLflux}, one has 
	\begin{equation*}
	\begin{split}
	&	\hat{\bf F} \big(  { \bf  U }_{jq}^{{\rm int}(K)},{ \bf  U }_{jq}^{{\rm ext}(K)} ; {\bm \xi}^{(j)}_K \big) = \sigma_{jq}^{K,-} 
	{\bf H}_{jq}^K 
	+ \left\langle {\bm \xi}_K^{(j)}, {\bf F}( {\bf U}_{jq}^{{\rm int}(K)} )
	\right\rangle - \sigma_{jq}^{K,-}  {\bf U}_{jq}^{{\rm int}(K)}
	\\
	& =  \big( \widehat \alpha_{jq}^{{\rm int}(K)} - \sigma_{jq}^{K,-} \big)  {\bf U}_{jq}^{{\rm int}(K)}
	- \left(  \widehat \alpha_{jq}^{{\rm int}(K)} {\bf U}_{jq}^{{\rm int}(K)} - \big\langle {\bm \xi}_K^{(j)}, {\bf F}( {\bf U}_{jq}^{{\rm int}(K)} )
	\big\rangle \right)
	+
	\sigma_{jq}^{K,-} 
	{\bf H}_{jq}^K, 	
	\end{split}
	\end{equation*}
	where 
	${\bf H}_{jq}^K:={\bf H}( { \bf  U }_{jq}^{{\rm int}(K)},{ \bf  U }_{jq}^{{\rm ext}(K)};{\bm \xi}_K^{(j)} ) \in {\mathcal G}_\rho$ and for $\forall {\bf v}^*,{\bf B}^* \in \mathbb R^3$, 
	\begin{equation}\label{eq:Hnstar2}
	{\bf H}_{jq}^K \cdot {\bf n}^* 
	+ \frac{|{\bf B}^*|^2}{2}  \ge 
	-\frac{{\bf v}^* \cdot {\bf B}^*}{ \sigma_{jq}^{K,+} - \sigma_{jq}^{K,-} } \left \langle {\bm \xi}_K^{(j)},  
	{\bf B}_{jq}^{{\rm ext}(K)} - {\bf B}_{jq}^{{\rm int}(K)} 
	\right  \rangle.
	\end{equation}
	Plugging the above formula of $\hat{\bf F}$ into 
	\eqref{eq:PP2DMHD:cellaverage}, 
	we can rewrite  the scheme
	\eqref{eq:PP2DMHD:cellaverage} as  
	\begin{equation}\label{eq:decomp}
	\bar {\bf U}_{K}^{n+1}  
	= \bar {\bf U}_{K}^{n} + {\bf \Xi}_1 + 
	{\bf \Xi}_2 + {\bf \Xi}_3  + {\bf \Xi}_4,
	\end{equation}	
	with  
	\begin{align*}
	& {\bf \Xi}_1 := \frac{\Delta t_n}{|K|} 
	\sum_{j=1}^{N_K} \sum_{q=1}^Q |{\mathscr E}_K^j|
	\omega_q \left( \sigma_{jq}^{K,-}- \widehat \alpha_{jq}^{{\rm int}(K)}  \right)  {\bf U}_{jq}^{{\rm int}(K)}
	\\
	& {\bf \Xi}_2 := \frac{\Delta t_n}{|K|} 
	\sum_{j=1}^{N_K} \sum_{q=1}^Q |{\mathscr E}_K^j|
	\omega_q \left(  \widehat \alpha_{jq}^{{\rm int}(K)} {\bf U}_{jq}^{{\rm int}(K)} - \big\langle {\bm \xi}_K^{(j)}, {\bf F}( {\bf U}_{jq}^{{\rm int}(K)} )
	\big\rangle \right),
	\\
	&
	{\bf \Xi}_3 := 
	\frac{\Delta t_n}{|K|} 
	\sum_{j=1}^{N_K} \sum_{q=1}^Q |{\mathscr E}_K^j|
	\omega_q \left( - \sigma_{jq}^{K,-} \right) {\bf H}_{jq}^K,		
	\\
	& {\bf \Xi}_4 := \frac{\Delta t_n}{|K|} \sum_{j=1}^{N_K} \sum_{q=1}^Q 
	\big|{\mathscr E}_K^j\big| \omega_q \eta_K \big( {\bf x}_K^{(jq)} \big)
	\left\langle {\bm \xi}_K^{(j)}, 
	{\bf B}_{jq}^{{\rm ext}(K)} - {\bf B}_{jq}^{{\rm int}(K)} 
	\right \rangle {\bf S} \big(  {\bf U}_{jq}^{{\rm int}(K)} \big). 	
	\end{align*}	
	For $1\le q \le Q$, let 
	$$
	\overline{{\bf U}}^{{\rm int}(K)}_q := \frac{1}{ \sum\limits_{j=1}^{N_K} \big|{\mathscr E}_K^j\big| \widehat \alpha_{jq}^{{\rm int}(K)} } \sum_{j=1}^{N_K} |{\mathscr E}_K^j|
	\left(  \widehat \alpha_{jq}^{{\rm int}(K)} {\bf U}_{jq}^{{\rm int}(K)} - \big\langle {\bm \xi}_K^{(j)}, {\bf F}( {\bf U}_{jq}^{{\rm int}(K)} )
	\big\rangle \right),
	$$
	then ${\bf \Xi}_2$ can be reformulated as 
	\begin{equation}\label{eq:Xi2}
	{\bf \Xi}_2 = 
	\frac{\Delta t_n}{|K|} 
	\sum_{q=1}^Q 
	\omega_q \left( \sum_{j=1}^{N_K} |{\mathscr E}_K^j| \widehat \alpha_{jq}^{{\rm int}(K)} \right) \overline{{\bf U}}^{{\rm int}(K)}_q.
	\end{equation}
	Thanks to Theorem \ref{lem:main} and Eq.~\eqref{eq:sumxi0}, we have, for all $1\le q \le Q$, 
	$\overline{{\bf U}}^{{\rm int}(K)}_q \in {\mathcal G}_\rho$ and    
	\begin{equation*}
	\overline{\bf U}^{{\rm int}(K)}_q \cdot {\bf n}^* 
	+ \frac{|{\bf B}^*|^2}{2} \ge - \frac{ {\bf v}^* \cdot {\bf B}^*  }{{\sum\limits_{j = 1}^{N_K}  { |{\mathscr E}_K^j|    \widehat \alpha_{jq}^{{\rm int}(K)}   } }} 
	\sum_{j=1}^{ N_K } 
	|{\mathscr E}_K^j|   \big\langle {\bm \xi}^{(j)}_K, {\bf B}_{jq}^{{\rm int}(K)} \big\rangle, \quad \forall {\bf v}^*,{\bf B}^* \in {\mathbb{R}}^3.
	\end{equation*}
	Note $\sum_{j=1}^{N_K} |{\mathscr E}_K^j| \widehat \alpha_{jq}^{{\rm int}(K)} >0$ as indicated in Remark \ref{rem:sumsposi}. Therefore, ${\bf \Xi}_2 \in {\mathcal G}_\rho$, and 
	\begin{equation}\label{eq:proofPi2}
	\begin{aligned} 
	{\Pi}_2 &:= \frac{\Delta t_n}{|K|} 
	\sum_{q=1}^Q 
	\omega_q \left( \sum_{j=1}^{N_K} |{\mathscr E}_K^j| \widehat \alpha_{jq}^{{\rm int}(K)} \right) 
	\left( \overline{{\bf U}}^{{\rm int}(K)}_q 
	\cdot {\bf n}^* + \frac{|{\bf B}^*|^2}{2}
	\right)
	\\  
	&\ge 
	-\frac{\Delta t_n}{|K|} \big({\bf v}^*\cdot{\bf B}^*\big) \sum_{q=1}^Q 
	\omega_q 
	\sum_{j=1}^{ N_K } 
	|{\mathscr E}_K^j|   \big\langle {\bm \xi}^{(j)}_K, {\bf B}_{jq}^{{\rm int}(K)} \big\rangle.
	\end{aligned}
	\end{equation}
	It follows that 
	\begin{equation}\label{eq:proofPi2>0}
	\begin{aligned} 
	\Pi_2 & \ge -\frac{\Delta t_n}{|K|} \big({\bf v}^*\cdot{\bf B}^*\big) 
	\sum_{j=1}^{ N_K } 
	\int_{{\mathscr E}_K^j}   \big\langle {\bm \xi}^{(j)}_K, \widetilde{\bf B}_{K}^n \big\rangle ds
	\\
	& = -\frac{\Delta t_n}{|K|} \big({\bf v}^*\cdot{\bf B}^*\big) 
	\int_K \big(
	\nabla \cdot \widetilde{\bf B}_{K}^n \big)
	d {\bf x} = 0,
	\end{aligned}
	\end{equation}
	where we have sequentially used the exactness of the $Q$-point quadrature rule on 
	each interface for polynomials of degree up to $k$, Green's theorem and the locally divergence-free 
	property of the polynomial vector 
	$\widetilde{\bf B}_{K}^n ({\bf x}) $.

	Now, we first show $\bar\rho_K^{n+1}>0$. 
	Recalling that the first component of ${\bf S}({\bf U})$ is zero, we know that the first component of ${\bf \Xi}_4$ is zero. 
	Since ${\bf \Xi}_2 \in {\mathcal G}_\rho$ and ${\bf H}_{jq}^K \in {\mathcal G}_\rho$, $1\le j \le N_K$, $1\le q \le Q$, we deduce from \eqref{eq:decomp} that 
 \begin{align*}
\bar{ \rho }_K^{n+1} & > \bar{ \rho }_K^{n} + \frac{\Delta t_n}{|K|} 
\sum_{j=1}^{N_K} \sum_{q=1}^Q |{\mathscr E}_K^j|
\omega_q \left( \sigma_{jq}^{K,-}- \widehat \alpha_{jq}^{{\rm int}(K)}  \right)  \rho_{jq}^{{\rm int}(K)}
\\
\begin{split}
& = \sum_{ q=1 }^{\widetilde Q} \widetilde \varpi_q   
\widetilde \rho_K^{n} ( \widetilde {\bf x}_K^{(q)} )
+ \sum_{j=1}^{N_K} \sum_{q=1}^Q \varpi_{jq} \rho_{jq}^{{\rm int}(K)} 
\\
& \quad + \frac{\Delta t_n}{|K|} 
\sum_{j=1}^{N_K} \sum_{q=1}^Q |{\mathscr E}_K^j|
\omega_q \left( \sigma_{jq}^{K,-}- \widehat \alpha_{jq}^{{\rm int}(K)}  \right)  \rho_{jq}^{{\rm int}(K)}
\end{split}
\\
& \ge  \sum_{j=1}^{N_K} \sum_{q=1}^Q \omega_q \rho_{jq}^{{\rm int}(K)}
\left( 
\frac{ \varpi_{jq} }{\omega_q} - \frac{\Delta t_n}{|K|} 
\big| {\mathscr E}_K^j \big|  \left( \widehat \alpha_{jq}^{{\rm int}(K)} - \sigma_{jq}^{K,-}
\right) \right) \ge 0,
\end{align*}
	where we have used in the above equality the exactness of the quadrature rule \eqref{eq:decomposition} for polynomials of degree up to $k$,  and in the last inequality the condition 
	\eqref{eq:CFL:2DMHD}.

	We then prove for any ${\bf v}^*,{\bf B}^* \in \mathbb R^3$ that 
	$\bar {\bf U}_K^{n+1} \cdot {\bf n}^* 
	+ \frac{|{\bf B}^*|^2}{2} > 0$. 
	It follows from \eqref{eq:decomp} 
	that
	\begin{equation}\label{eq:decompEQ}
	\bar {\bf U}_K^{n+1} \cdot {\bf n}^* 
	+ \frac{|{\bf B}^*|^2}{2} 
	= \Pi_0 + \Pi_1 + \Pi_2 + \Pi_3 + \Pi_4,
	\end{equation}	
	where $\Pi_2 \ge 0$ is defined in \eqref{eq:proofPi2}, $\Pi_4 := {\bf \Xi}_4 \cdot {\bf n}^*$, and 
	\begin{align}\label{eq:Pi00}
	& \Pi_0 :=  \bar {\bf U}_K^{n} \cdot {\bf n}^* 
	+ \frac{|{\bf B}^*|^2}{2},
	\\ \label{eq:Pi11}
	& \Pi_1 := \frac{\Delta t_n}{|K|} 
	\sum_{j=1}^{N_K} \sum_{q=1}^Q |{\mathscr E}_K^j|
	\omega_q \left( \sigma_{jq}^{K,-}- \widehat \alpha_{jq}^{{\rm int}(K)}  \right) \left(  {\bf U}_{jq}^{{\rm int}(K)} \cdot {\bf n}^* 
	+ \frac{|{\bf B}^*|^2}{2}   \right),
	\\ \label{eq:Pi33}
	& \Pi_3 := \frac{\Delta t_n}{|K|} 
	\sum_{j=1}^{N_K} \sum_{q=1}^Q |{\mathscr E}_K^j|
	\omega_q \left( - \sigma_{jq}^{K,-} \right) 
	\left( {\bf H}_{jq}^K \cdot {\bf n}^* 
	+ \frac{|{\bf B}^*|^2}{2}   \right).
	\end{align}
	We now estimate the lower bounds of $\Pi_0$, $\Pi_3$ and 
	$\Pi_4$ respectively. 	
	Based on the exactness of the quadrature rule 
	\eqref{eq:decomposition} for polynomials of degree up to $k$, we can decompose the cell average as 
	$$
	\bar {\bf U}_K^{n}  
	= \frac{1}{|K|}\int_K \widetilde {\bf U}_h^n({\bf x}) d{\bf x}  = \sum_{ q=1 }^{\widetilde Q} \widetilde \varpi_q   
	\widetilde {\bf U}_h^n ( \widetilde {\bf x}_K^{(q)} )
	+ \sum_{j=1}^{N_K} \sum_{q=1}^Q \varpi_{jq} 
	{\bf U}_{jq}^{ {\rm int}(K) }.
	$$
	It follows that 
	\begin{equation}\label{eq:Pi0}
	\begin{aligned}
	\Pi_0 & = 
	\sum_{ q=1 }^{\widetilde Q} \widetilde \varpi_q   
	\left( 
	\widetilde {\bf U}_h^n ( \widetilde {\bf x}_K^{(q)} ) 
	\cdot {\bf n}^* 
	+ \frac{|{\bf B}^*|^2}{2}
	\right)
	+ \sum_{j=1}^{N_K} \sum_{q=1}^Q \varpi_{jq} 
	\left( 
	{\bf U}_{jq}^{ {\rm int}(K) } \cdot {\bf n}^* 
	+ \frac{|{\bf B}^*|^2}{2} \right) 
	\\
	& \ge \sum_{j=1}^{N_K} \sum_{q=1}^Q \varpi_{jq} 
	\left( 
	{\bf U}_{jq}^{ {\rm int}(K) } \cdot {\bf n}^* 
	+ \frac{|{\bf B}^*|^2}{2} \right),
	\end{aligned}
	\end{equation}	
	where the inequality follows from Lemma \ref{theo:eqDefG} and	$\widetilde {\bf U}_h^n ( \widetilde {\bf x}_K^{(q)} ) \in {\mathcal G}$ according to \eqref{eq:FVDGsuff}.
	Noting $\sigma_{jq}^{K,-}\le 0$ and using \eqref{eq:Hnstar2} give a lower bound of 
	$\Pi_3$ as 
	\begin{equation}\label{eq:lowerPi3}
	\begin{aligned}
	\Pi_3 & \ge \frac{\Delta t_n}{|K|} ( {\bf v}^* \cdot {\bf B}^* ) 
	\sum_{j=1}^{N_K} \sum_{q=1}^Q |{\mathscr E}_K^j|
	\omega_q  
	\frac{\sigma_{jq}^{K,-}}{ \sigma_{jq}^{K,+} - \sigma_{jq}^{K,-} } \left \langle {\bm \xi}_K^{(j)},  
	{\bf B}_{jq}^{{\rm ext}(K)} - {\bf B}_{jq}^{{\rm int}(K)} 
	\right  \rangle
	\\
	& =  \frac{\Delta t_n}{|K|} ( {\bf v}^* \cdot {\bf B}^* ) 
	\sum_{j=1}^{N_K} \sum_{q=1}^Q |{\mathscr E}_K^j|
	\omega_q  
	\eta_K \big( {\bf x}_K^{(jq)} \big) \left \langle {\bm \xi}_K^{(j)},  
	{\bf B}_{jq}^{{\rm ext}(K)} - {\bf B}_{jq}^{{\rm int}(K)} 
	\right  \rangle. 
	\end{aligned}
	\end{equation}
	A lower bound of $\Pi_4$ can be derived 
	by using the inequality \eqref{eq:widelyusedIEQ2} as 
	\begin{align*}
	\begin{split}
	\Pi_4 & \ge 
	\frac{\Delta t_n}{|K|} \sum_{j=1}^{N_K} \sum_{q=1}^Q 
	\big|{\mathscr E}_K^j\big| \omega_q 
	\Bigg[ \eta_K \big( {\bf x}_K^{(jq)} \big)
	\left\langle {\bm \xi}_K^{(j)}, 
	{\bf B}_{jq}^{{\rm int}(K)} - {\bf B}_{jq}^{{\rm ext}(K)} 
	\right \rangle  ( {\bf v}^* \cdot {\bf B}^* )
	\\
	& \quad  - \left(\rho_{jq}^{{\rm int}(K)} \right)^{-\frac12} 
	\left| \eta_K \big( {\bf x}_K^{(jq)} \big) 
	\big\langle 
	{\bm \xi}_K^{(j)},  
	{\bf B}_{jq}^{{\rm int}(K)} - {\bf B}_{jq}^{{\rm ext}(K)} 
	\big\rangle
	\right|
	\left(
	{\bf U}_{jq}^{{\rm int}(K)} 
	\cdot {\bf n}^* 
	+ \frac{|{\bf B}^*|^2}{2} 
	\right)
	\Bigg],
	\end{split}
	\end{align*}
	which, along with \eqref{eq:lowerPi3} and $\eta_K ( {\bf x}_K^{(jq)} )\le 0$, further imply 
	\begin{equation}\label{eq:Pi3Pi4}
	\begin{aligned}
	\Pi_3+\Pi_4 &\ge 
	\frac{\Delta t_n}{|K|} \sum_{j=1}^{N_K} \sum_{q=1}^Q 
	\bigg[
	\big|{\mathscr E}_K^j\big| \omega_q  \eta_K \big( {\bf x}_K^{(jq)} \big) \left(\rho_{jq}^{{\rm int}(K)} \right)^{-\frac12} 
	\\
	&\quad \times
	\left| 
	\big\langle 
	{\bm \xi}_K^{(j)},  
	{\bf B}_{jq}^{{\rm int}(K)} - {\bf B}_{jq}^{{\rm ext}(K)} 
	\big\rangle
	\right| \left(
	{\bf U}_{jq}^{{\rm int}(K)} 
	\cdot {\bf n}^* 
	+ \frac{|{\bf B}^*|^2}{2} 
	\right)\bigg].
	\end{aligned}
	\end{equation}
	Combining the lower bounds in \eqref{eq:proofPi2>0}, 
	\eqref{eq:Pi0}, \eqref{eq:Pi3Pi4}, 
	with \eqref{eq:decompEQ}, we obtain 
	\begin{align*}
	\begin{split}
	\bar {\bf U}_K^{n+1} \cdot {\bf n}^* 
	+ \frac{|{\bf B}^*|^2}{2}   & \ge \sum_{j=1}^{N_K} \sum_{q=1}^Q \varpi_{jq} 
	\left( 
	{\bf U}_{jq}^{ {\rm int}(K) } \cdot {\bf n}^* 
	+ \frac{|{\bf B}^*|^2}{2} \right) 
	\\
	& \quad + \frac{\Delta t_n}{|K|} 
	\sum_{j=1}^{N_K} \sum_{q=1}^Q |{\mathscr E}_K^j|
	\omega_q \Bigg[ \left( \sigma_{jq}^{K,-}- \widehat \alpha_{jq}^{{\rm int}(K)}  \right)  +   \eta_K \big( {\bf x}_K^{(jq)}  \big)
	\\
	& \quad  \times 
	\frac{ \left| 
		\big\langle 
		{\bm \xi}_K^{(j)},  
		{\bf B}_{jq}^{{\rm int}(K)} - {\bf B}_{jq}^{{\rm ext}(K)} 
		\big\rangle
		\right| }{ \sqrt{\rho_{jq}^{{\rm int}(K)} }} \left(
	{\bf U}_{jq}^{{\rm int}(K)} 
	\cdot {\bf n}^* 
	+ \frac{|{\bf B}^*|^2}{2} 
	\right)\Bigg]
	\end{split}
	\\
	& = \sum_{j=1}^{N_K} \sum_{q=1}^Q 
	\left( \varpi_{jq} -  \frac{\Delta t_n}{|K|} 
	\big| {\mathscr E}_K^j \big| \omega_q 
	\alpha_{jq}^{K}   \right)
	\left(
	{\bf U}_{jq}^{{\rm int}(K)} 
	\cdot {\bf n}^* 
	+ \frac{|{\bf B}^*|^2}{2} 
	\right) > 0,
	\end{align*}
	where the CFL condition \eqref{eq:CFL:2DMHD} 
	and ${\bf U}_{jq}^{{\rm int}(K)} \in { \mathcal G}={\mathcal G}_*$ have been used in the last inequality. 
	Therefore, we have 
	$$\bar {\bf U}_K^{n+1} \cdot {\bf n}^* 
	+ \frac{|{\bf B}^*|^2}{2} > 0, \quad \forall {\bf v}^*,{\bf B}^* \in \mathbb R^3,$$ 
	which, along with $\bar\rho_{K}^{n+1}>0$, imply $\bar{\bf U}_{K}^{n+1} \in {\mathcal G}$ by Lemma \ref{theo:eqDefG}. 
	
	The proof is completed. 
\end{proof}

Let us further understand the above PP DG schemes and the result in Theorem \ref{thm:PP:2DMHD} 
on two special meshes.

{\bf Example 1.} Assume that the mesh is rectangular with cells $\{[x_{i-\frac12},x_{i+\frac12}]\times [y_{\ell-\frac12},y_{\ell+\frac12}] \}$ 
and spatial step-sizes $\Delta x_i:=x_{i+\frac12}-x_{i-\frac12}$ and $\Delta y_\ell:=y_{\ell+\frac12}-y_{\ell-\frac12}$ in $x$- and $y$-directions respectively, where $(x,y)$ denotes the 2D spatial coordinate variables. 
Let ${\mathbb S}_i^x=\{ x_i^{(q)}  \}_{q=1}^Q$
and ${\mathbb S}_\ell^y=\{ y_\ell^{(q)}  \}_{q=1}^Q$ 
denote the $Q$-point Gauss quadrature points in the intervals 
$[x_{i-\frac12},x_{i+\frac12}]$ and 
$[y_{\ell-\frac12},y_{\ell+\frac12}]$ respectively. 
For the cell $K=[x_{i-\frac12},x_{i+\frac12}]\times [y_{\ell-\frac12},y_{\ell+\frac12}]$, 
the point set $\mathbb S_K$ in \eqref{eq:FVDGsuff} is given by (cf.~\cite{zhang2010,zhang2010b})
\begin{equation}\label{eq:RectS}
{\mathbb S}_K = \big(  \widehat{\mathbb S}_i^x \otimes 
{\mathbb S}_\ell^y \big) \cup \big(  {\mathbb S}_i^x \otimes 
\widehat{\mathbb S}_\ell^y \big),
\end{equation}
where 
$\widehat{\mathbb S}_i^x=\{ \widehat x_i^{(\mu)}  \}_{\mu=1}^{\tt L}$
and $\widehat {\mathbb S}_\ell^y=\{\widehat y_\ell^{(\mu)}  \}_{\mu=1}^{\tt L}$ 
denote the $\tt L$-point Gauss--Lobatto quadrature points in the intervals 
$[x_{i-\frac12},x_{i+\frac12}]$ and 
$[y_{\ell-\frac12},y_{\ell+\frac12}]$ respectively, 
where ${\tt L} \ge \frac{k+3}2$ such that the associated quadrature has algebraic accuracy of at least degree $k$.  
See Fig.\ \ref{fig:GS} for an illustration of $\mathbb S_K$ for $k=2$. 
With $\mathbb S_K$ in \eqref{eq:RectS}, a special quadrature (cf.~\cite{zhang2010,zhang2010b}) satisfying \eqref{eq:decomposition} can be constructed:
\begin{equation} \label{eq:U2Dsplit}
\begin{split}
\frac{1}{|K|}\int_K u({\bf x}) d {\bf x}
&= \frac{ \Delta x_i \widehat \omega_1}{ \Delta x_i + \Delta y_\ell } \sum \limits_{q = 1}^{ Q}  \omega_q \left(  
u\big( x_i^{(q)},y_{\ell-\frac12} \big) 
+ u\big( x_i^{(q)},y_{\ell+\frac12} \big) 
\right)
\\
&
+ \frac{ \Delta y_\ell \widehat \omega_1}{ \Delta x_i + \Delta y_\ell } \sum \limits_{q = 1}^{ Q}  \omega_q \left( u \big(x_{i-\frac12},y_\ell^{(q)}\big) +
u\big(x_{i+\frac12},y_\ell^{(q)}\big) \right)
\\
& + \frac{\Delta x_i}{\Delta x_i + \Delta y_\ell} \sum \limits_{\mu = 2}^{{\tt L}-1} \sum \limits_{q = 1}^{Q}  \widehat \omega_\mu \omega_q  u\big(  x_i^{(q)},\widehat y_\ell^{(\mu)} \big)
\\
& + \frac{\Delta y_\ell}{\Delta x_i + \Delta y_\ell}  \sum \limits_{\mu = 2}^{{\tt L}-1} \sum \limits_{q = 1}^{ Q}  \widehat \omega_\mu \omega_q  u\big(\widehat x_i^{(\mu)},y_\ell^{(q)}\big),\quad \forall u \in {\mathbb P}^k(K),
\end{split}
\end{equation}
where $\{\widehat w_\mu\}_{\mu=1}^{\tt L}$ are the weights of the $\tt L$-point Gauss--Lobatto quadrature.
If labeling the bottom, right, top and left adjacent cells of $K$ as  
$K_1$, $K_2$, $K_3$ and $K_4$, respectively, as illustrated in Fig.~\ref{fig:ILLMesh}, then \eqref{eq:U2Dsplit} implies 
\begin{equation*}
\varpi_{jq} = \frac{ \Delta x_i \widehat \omega_1 \omega_q}{ \Delta x_i + \Delta y_\ell }, \quad j=1,3; \qquad 
\varpi_{jq} = \frac{ \Delta y_\ell \widehat \omega_1 \omega_q}{ \Delta x_i + \Delta y_\ell },\quad j=2,4.
\end{equation*}
Then according to Theorem \ref{thm:PP:2DMHD}, the CFL condition \eqref{eq:CFL:2DMHD} for our PP DG schemes on rectangular meshes is 
\begin{equation*}
\Delta t_n
\left( \frac{1}{\Delta x_i} 
+ \frac{1}{\Delta y_\ell} \right) 
\alpha_{jq}^{K} < \widehat \omega_1= \frac{1}{{\tt L}({\tt L}-1)},\quad \forall
K\in{\mathcal T}_h,~1\le j \le 4,~1\le q \le Q.
\end{equation*}

{\bf Example 2.} Assume that the mesh is triangular. A special quadrature satisfying \eqref{eq:decomposition} was introduced in 
\cite{zhang2012maximum}, with the point set $\mathbb S_K$, denoted by local barycentric coordinates, as 
\begin{equation*}
\begin{aligned}
& \Bigg\{
\bigg( \frac12 + \zeta_q, (\frac12 + \widehat\zeta_\mu)(\frac12 - \zeta_q),
(\frac12-\widehat \zeta_{\mu})(\frac12 - \zeta_{q})   \bigg), \\
& \quad
\bigg( (\frac12 - \widehat\zeta_\mu)(\frac12 - \zeta_q), \frac12 + \zeta_q,
(\frac12+\widehat \zeta_{\mu})(\frac12 - \zeta_{q})   \bigg),
\\
&
\quad
\bigg( (\frac12 + \widehat\zeta_\mu)(\frac12 - \zeta_q),
(\frac12-\widehat \zeta_{\mu})(\frac12 - \zeta_{q}) , \frac12 + \zeta_q  \bigg),
1\le q \le Q, 1\le \mu \le {\tt L}
\Bigg\},
\end{aligned}
\end{equation*}
where $\{\zeta_q\}_{q=1}^Q$ and $\{\widehat \zeta_{\mu}\}_{\mu=1}^{\tt L}$ are the Gauss quadrature points and the Gauss--Lobatto quadrature points on $\big[-\frac12,\frac12 \big]$ respectively, and ${\tt L} \ge \frac{k+3}2$.
For this quadrature,  \eqref{eq:decomposition} becomes  (cf.~\cite{zhang2012maximum}) 
\begin{equation}\label{eq:tridecop}
\frac{1}{|K|} \int_K u({\bf x}) d {\bf x}
= \frac{2}{3} \widehat \omega_1 \sum_{j=1}^3  \sum_{ q =1 }^Q \omega_q u ({\bf x}_{K}^{(jq)}) 
+ \sum_{ q=1 }^{\widetilde Q} \widetilde \varpi_q   
u ( \widetilde {\bf x}_K^{(q)} ), \quad \forall u \in {\mathbb P}^k(K),
\end{equation}
where $\widetilde Q= 3 ( {\tt L} -2 ) Q$. The specific expressions of the weights $\widetilde \varpi_q$ at quadrature points in the interior of $K$ are omitted here. Eq.~\eqref{eq:tridecop} implies 
$$
\varpi_{jq} = \frac{2}{3} \widehat \omega_1 \omega_q,\quad 1\le j \le 3.
$$
Then, according to Theorem \ref{thm:PP:2DMHD}, the CFL condition \eqref{eq:CFL:2DMHD} for our PP DG schemes on triangular meshes is 
\begin{equation*}
\Delta t_n \frac{| {\mathscr E}_K^j |}{|K|} 
\alpha_{jq}^{K} < \frac{2}{3} \widehat \omega_1= \frac{2}{3 {\tt L}( {\tt L}-1)},\quad \forall
K\in{\mathcal T}_h,~1\le j \le 3,~1\le q \le Q.
\end{equation*}

\begin{remark}
	Our PP schemes have two features: the locally divergence-free spatial discretization and the penalty term properly discretized from the Godunov--Powell source term. 
	The former feature ensures zero divergence of numerical magnetic field within each cell, while the latter 
	controls the divergence error across the cell interfaces.  
	The proof of Theorem \ref{thm:PP:2DMHD} clearly shows 
	that, thanks to these two features, the PP property of the proposed schemes is obtained without requiring the DDF condition, 
	which is needed for the PP property of the conservative schemes without the penalty term, see the following theorem. 	   
\end{remark}

The scheme \eqref{eq:PP2DMHD:cellaverage} without the penalty term becomes 
\begin{equation}\label{eq:2DMHD:cellaverage:con}
\bar {\bf U}_{K}^{n+1}  = \bar {\bf U}_{K}^{n}
- \frac{\Delta t_n}{|K|} 
\sum_{j=1}^{N_K} \sum_{q=1}^Q |{\mathscr E}_K^j|
\omega_q
\hat{\bf F} \left( {\bf U}_{jq}^{{\rm int}(K)}, {\bf U}_{jq}^{{\rm ext}(K)} ; {\bm \xi}^{(j)}_{K} \right),
\end{equation} 
which is a conservative finite volume scheme or the scheme satisfied by the cell averages of a DG 
method for the conservative MHD system \eqref{eq:MHD}. 
As mentioned before, even the first-order version ($k=0$) of the scheme \eqref{eq:2DMHD:cellaverage:con}, i.e., the scheme \ref{eq:1stscheme}, is generally not PP unless a DDF condition is satisfied by the numerical magnetic field. 
The DDF condition can also be generalized to high-order schemes ($k\ge 1$), as shown in Theorem \ref{thm:2DPPcon}.

\begin{theorem}\label{thm:2DPPcon}
	Let the wave speeds in the HLL flux satisfy \eqref{eq:HLLsigma}.  	
	If the polynomial vectors $\{\widetilde{\bf U}_{K}^n({\bf x})\}$ satisfy 
	the condition \eqref{eq:FVDGsuff}, then 
	under the CFL-type condition
	\begin{equation}\label{eq:CFL:2DMHDcon}
	\Delta t_n \frac{| {\mathscr E}_K^j |}{|K|} 
	\left( \widehat \alpha_{jq}^{{\rm int}(K)}  - \sigma_{jq}^{K,-} \right) < \frac{ \varpi_{jq}}{\omega_q},\qquad \forall
	K\in{\mathcal T}_h,~1\le j \le N_K,~1\le q \le Q,
	\end{equation}
	the solution $\bar{\bf U}_{K}^{n+1}$ of the scheme \eqref{eq:2DMHD:cellaverage:con} satisfies 
	that  
	$\bar \rho_{K}^{n+1} >0$ 
	and  
	\begin{equation}\label{eq:lowercalE}
	{\mathcal E}( \bar{\bf U}_{K}^{n+1} ) > -
	\Delta t_n 
	\left( \bar{ \rho}_K^{n+1} \right)^{-1} ( \bar{ \bf  m}_K^{n+1} \cdot \bar{ \bf  B }_K^{n+1} ) \left( \mbox{\rm div} _{K} \widetilde {\bf B}^n_h \right),
	\end{equation}
	where 
	$\mbox{\rm div} _{K} \widetilde {\bf B}^n_h$ is the discrete divergence defined by 
	\begin{equation}\label{eq:DefDisDivBH}
	\mbox{\rm div} _{K} \widetilde {\bf B}^n_h := \frac{1}{|K|}
	\sum_{j=1}^{N_K} \sum_{q=1}^Q \big| {\mathscr E}_{K}^{j} \big| \omega_q
	\left\langle {\bm \xi}_{K}^{(j)},  
	\frac{ \sigma_{jq}^{K,+}  {\bf B}_{jq}^{{\rm int}(K)}  - \sigma_{jq}^{K,-} {\bf B}_{jq}^{ {\rm ext}(K) }  }{
		\sigma_{jq}^{K,+} - \sigma_{jq}^{K,-} 
	} \right\rangle.
	\end{equation} 
	Furthermore, if the magnetic field $\widetilde{\bf B}_h^n({\bf x})$ satisfies the DDF condition
	\begin{equation}\label{eq:DDFhigh}
	\mbox{\rm div} _{K} \widetilde {\bf B}^n_h=0,
	\end{equation}
	then 
	$\bar {\bf U}_K^{n+1}\in {\mathcal G}$.
\end{theorem}

\begin{proof}
	Since the first component of ${\bf S}({\bf U})$ 
	is zero, the discrete equations for $\rho$ 
	in the two schemes \eqref{eq:PP2DMHD:cellaverage} 
	and \eqref{eq:2DMHD:cellaverage:con} are the same. Hence $\bar \rho_{K}^{n+1} >0$ directly follows from the proof of Theorem \ref{thm:2DPPcon}. 
	
	Similar to the proof of Theorem \ref{thm:2DPPcon}, 
	it can be derived for any ${\bf v}^*,{\bf B}^* \in \mathbb R^3$ that 
	\begin{equation}\label{eq:decompEQ:con}
	\bar {\bf U}_K^{n+1} \cdot {\bf n}^* 
	+ \frac{|{\bf B}^*|^2}{2} 
	= \Pi_0 + \Pi_1 + \Pi_2 + \Pi_3,
	\end{equation}	
	where $\Pi_2$ 
	is defined \eqref{eq:proofPi2}, 
	and $\Pi_0$, $\Pi_1$ and $\Pi_3$ are defined in 
	\eqref{eq:Pi00}--\eqref{eq:Pi33}, respectively. 
	Combining the estimates \eqref{eq:proofPi2}, 
	\eqref{eq:Pi0} and \eqref{eq:lowerPi3}, 
	gives   
	\begin{equation*}
	\begin{split}
	& \bar {\bf U}_K^{n+1} \cdot {\bf n}^* 
	+ \frac{|{\bf B}^*|^2}{2}   \ge - \Delta t_n ( {\bf v}^* \cdot {\bf B}^* ) \left( \mbox{\rm div} _{K} \widetilde {\bf B}^n_h\right)
	\\
	& \quad  
	+
	\sum_{j=1}^{N_K} \sum_{q=1}^Q 
	\left( \varpi_{jq} -  \frac{\Delta t_n}{|K|} 
	\big| {\mathscr E}_K^j \big| \omega_q 
	\big( \widehat \alpha_{jq}^{{\rm int}(K)}  - \sigma_{jq}^{K,-} \big)   \right)
	\left(
	{\bf U}_{jq}^{{\rm int}(K)} 
	\cdot {\bf n}^* 
	+ \frac{|{\bf B}^*|^2}{2} 
	\right)
	\\
	& \quad > - \Delta t_n ( {\bf v}^* \cdot {\bf B}^* ) \left( \mbox{\rm div} _{K} \widetilde {\bf B}^n_h \right).
	\end{split}
	\end{equation*}
	Taking ${\bf v}^*=\bar{ \bf  m}_K^{n+1}/\bar{ \rho}_K^{n+1}$ 
	and ${\bf B}^*=\bar{ \bf  B }_K^{n+1}$ gives \eqref{eq:lowercalE}.
	
	Under the DDF condition \eqref{eq:DDFhigh}, the estimate $\eqref{eq:lowercalE}$ becomes 
	${\mathcal E}( \bar {\bf U}_K^{n+1} )>0$, 
	which along with $\bar {\rho}_K^{n+1}>0$ imply 
	$\bar {\bf U}_K^{n+1}\in {\mathcal G}$. 
\end{proof}

\begin{remark}
	In practice, it is not easy to meet the DDF condition 
	\eqref{eq:DDFhigh}, because it depends on the limiting values of the magnetic field evaluated at the 
	adjacent cells of $K$. The locally divergence-free property cannot ensure the DDF condition 
	\eqref{eq:DDFhigh}. 
	If $ {\bf B}_h^n({\bf x})$ is globally divergence-free, that is,  it is locally divergence-free in each cell with normal magnetic component continuous across the cell interfaces, then by Green's theorem, the DDF condition
	$\mbox{\rm div} _{K}  {\bf B}^n_h=0$ is satisfied naturally. 
	However, the usual PP limiting technique (cf.~\cite{zhang2010b,cheng}) with local scaling can destroy the globally divergence-free property of $ {\bf B}_h^n({\bf x})$. Therefore, it is nontrivial and still open to devise a limiting procedure that can enforce 
	the two conditions \eqref{eq:FVDGsuff} and \eqref{eq:DDFhigh} 
	at the same time. 
\end{remark}

\begin{remark}\label{eq:remark3}
	We can split the discrete divergence into two parts:
	\begin{equation*}
	\begin{split}
	\mbox{\rm div} _{K} \widetilde {\bf B}^n_h &= \frac{1}{|K|}
	\sum_{j=1}^{N_K} \sum_{q=1}^Q \big| {\mathscr E}_{K}^{j} \big| \omega_q
	\left\langle {\bm \xi}_{K}^{(j)},  
	{\bf B}_{jq}^{{\rm int}(K)}  \right\rangle
	\\
	& \quad +\frac{1}{|K|}
	\sum_{j=1}^{N_K} \sum_{q=1}^Q \big| {\mathscr E}_{K}^{j} \big| \omega_q \eta_K( {\bf x}_K^{(jq)} )
	\left\langle {\bm \xi}_{K}^{(j)},  
	{\bf B}_{jq}^{{\rm int}(K)}  - {\bf B}_{jq}^{ {\rm ext}(K) }  \right\rangle.
	\end{split}
	\end{equation*} 	
	The first part becomes zero if the locally divergence-free discretization is used, while the second part, 
	which involves the divergence error across the cell interfaces, can be handled by including our properly discretized Godunov--Powell source term. As we have seen in the above analysis, a coupling of these two divergence-controlling techniques 
		is very important in our PP DG methods, because they exactly 
		contribute the discrete divergence terms 
		which are absent in a standard multidimensional DG scheme  \eqref{eq:2DMHD:cellaverage:con} but 
		crucial for ensuring the PP property. 
\end{remark}

\begin{remark}
	In the above discussions, we restrict ourselves to the first-order accurate forward Euler time discretization. One can also use SSP high-order accurate time discretizations (cf.~\cite{Gottlieb2009})
	to solve the ODE system $\frac{d}{dt}{\bf U}_h = {\bf L} ({\bf U}_h)$. 
	For instance, the explicit third-order SSP Runge-Kutta method reads 
	\begin{equation}\label{eq:SSP3}
	\begin{split}
	& {\bf U}_h^{*} = \widetilde {\bf U}_h^{n}+\Delta t_n {\bf L} ( \widetilde {\bf U}_h^{n} ),
	\\
	& {\bf U}_h^{**} = \frac34 \widetilde {\bf U}_h^{n}
	+ \frac14 \Big(  \widetilde {\bf U}_h^{*}+\Delta t_n {\bf L} ( \widetilde {\bf U}_h^{*} ) \Big),
	\\
	&{\bf U}_h^{n+1} = \frac13 \widetilde {\bf U}_h^{n}
	+ \frac23 \Big(  \widetilde {\bf U}_h^{**}+\Delta t_n {\bf L} ( \widetilde {\bf U}_h^{**} ) \Big),
	\end{split}
	\end{equation}
	where the approximate solution functions with ``$\sim$'' at above denote the PP limited solution.  
	Based on the the convexity of $\mathcal G$ and that an SSP method is a convex combination of the forward Euler  method, 
	the PP property of the full high-order scheme also holds.
\end{remark}











\section{Numerical tests}\label{app:num}
In this section, 
we present some numerical results of the proposed PP DG schemes 
for several extreme MHD problems involving low density, low pressure, low plasma-beta $\beta:=2p/|{\bf B}|^2$, and/or 
strong discontinuity, 
to verify the provenly PP property and to demonstrate the effectiveness of our HLL flux and the proposed discretization 
of the Godunov--Powell source term. 
The tests below are conducted 
on uniform 1D meshes or 2D rectangular meshes, 
while the implementation of 
our PP schemes on unstructured triangular meshes 
is ongoing and will be reported in a separate paper. 
Without loss of generality, we focus on the proposed PP third-order (${\mathbb P}^2$) DG methods with the SSP Runge-Kutta time discretization \eqref{eq:SSP3}. 
Although our analysis has suggested a CFL condition for guaranteeing the provably PP property, we observe that our PP DG methods still work robustly and maintain the desired positivity 
with suitably larger time step-size in the tested cases. 
Unless otherwise stated, 
the following computations are restricted to the ideal EOS $p=(\gamma-1)\rho e$ with $\gamma=1.4$, and the CFL number is set as $0.15$.  
The HLL flux is always used with the local wave speeds given by \eqref{eq:choiceHLL}. 

\subsection{Smooth problems}
A 1D and a 2D smooth problems are respectively
solved on the uniform meshes of $M^d$ cells 
to test the accuracy of the PP third-order DG methods. 
The 1D problem is similar to the one simulated in \cite{zhang2010b} for testing the PP
DG scheme for the Euler equations, and has the exact solution  
$$
(\rho,{\bf v},p,{\bf B})(x,t)=( 1+0.99\sin(x-t),~1,~,0,~,0,~1,~0.1,~0,~0 ),\quad x\in[0,2\pi],~t\ge 0,
$$
which describes a sine wave propagating with low density. The 2D problem is the vortex problem with the same setup as in \cite{Christlieb} 
 and has a extremely low pressure (about $5.3 \times 10^{-12}$) in the vortex center; the adiabatic index $\gamma=\frac53$; the computational domain
 is $[−10, 10]^2$ with periodic boundary conditions.
Fig.~\ref{fig:smooth} displays the numerical errors
obtained by our third-order DG scheme at different mesh resolutions. It is seen that the expected convergence order is achieved. 

	\begin{figure}[htb]
	\centering
	{\includegraphics[width=0.49\textwidth]{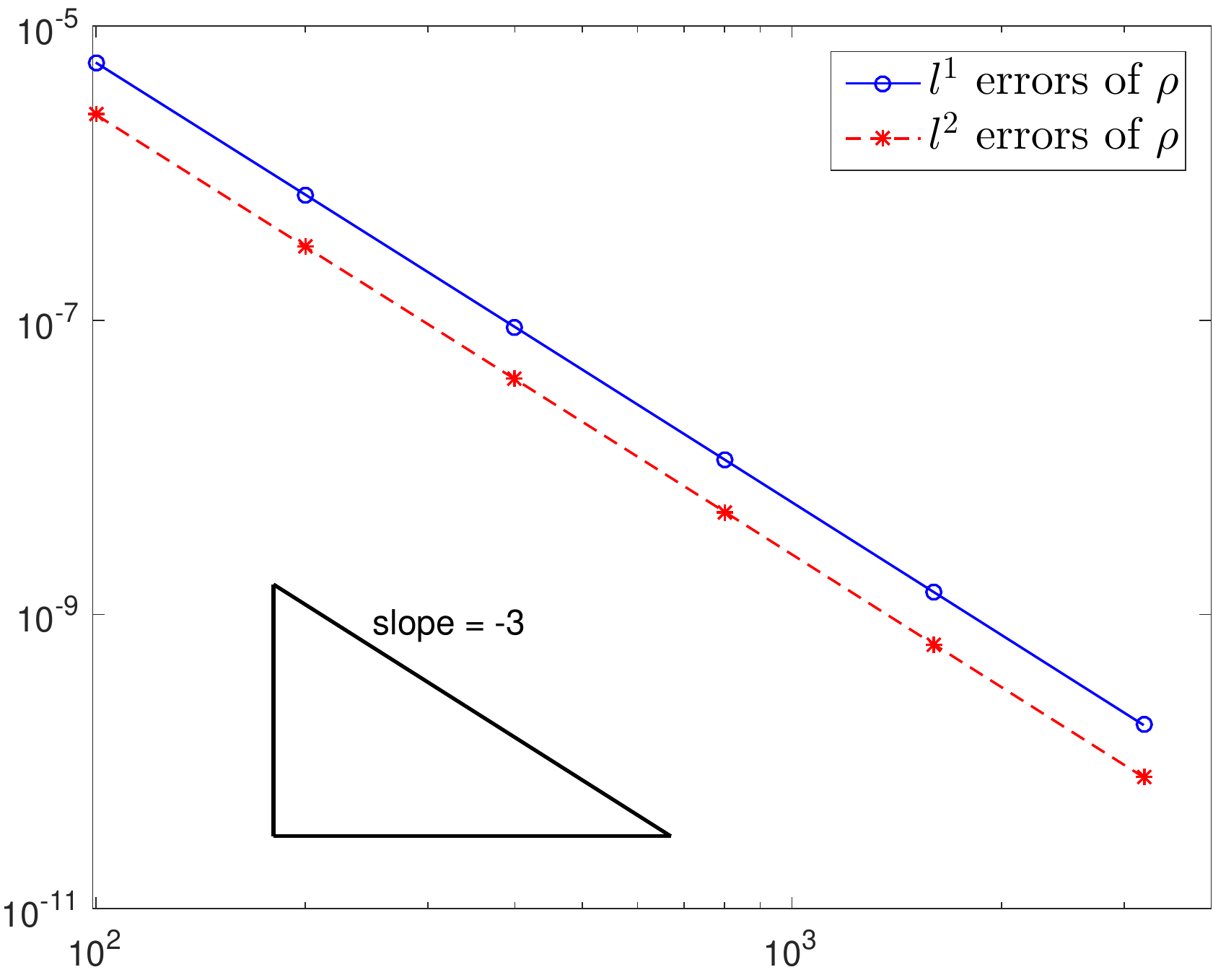}}
	{\includegraphics[width=0.49\textwidth]{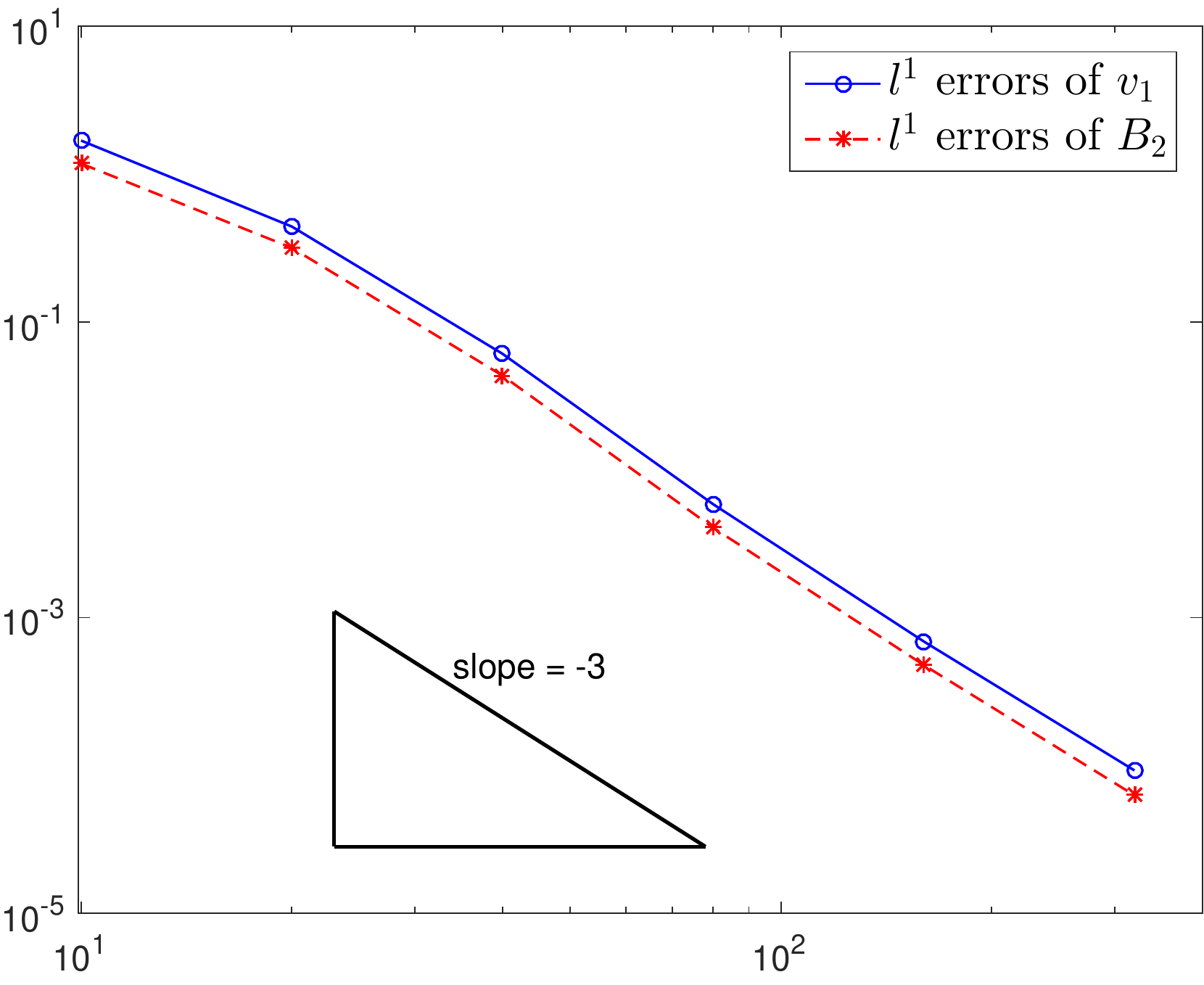}}
	\caption{\small Numerical errors obtained by the PP third-order DG scheme at different mesh resolutions
		with $M^d$ cells. Left: the 1D smooth problem at $t=0.1$; right: the 2D
		smooth problem at $t = 0.05$. The horizontal axis denotes the value of $M$.} 
	\label{fig:smooth}
\end{figure}

Next, we simulate several MHD problems involving discontinuities. 
 Before the PP limiting procedure, the WENO limiter \cite{Qiu2005} is also implemented with the aid of the local characteristic decomposition,  
to enhance the numerical stability of high-order DG schemes in resolving the strong discontinuities and their interactions. 
The 2D WENO limiter is combined with the locally divergence-free reconstruction approach in \cite{ZhaoTang2017}. 
The WENO limiter is only employed adaptively  
in the ``trouble'' cells detected by the indicator of \cite{Krivodonova}.

\subsection{Riemann problems}

	Two 1D Riemann problems are solved. 
	The first is a 1D vacuum shock tube
	problem (cf.~\cite{Christlieb}) with the initial data 
	\begin{equation*}
	(\rho,{\bf v},p,{\bf B}) (x,0)
	=\begin{cases}
	(10^{-12},~0,~0,~0,~10^{-12},~0,~0,~0), \quad & x<0,
	\\
	(1,~0,~0,~0,~0.5,~0,~1,~0), \quad & x>0.
	\end{cases}
	\end{equation*}
	It is used to demonstrate that our 1D PP DG scheme can handle extremely low density and pressure. 
	The adiabatic index $\gamma=\frac53$, and the computational domain 
	is set as $[-0.5,0.5]$. 
	Fig.~\ref{fig:1DRP1} shows the density and pressure of the numerical solution on, respectively, the mesh of $200$ cells as well as those of 
	a highly resolved solution with $5000$ cells at time $t=0.1$. 
	One can observe that the solutions of low
	resolution and high resolution are in quite good agreement. 
	We confirm that the low pressure and the low density are both correctly captured by comparing with the results in \cite{Christlieb}. 
	The PP third-order DG code is very robust during the simulation. 
	It is noticed that, if the PP limiter is not used to enforce the condition 
	\eqref{eq:1DDG:con2}, the code breaks down within a few time steps due to unphysical solution. 
	
		\begin{figure}[htb]
		\centering
		\includegraphics[width=0.49\textwidth]{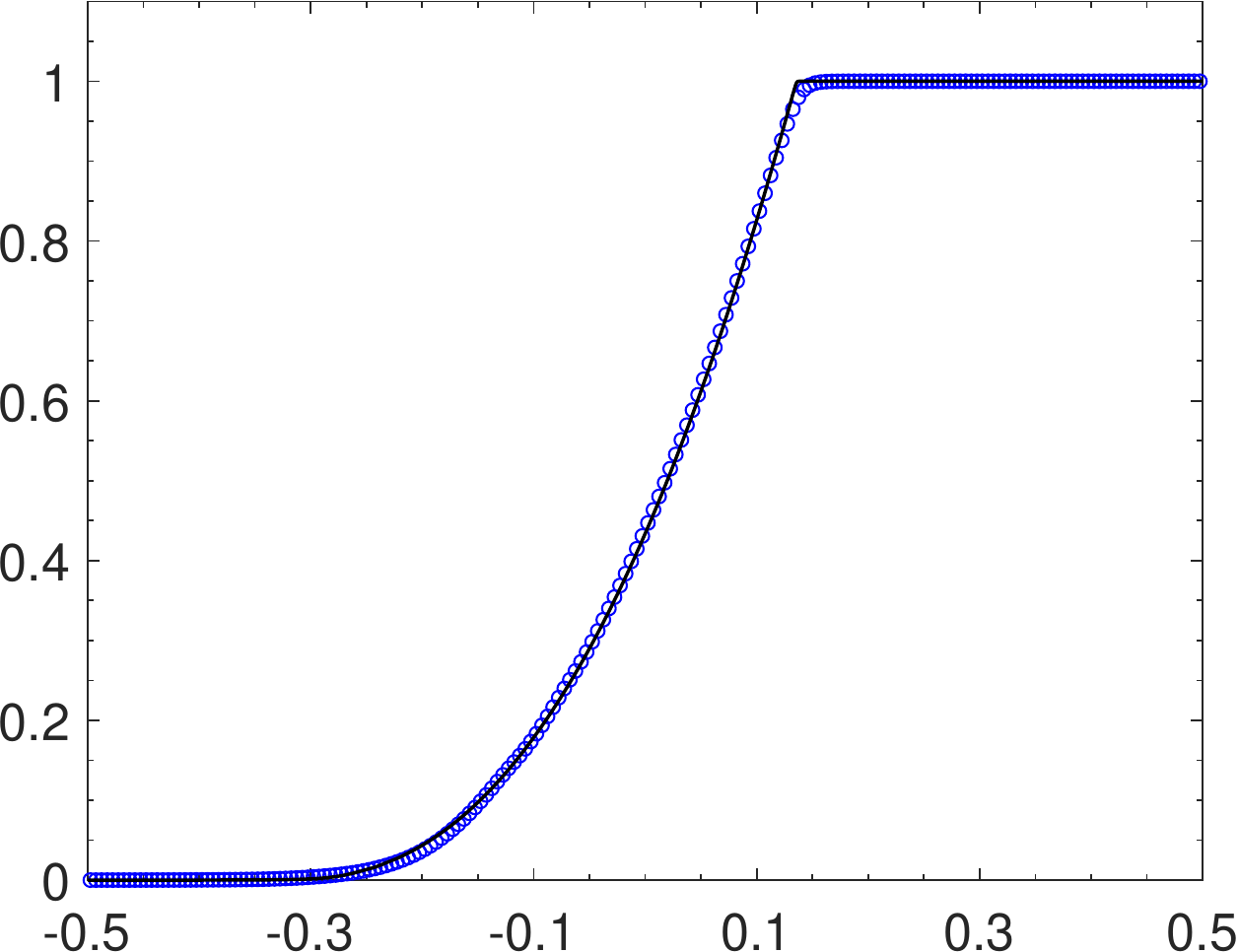}
		\includegraphics[width=0.49\textwidth]{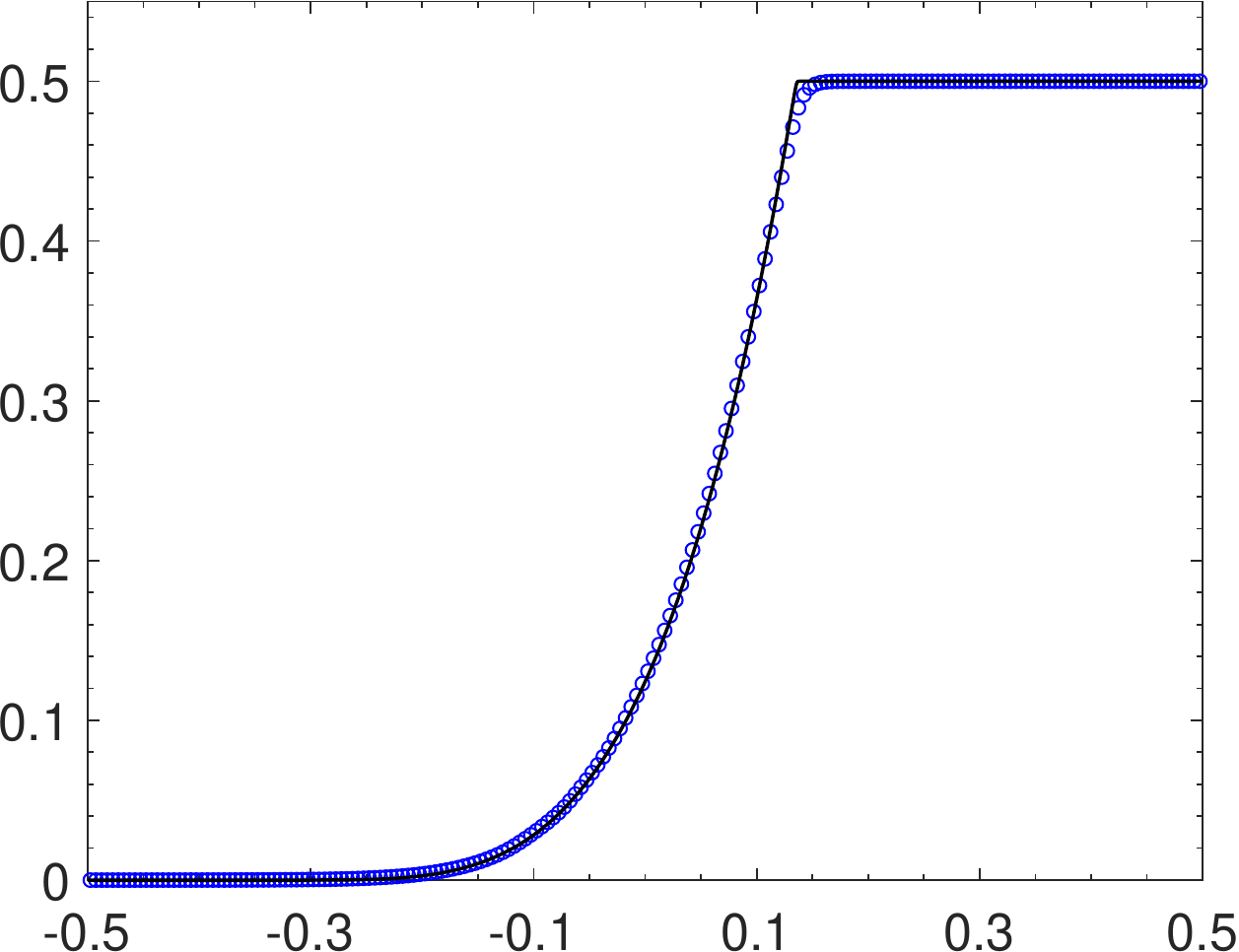}
		\caption{\small
			The density (left) and pressure (right) obtained by the PP third-order DG method 
			on the meshes of $200$ cells (symbols ``$\circ$'') and $5000$ cells (solid lines), respectively. 
		}\label{fig:1DRP1}
	\end{figure}

		\begin{figure}[htb]
		\centering
		\includegraphics[width=0.49\textwidth]{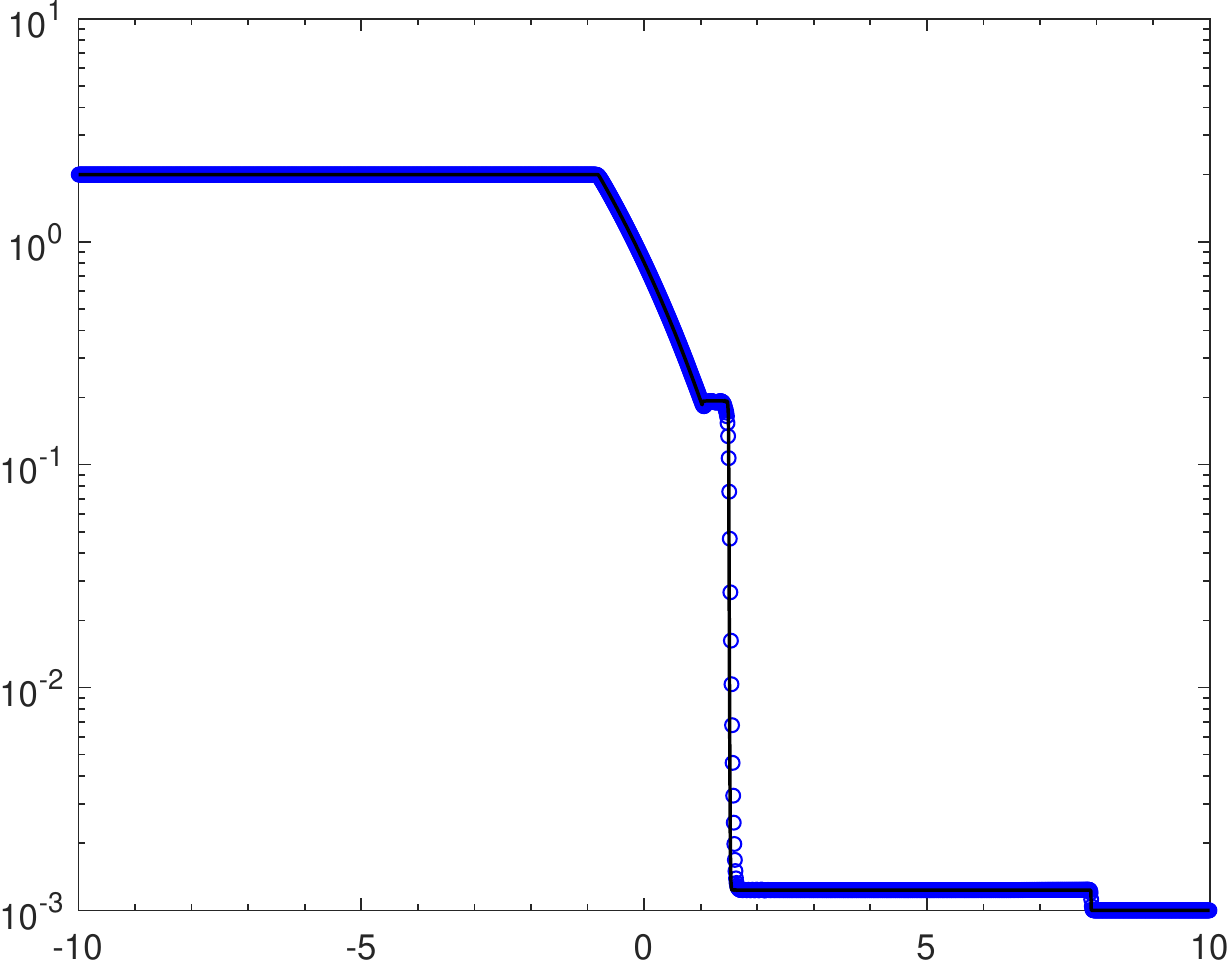}
		\includegraphics[width=0.49\textwidth]{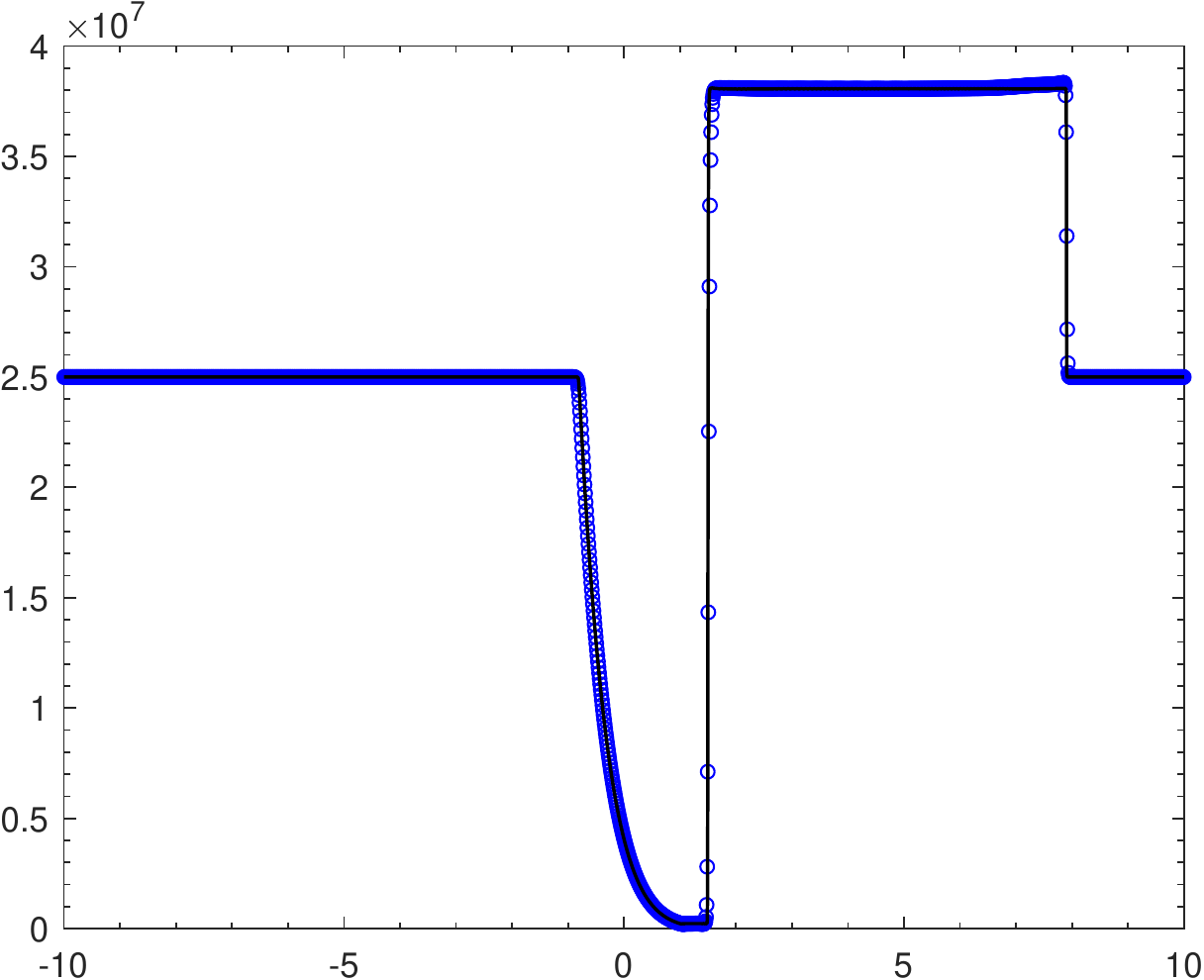}
		\caption{\small
			Numerical results at $t=0.00003$ obtained by the PP third-order DG method with $2000$ cells (symbols ``$\circ$'') and $10000$ cells (solid lines). 
			Left: log plot of density; 
			right: magnetic pressure. 
		}\label{fig:1DRP2}
	\end{figure}

	The second Riemann problem is a variant of the Leblanc problem (cf.~\cite{zhang2010b}) of gas dynamics by adding a strong magnetic field. The initial condition is  
	\begin{equation*}
	(\rho,{\bf v},p,{\bf B}) (x,0)
	=\begin{cases}
	(2,~0,~0,~0,~10^{9},~0,~5000,~5000), \quad & x<0,
	\\
	(0.001,~0,~0,~0,~1,~0,~5000,~5000), \quad & x>0.
	\end{cases}
	\end{equation*}
	The initial pressure has a very large jump, and the plasma-beta at the right state is extremely low 
	($\beta = 4 \times 10^{-8}$), making the successful simulation of 
	this problem a challenge. 
	The computational domain 
	is taken as $[-10,10]$. To fully resolve the wave structure of such a problem,  
	a fine mesh is often required \cite{zhang2010b}. 
	Fig.~\ref{fig:1DRP2} displays the numerical results  
	at $t=0.00003$ obtained by the PP third-order DG scheme with 
	$2000$ cells and $10000$ cells, respectively. It is observed  
	that the strong discontinuities are well resolved, and 
	the low
	resolution and high resolution are highly in agreement. 
	Fig.~\ref{fig:1DRP2b} gives a comparison of the numerical solutions resolved 
	by using the proposed HLL flux and the global LF flux of \cite{Wu2017a}, respectively. 
	As expected, the PP DG method with the HLL flux exhibits better resolution. 
	In this extreme test, it is also necessary to enforce the 
	condition \eqref{eq:1DDG:con2} by the PP limiting procedure, otherwise negative pressure will appear in the cell averages of the DG solution.

		\begin{figure}[htb]
		\centering
		\includegraphics[width=0.49\textwidth]{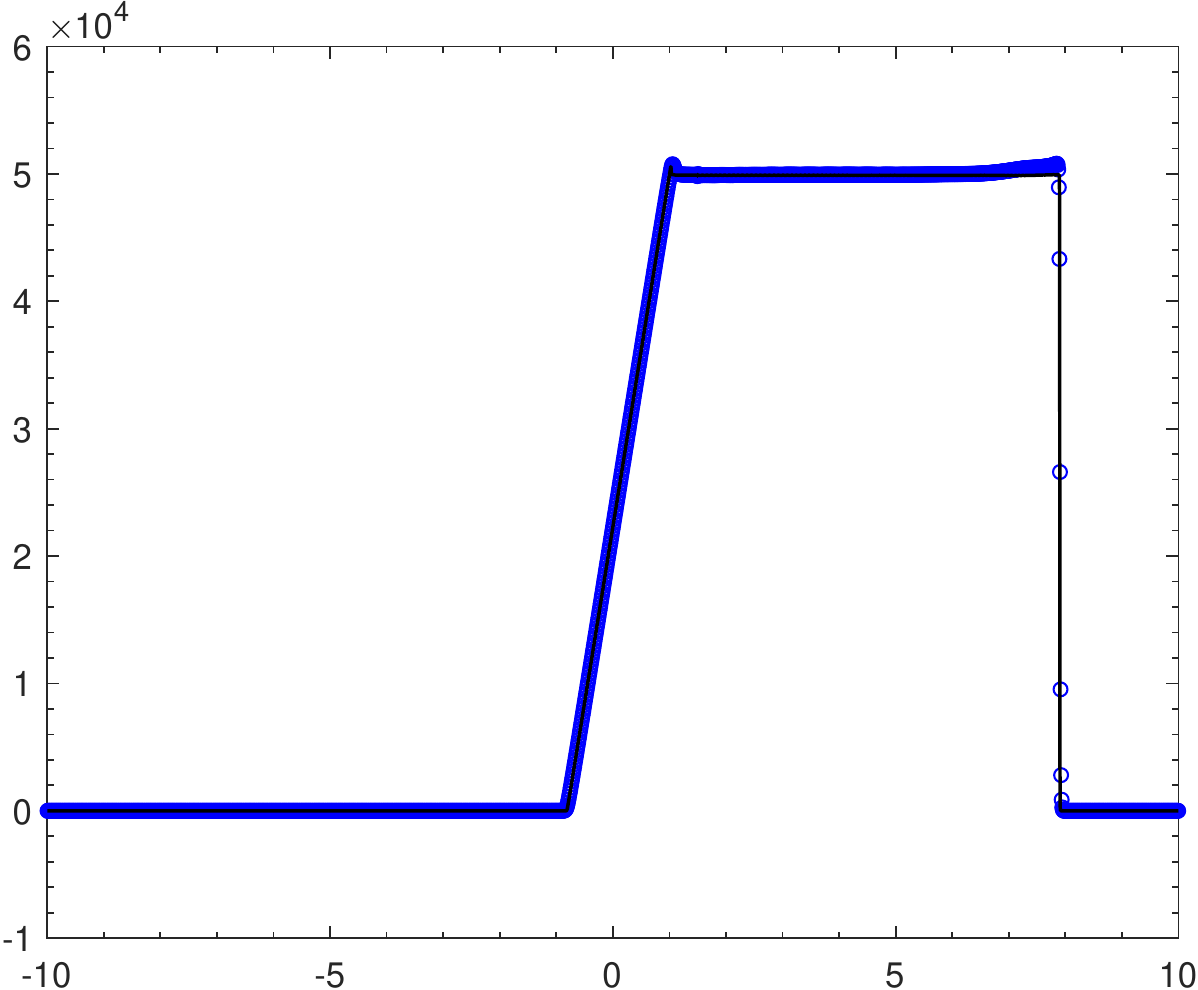}
		\includegraphics[width=0.49\textwidth]{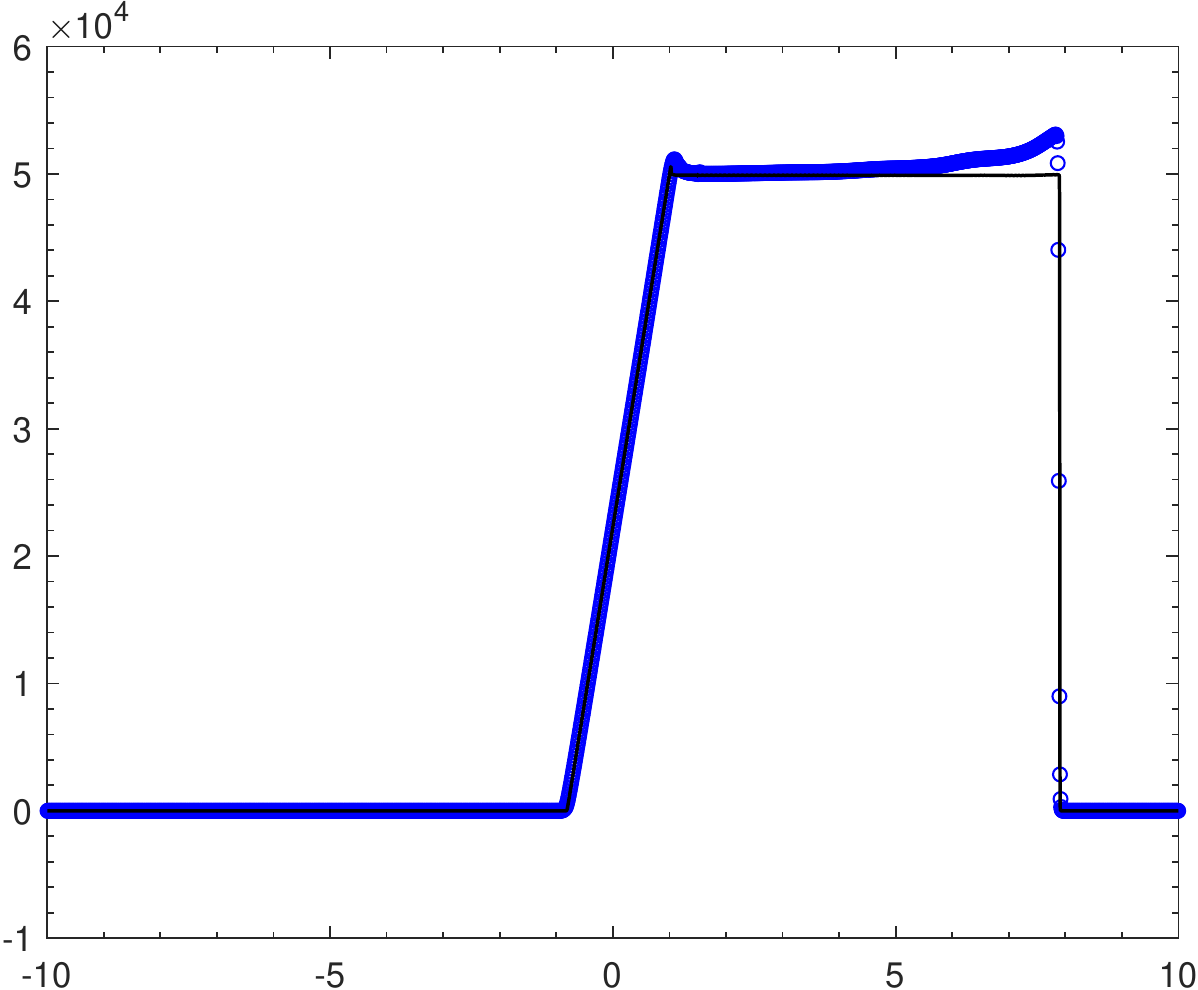}
		\caption{\small
			Same as Fig.~\ref{fig:1DRP2} except for the velocity $v_1$ obtained by using the proposed HLL flux (left) 
			and the global LF flux (right). 
		}\label{fig:1DRP2b}
	\end{figure}

\subsection{Blast problem}

This test was first introduced by Balsara and Spicer  \cite{BalsaraSpicer1999}, and has become 
 a benchmark for testing 2D MHD codes. 
If the low gas pressure, strong magnetic field or
small plasma-beta is involved, then simulating such MHD blast
problems can be very challenging. 
Therefore, it is often used to 
examine the robustness of MHD schemes; see e.g., \cite{cheng,Christlieb}. 

\begin{figure}[htb]
	\centering
	{\includegraphics[width=0.49\textwidth]{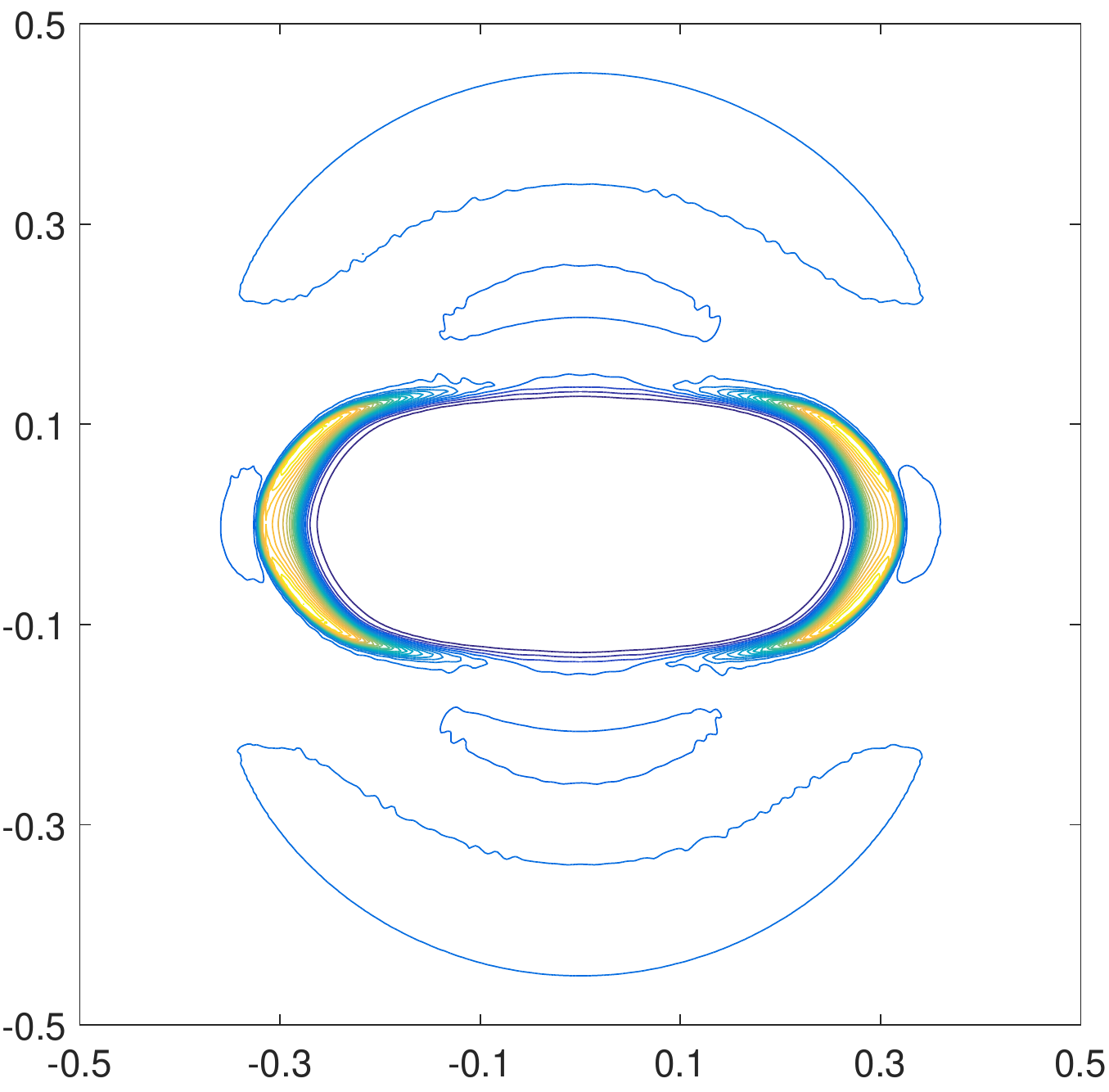}}
	{\includegraphics[width=0.49\textwidth]{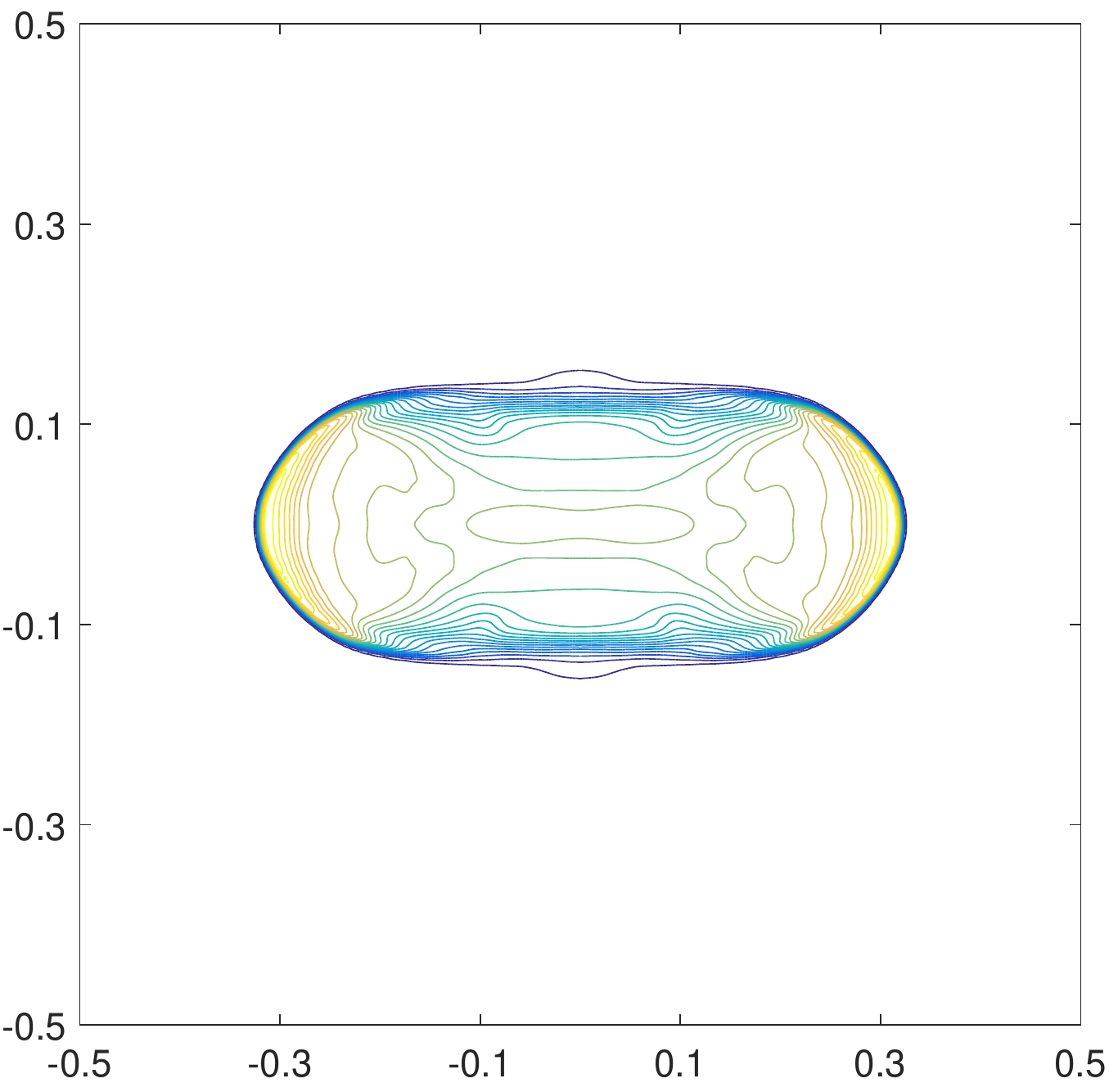}}
	{\includegraphics[width=0.49\textwidth]{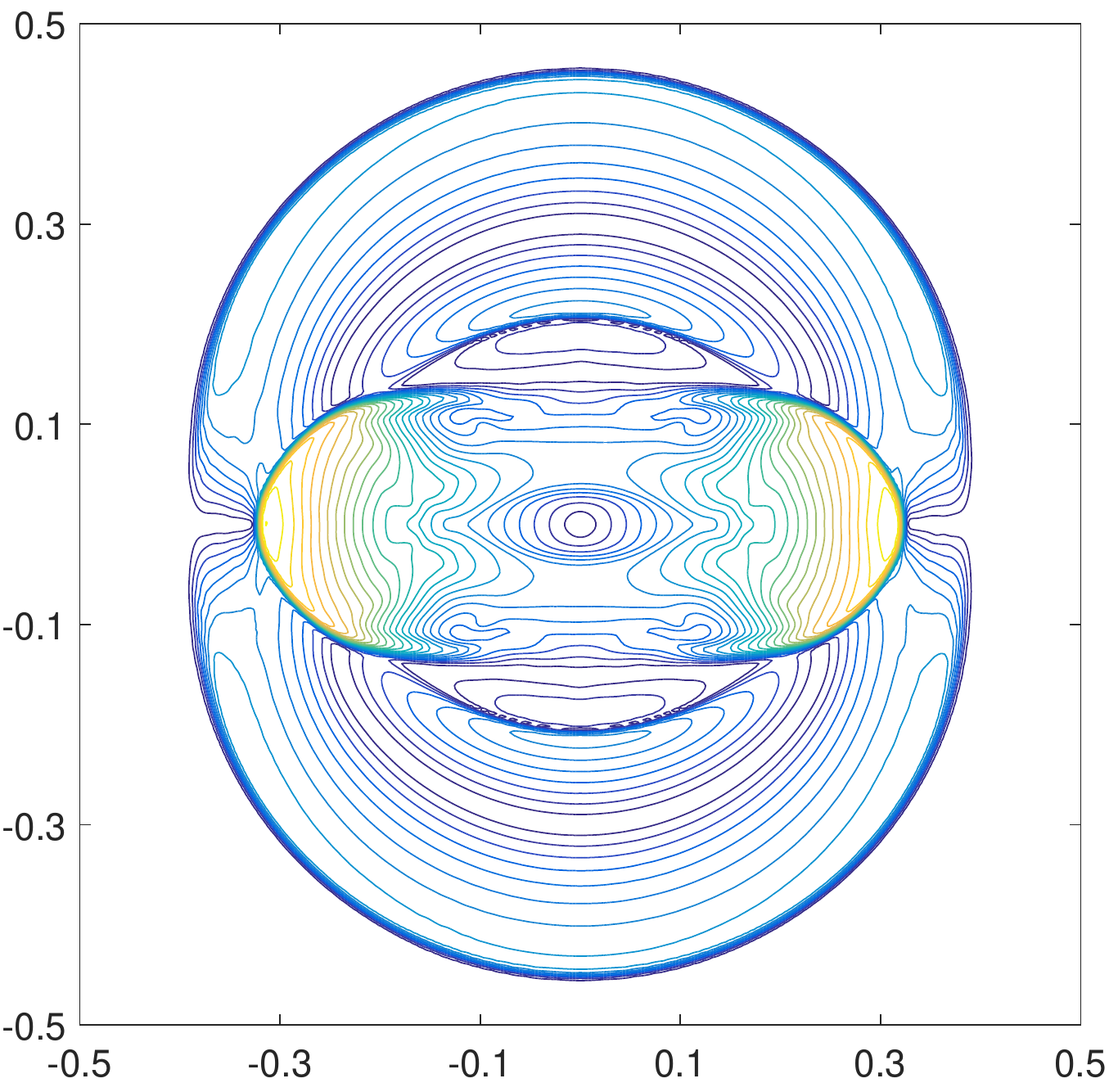}}
	{\includegraphics[width=0.49\textwidth]{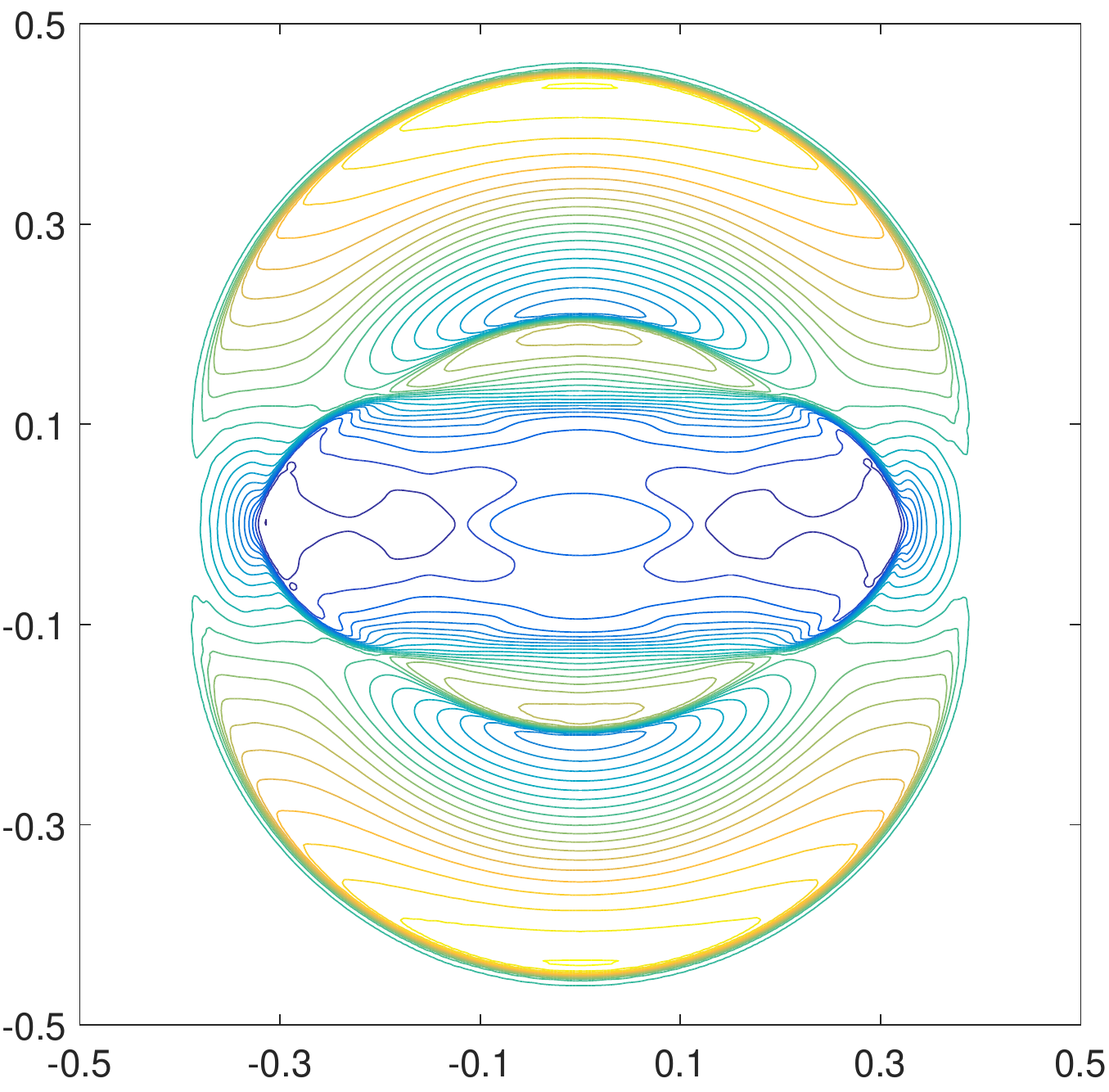}}
	\caption{\small The 
		contour plots of density (top left), 
		pressure (top right), velocity $|{\bf v}|$ (bottom left) 
		and magnetic pressure (bottom right) at time $t=0.01$ for the blast problem.} 
	\label{fig:BL1}
\end{figure}

The simulation is implemented in $[-0.5,0.5]^2$ with outflow boundary conditions. Our setup is the same as in \cite{BalsaraSpicer1999,cheng}. 
Initially, the domain 
is filled with fluid at rest with unit density. The explosion zone 
$(r<0.1)$ has a pressure of $1000$, while the ambient medium $(r>0.1)$ has a pressure of $0.1$, 
where $r=\sqrt{x^2+y^2}$. The magnetic field is initialized in the $x$-direction as $100/\sqrt{4\pi}$.  
For this setup, the ambient medium has a small 
plasma-beta (about $2.51\times 10^{-4}$).  
Our numerical results at $t=0.01$, obtained by the PP third-order DG method with $320\times320$ cells, are displayed in 
Fig.~\ref{fig:BL1}. 
Our results agree well with those 
in \cite{BalsaraSpicer1999,Li2011,Christlieb}, 
and the density profile is well resolved with much less oscillations than those shown in \cite{BalsaraSpicer1999,Christlieb}. The velocity profile clearly shows higher resolution than that in \cite{WuShu2018} obtained by the same DG method but with the global LF flux.  
We also notice that, if the PP limiter is turned off, 
the condition \eqref{eq:FVDGsuff} will be violated since   
$t \approx 2.24\times 10^{-4}$, and the method will fail due to negative numerical pressure.

%
%
%

\subsection{Shock cloud interaction}

	This test \cite{Dai1998} simulates the disruption of a high density cloud by a strong shock wave, and has been widely simulated in the literature (e.g., \cite{Toth2000,Balbas2006}). 
	We employ the same setup as in \cite{Toth2000,Balbas2006}.  
	The simulation is implemented in the domain $\Omega = [0,1]^2$ with the right boundary specified as 
	supersonic inflow condition and the others as outflow conditions. 
	The adiabatic index $\gamma=\frac53$, and the initial conditions are given by the two states 
	$$(\rho,{\bf v},p,{\bf B})=
	\begin{cases}
	(3.86859,0,0,0,167.345,0,2.1826182,-2.1826182),\quad & x<0.6,
	\\
	(1,-11.2536,0,0,1,0,0.56418958,0.56418958),\quad & x>0.6,
	\end{cases}$$ 
	separated by a discontinuity parallel to the $y$-axis at $x=0.6$. 
	To the right of the discontinuity there is a circular cloud of radius $0.15$, centered at $x=0.8$ and $y=0.5$. The cloud has the same states as the surrounding fluid except for a higher density of $10$. 
	
		\begin{figure}[htb]
		\centering
		{\includegraphics[width=0.99\textwidth]{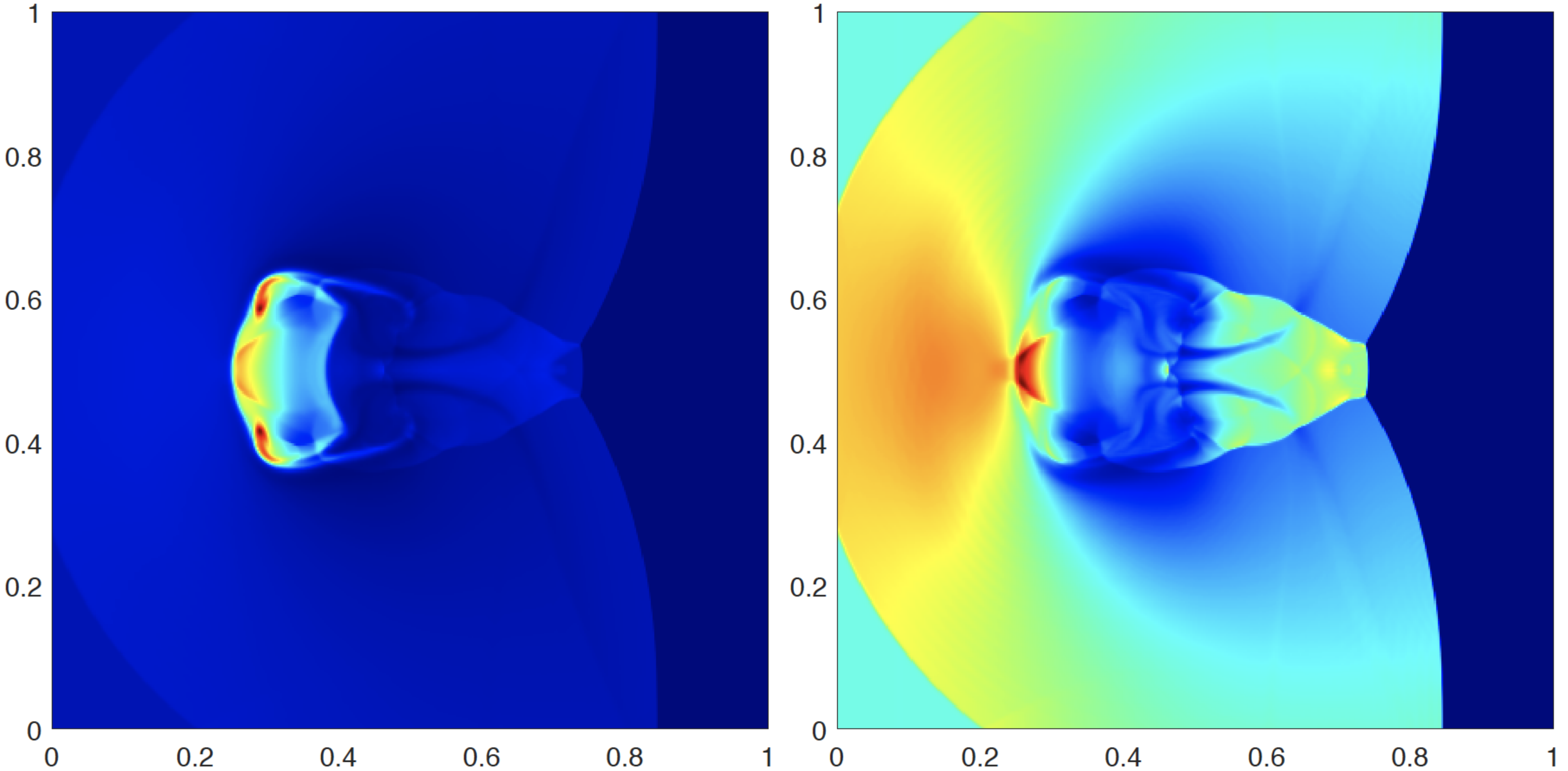}}
		\caption{\small The schlieren images of density (left) and  
			pressure (right) at time $t=0.06$ for the shock cloud interaction problem.} 
		\label{fig:sc}
	\end{figure}

	We simulate this problem by using our PP third-order DG 
	method with $400\times 400$ cells. 
	The numerical results at time $t=0.06$ 
	are shown in 
	Fig.~\ref{fig:sc}.  
	It is seen that 
	the complex flow structures and interactions are correctly captured, 
	and the results agree well with
	those in, for example, \cite{Toth2000,Balbas2006}. 
	In this test, it is also necessary to employ the PP limiter 
	to enforce the condition \eqref{eq:FVDGsuff}. 
	We also observe that, if the penalty term is dropped from our PP DG method, negative pressure will appear in the cell average of the DG solutions and the code breaks down at $t\approx 0.014$, because 
	the resulting scheme (namely the locally divergence-free DG method with the proposed HLL flux and the PP and WENO limiters) is not PP in general. This further confirms the importance of 
	the penalty term.

\subsection{Astrophysical jets}
The last test is to simulate jet flow, 
which is relevant in astrophysics. 
In a high Mach number jet with strong magnetic field, 
the internal energy is very
small compared to the huge magnetic and/or kinetic energies, thus 
negative pressure is very likely to be produced in the numerical simulations. 
Moreover, there may exist shear flows, strong shock waves, and  interface
instabilities in high-speed jet flows. 
Successfully simulating such jet flows is indeed a challenge, cf.~\cite{zhang2010b,Balsara2012,WuTang2015,WuTang2017ApJS}.

		\begin{figure}[htb]
	\centering
	{\includegraphics[width=0.32\textwidth]{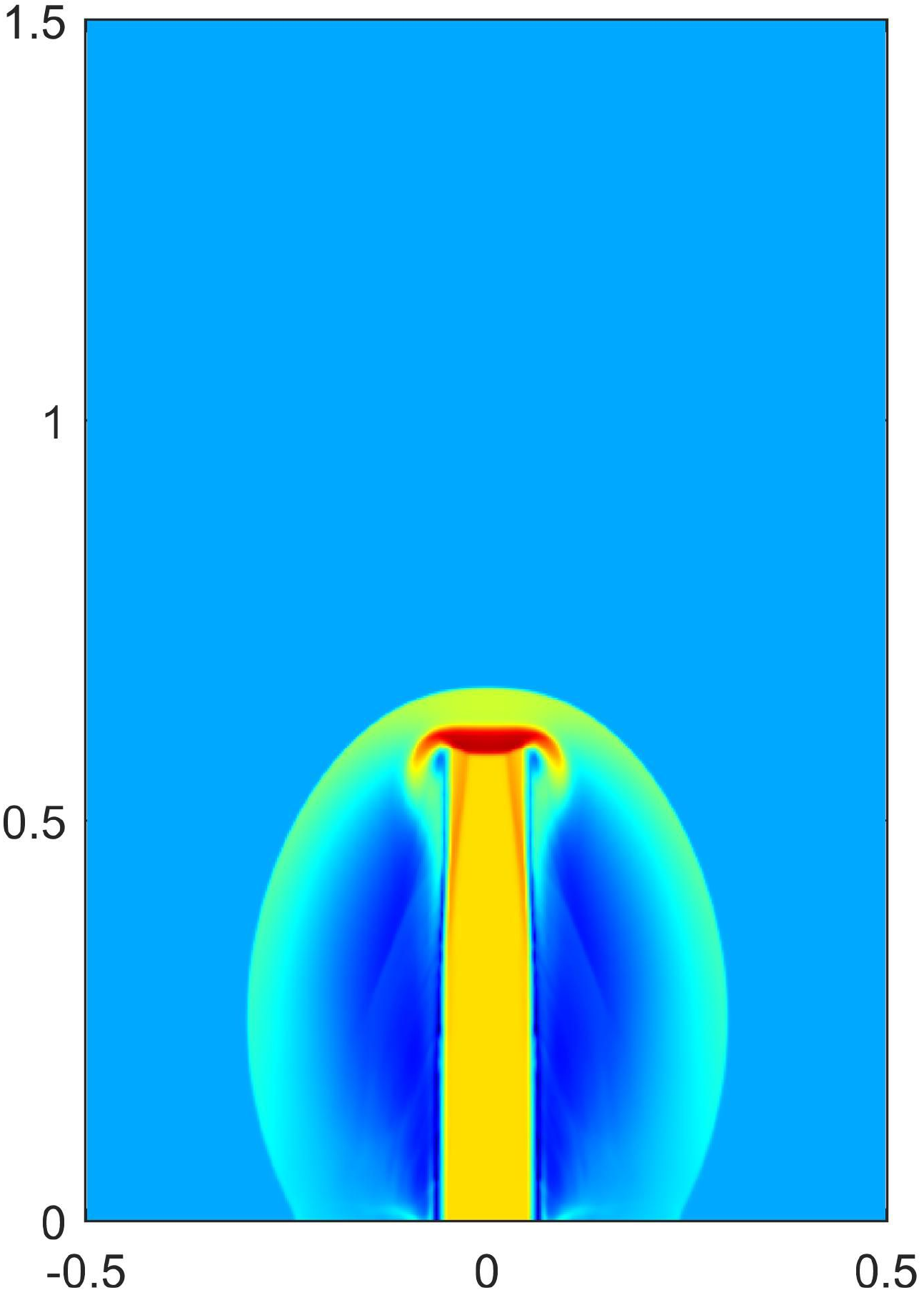}}
	{\includegraphics[width=0.32\textwidth]{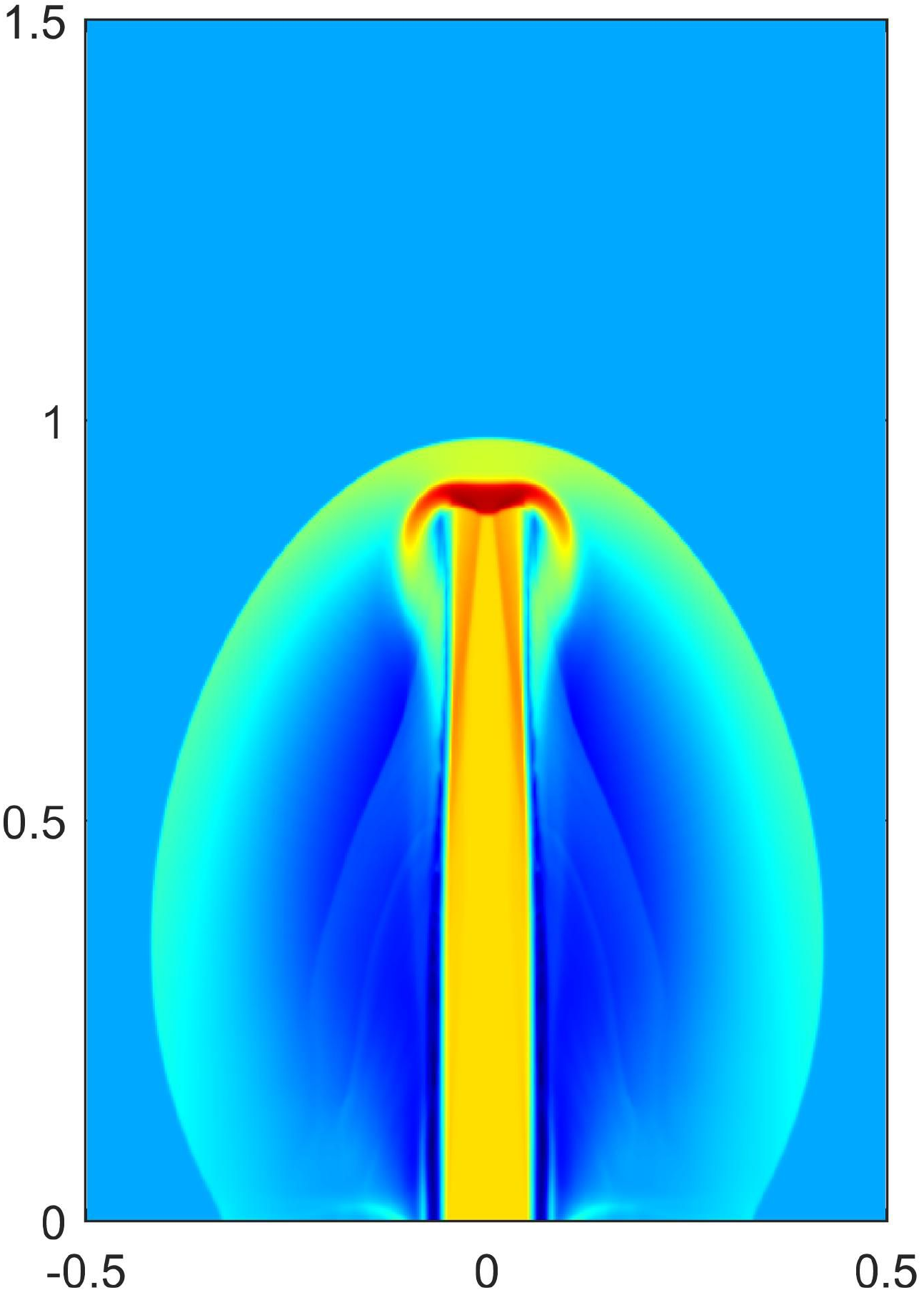}}
	{\includegraphics[width=0.32\textwidth]{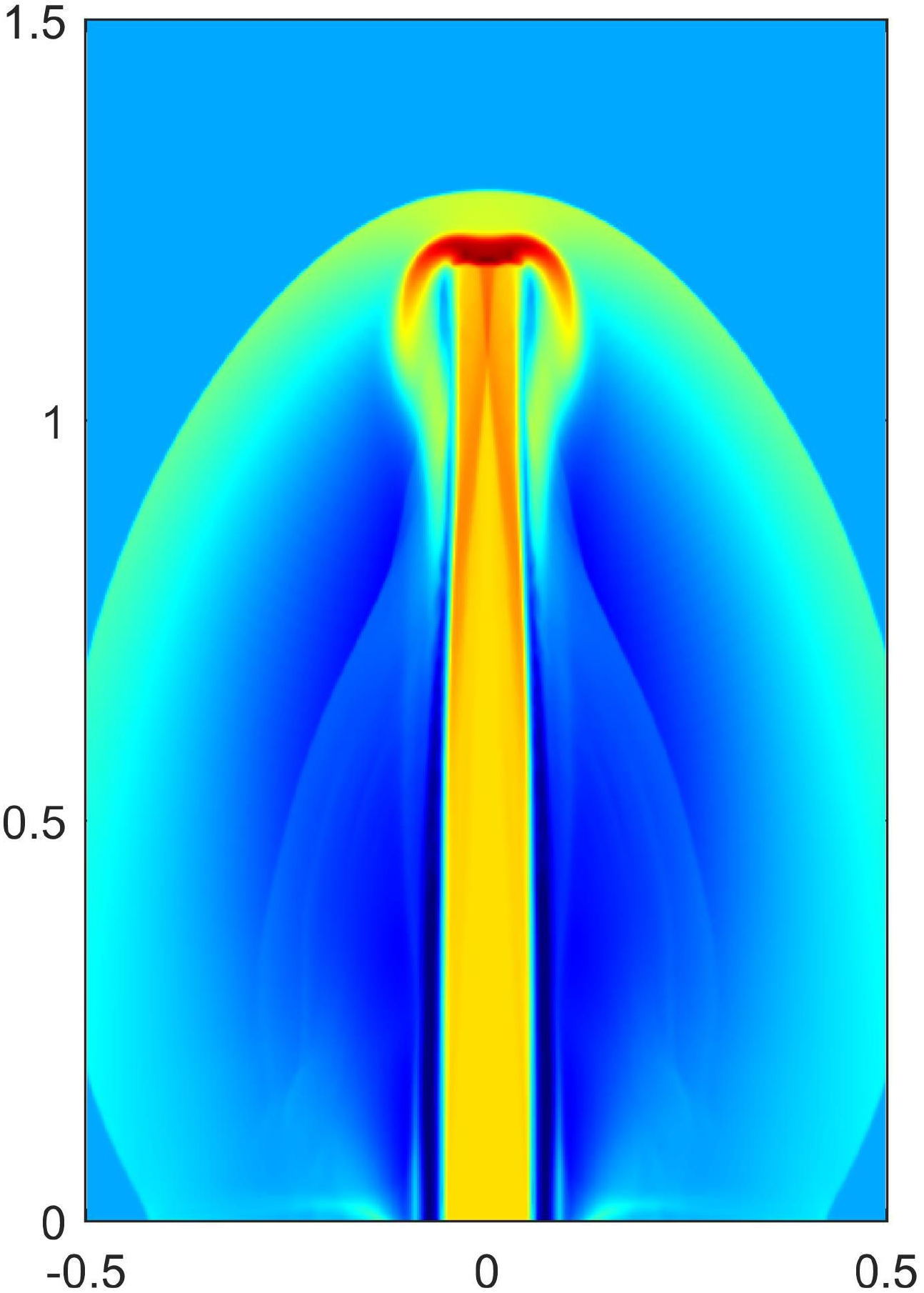}}
	{\includegraphics[width=0.32\textwidth]{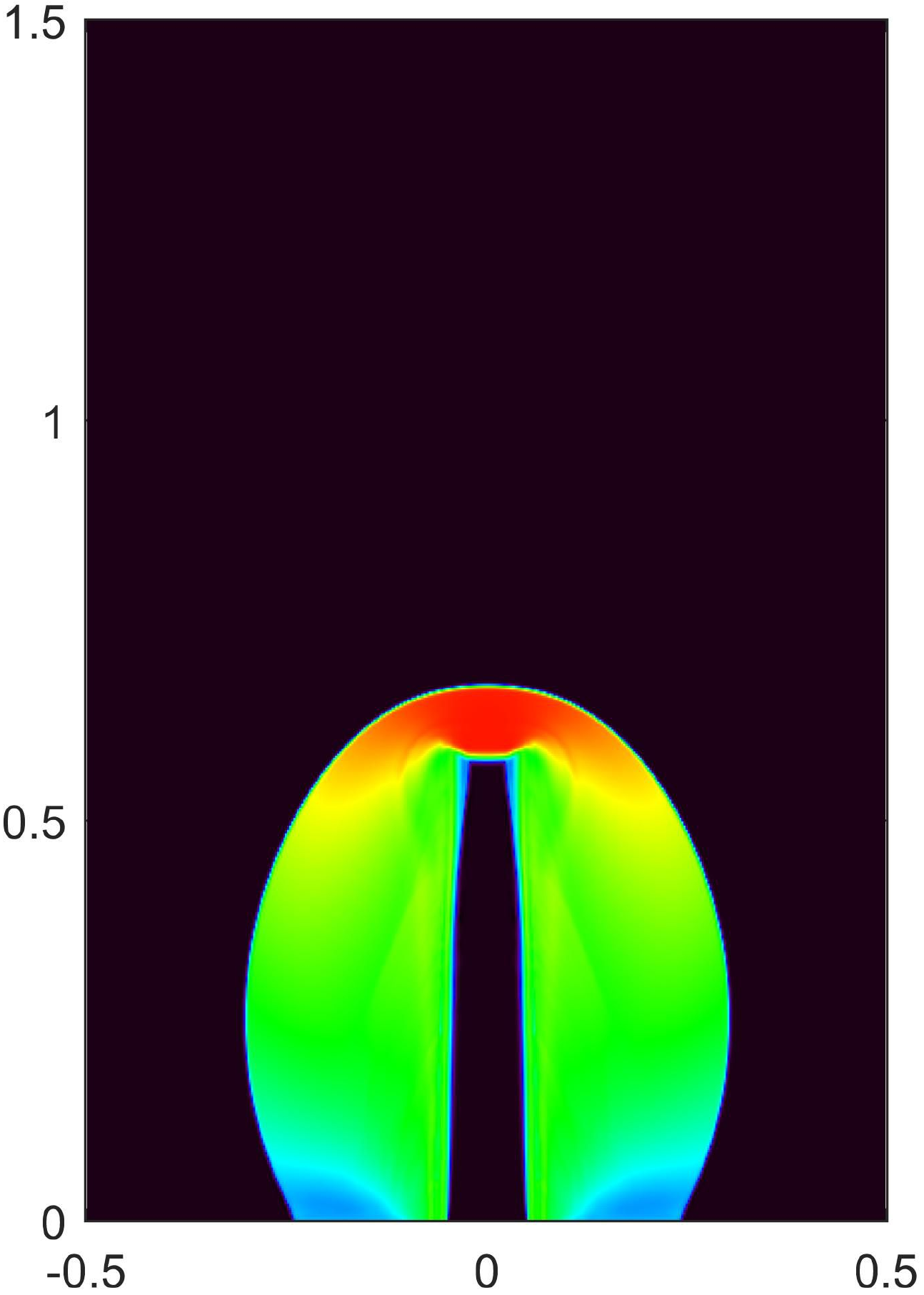}}
	{\includegraphics[width=0.32\textwidth]{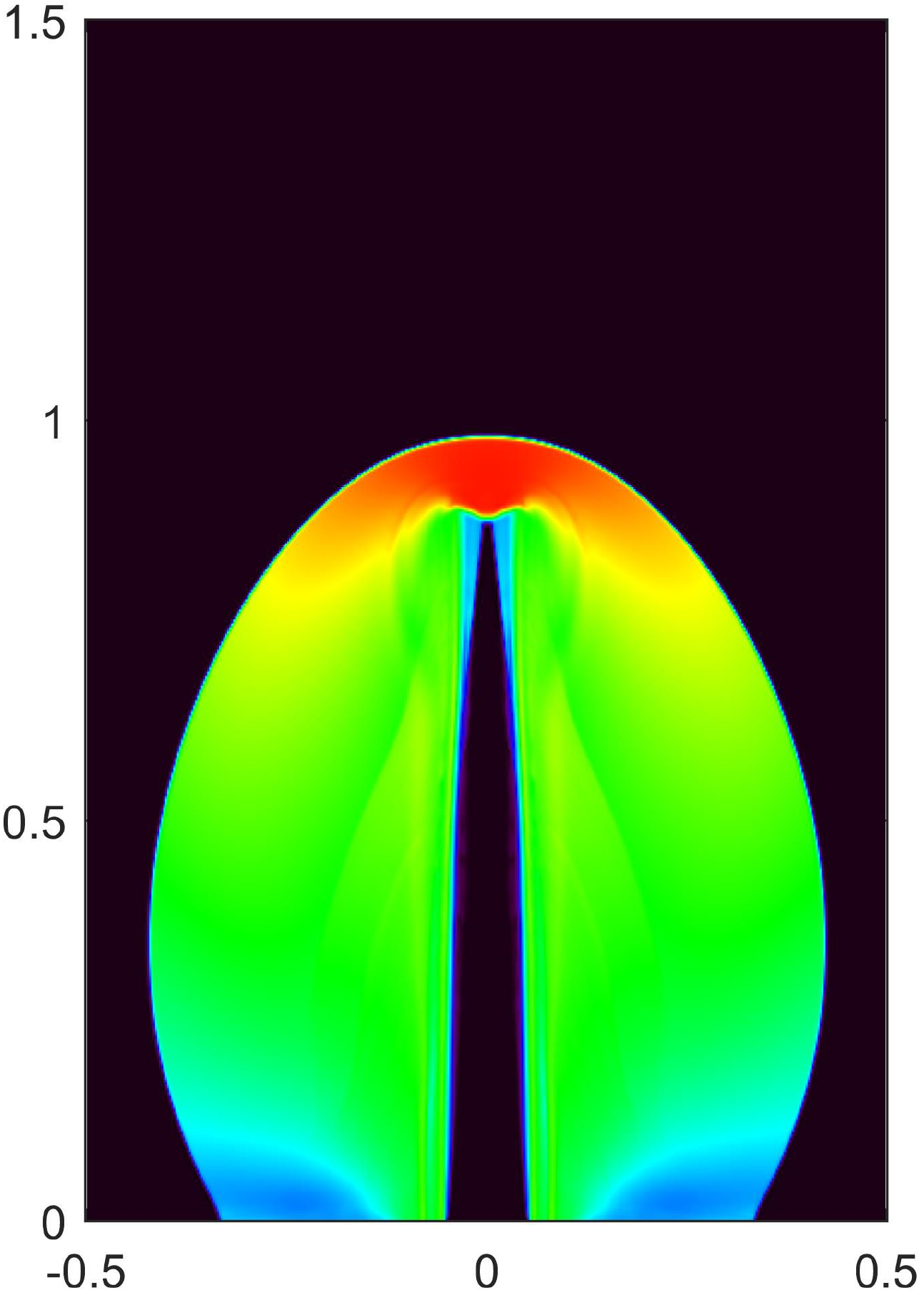}}
	{\includegraphics[width=0.32\textwidth]{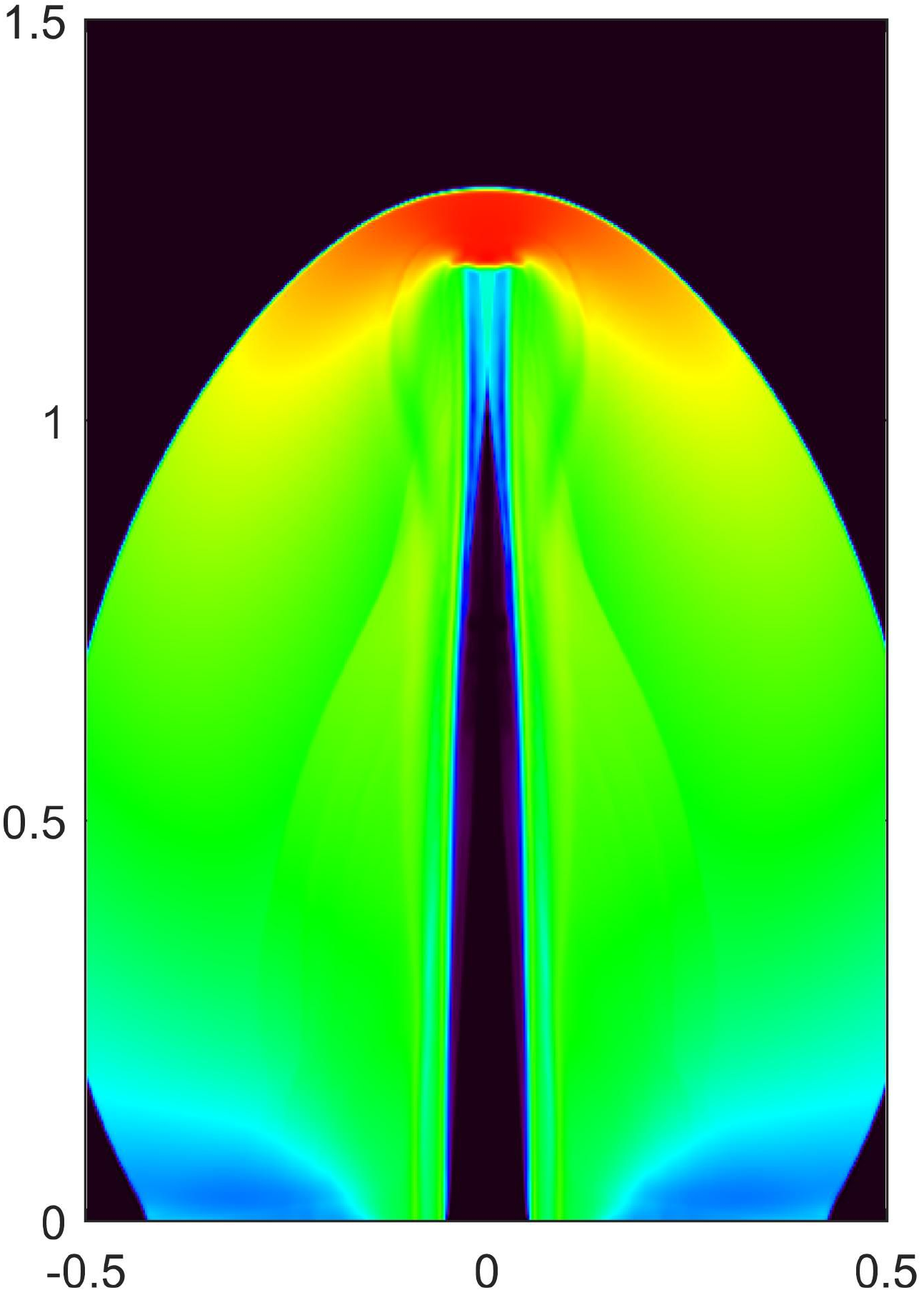}}
	\caption{\small The schlieren images of density logarithm (top) and  
		gas pressure
		logarithm (bottom) for the Mach 800 jet problem with $B_a=\sqrt{2000}$. From left to right: $t=0.001$, $0.0015$ and $0.002$.} 
	\label{fig:jet1}
\end{figure}

		\begin{figure}[htb]
	\centering
	{\includegraphics[width=0.32\textwidth]{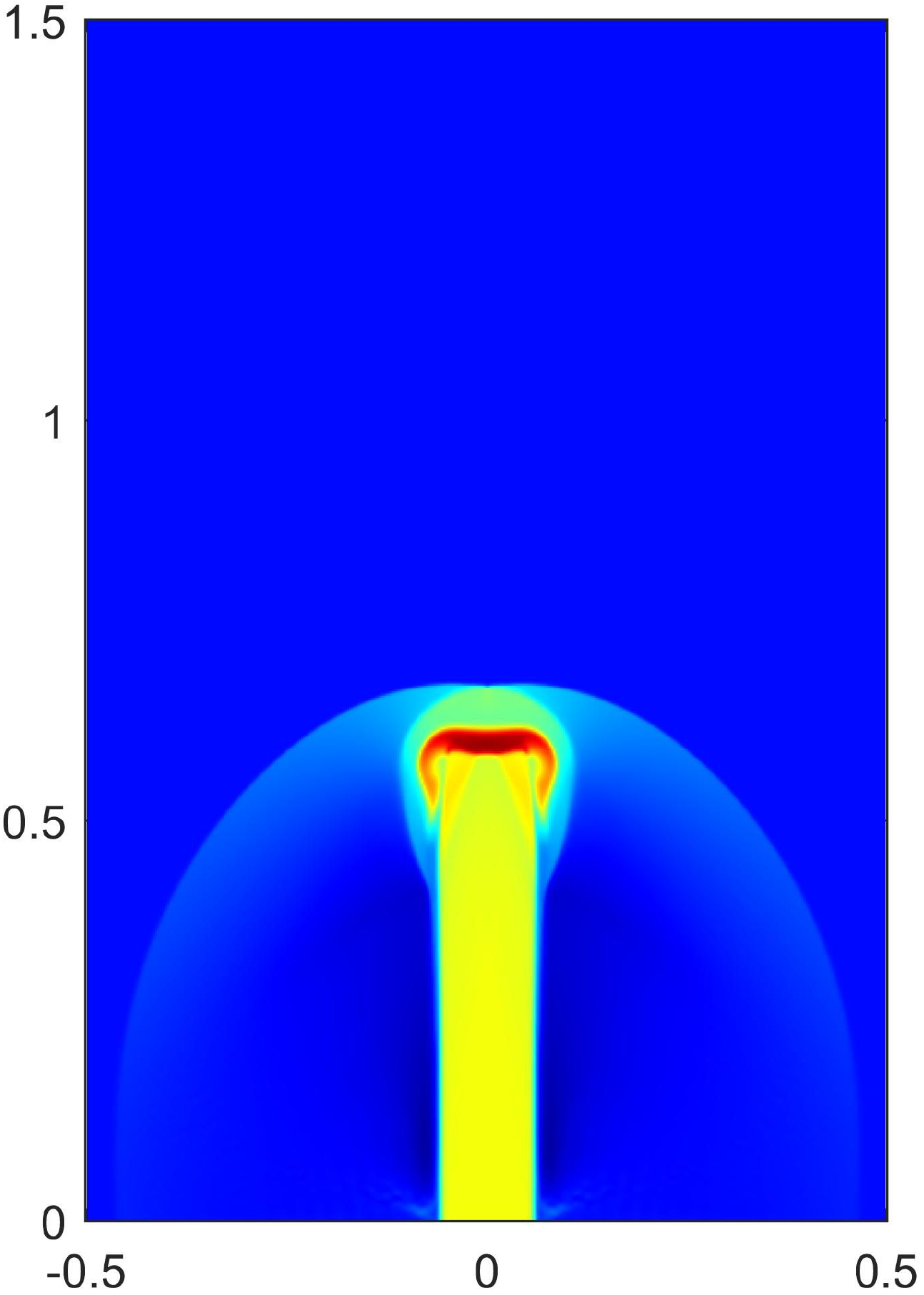}}
	{\includegraphics[width=0.32\textwidth]{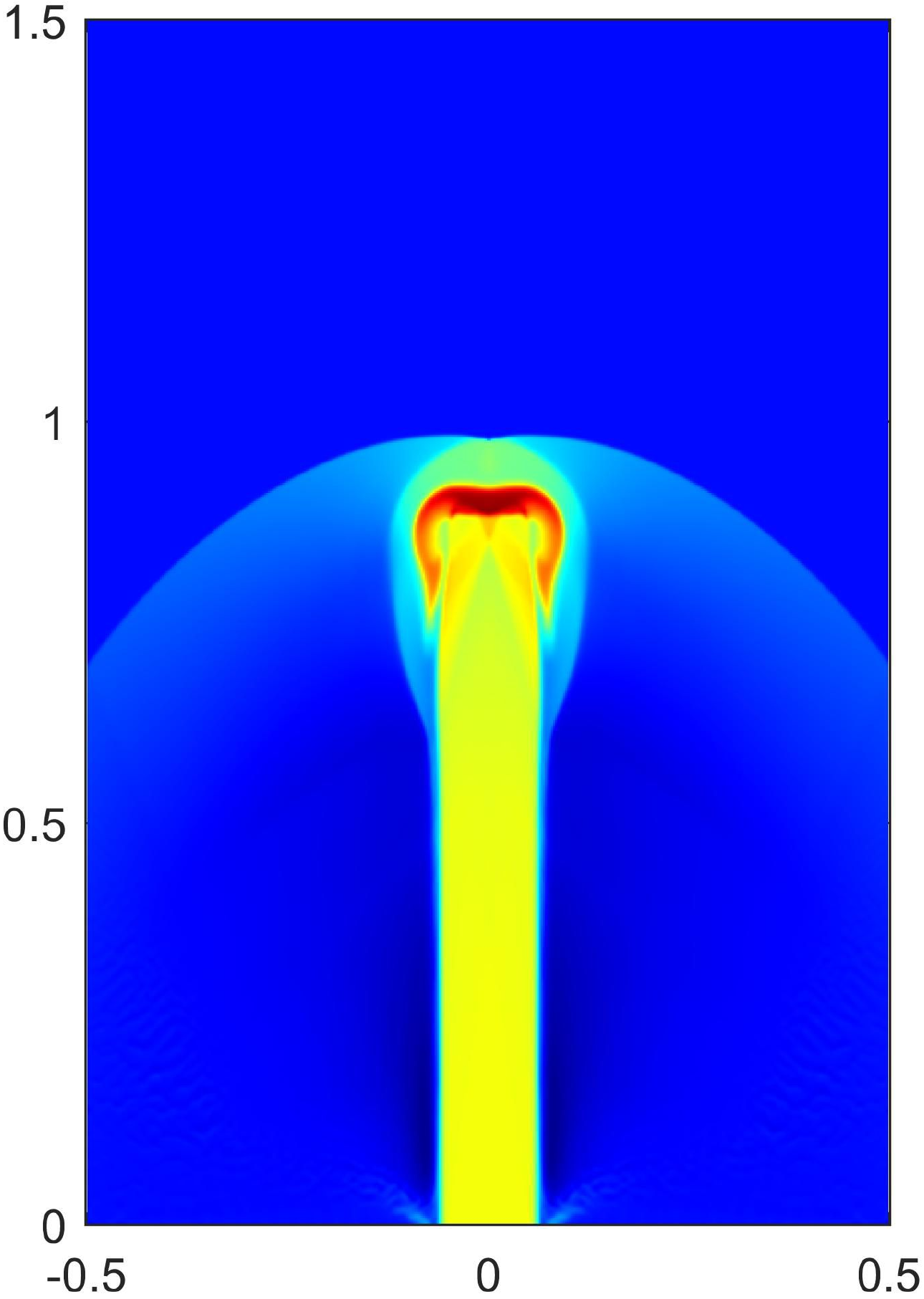}}
	{\includegraphics[width=0.32\textwidth]{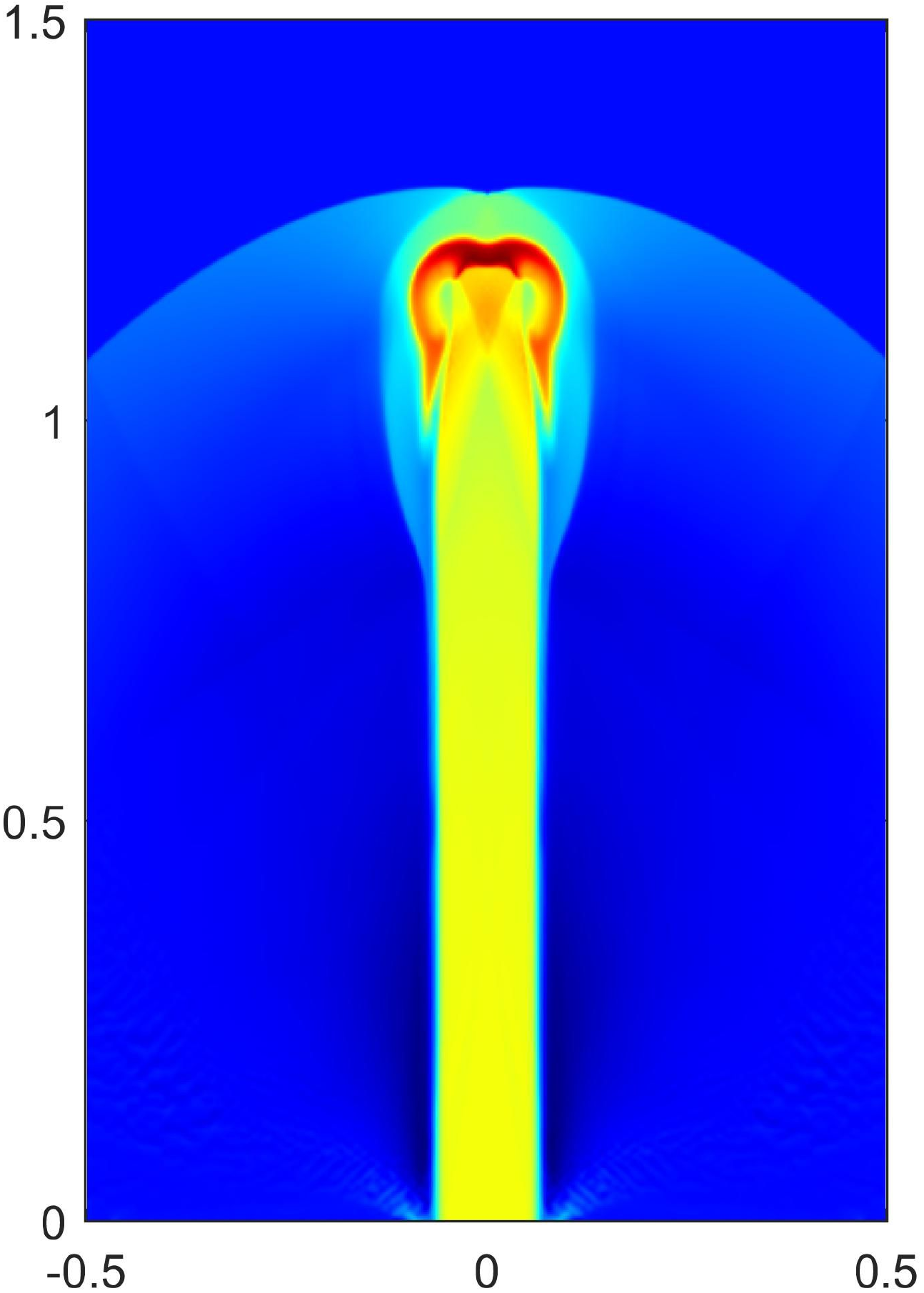}}
	{\includegraphics[width=0.32\textwidth]{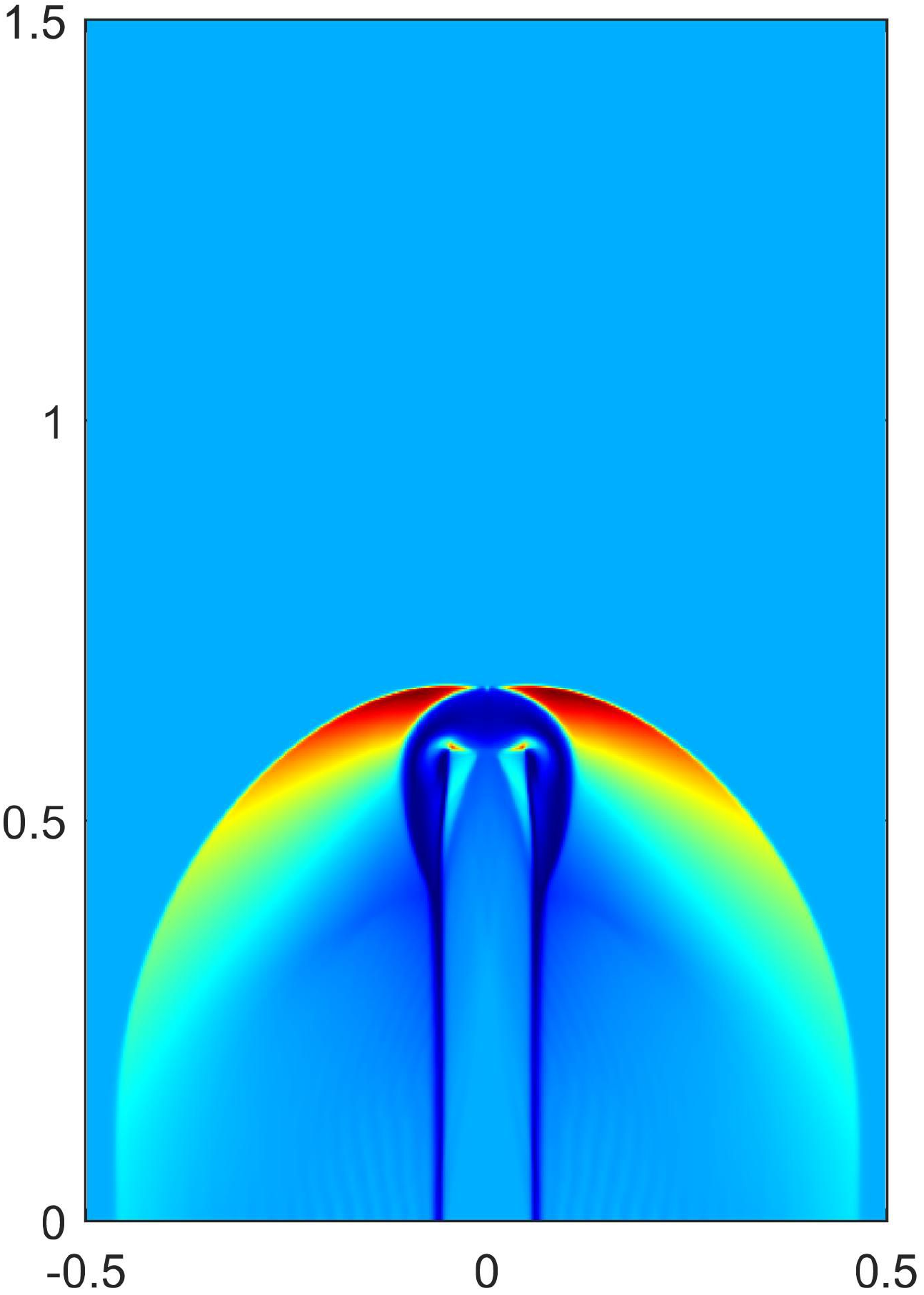}}
	{\includegraphics[width=0.32\textwidth]{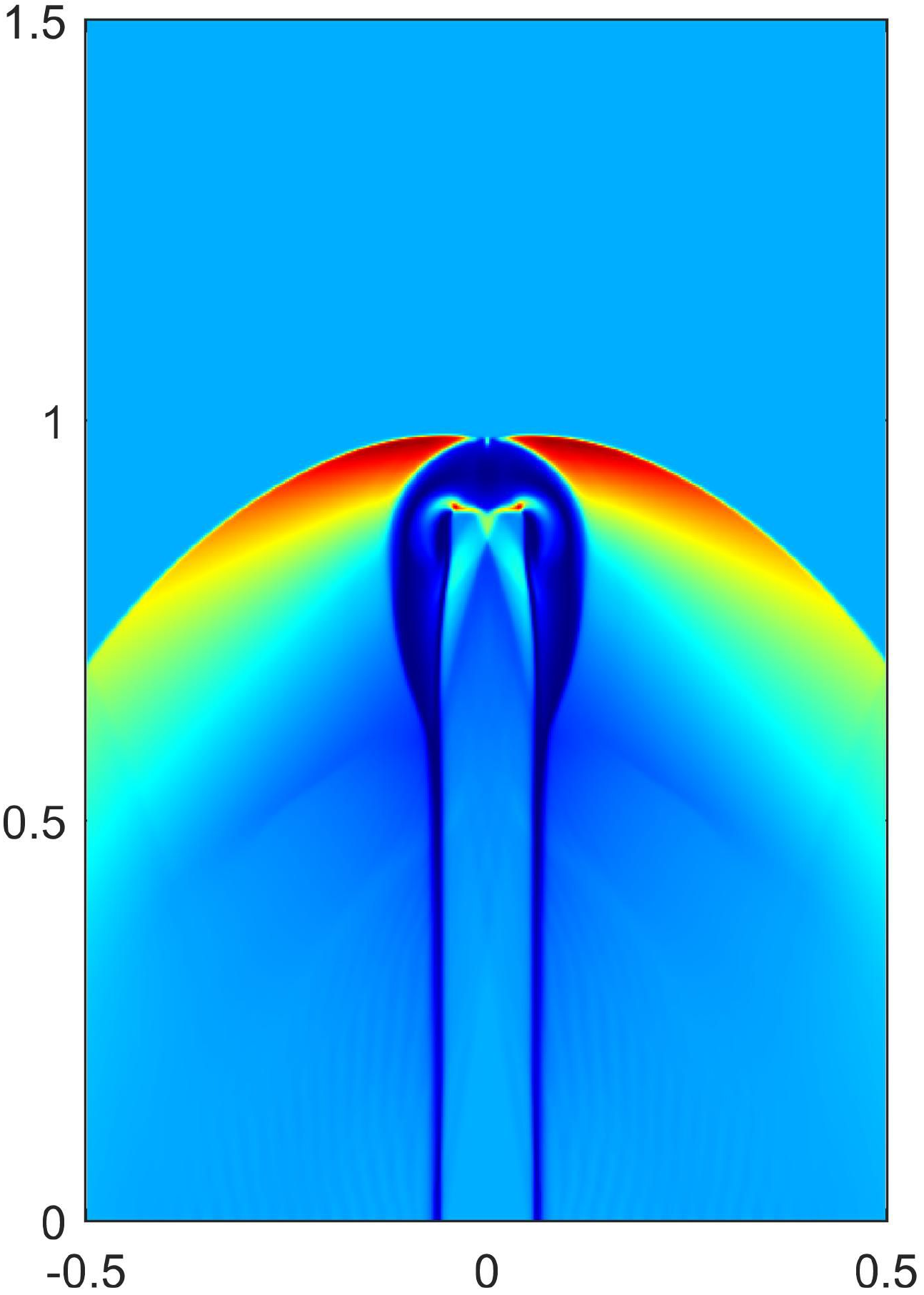}}
	{\includegraphics[width=0.32\textwidth]{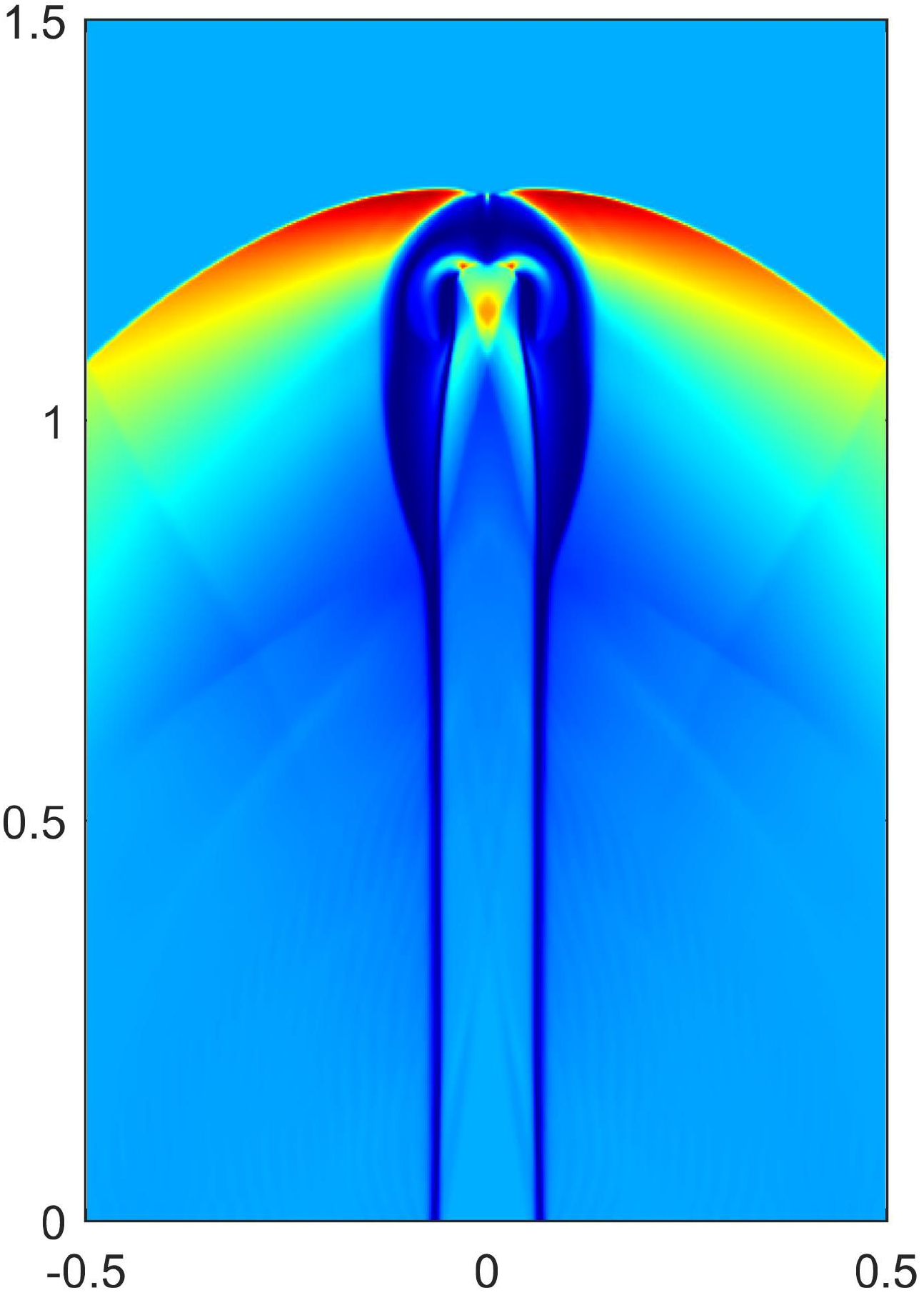}}
	\caption{\small The schlieren images of density logarithm (top) and  
		magnetic pressure
		(bottom) for the Mach 800 jet problem with $B_a=\sqrt{20000}$. From left to right: $t=0.001$, $0.0015$ and $0.002$.} 
	\label{fig:jet2}
\end{figure}

	\begin{figure}[htb]
	\centering
	{\includegraphics[width=0.32\textwidth]{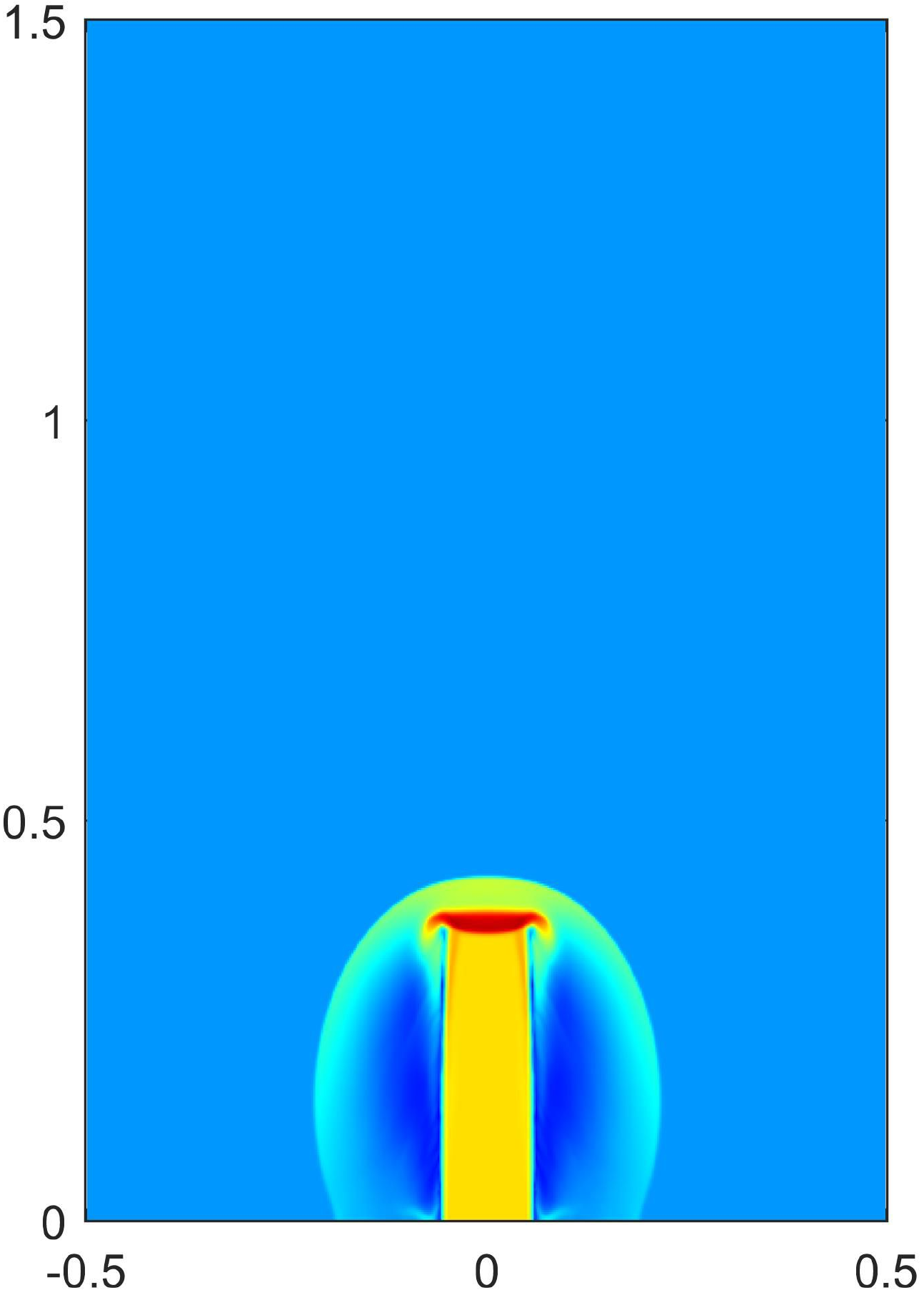}}
	{\includegraphics[width=0.32\textwidth]{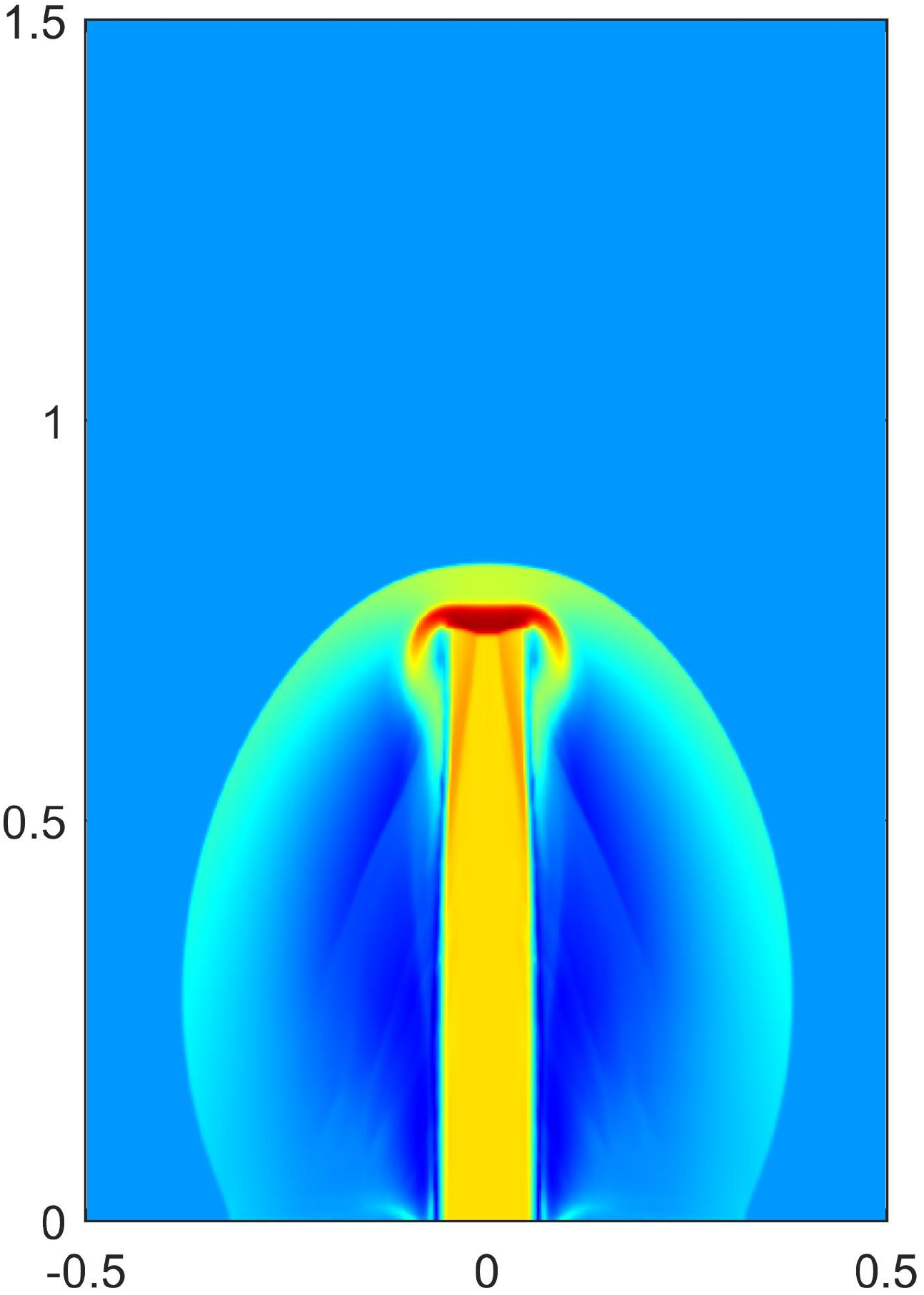}}
	{\includegraphics[width=0.32\textwidth]{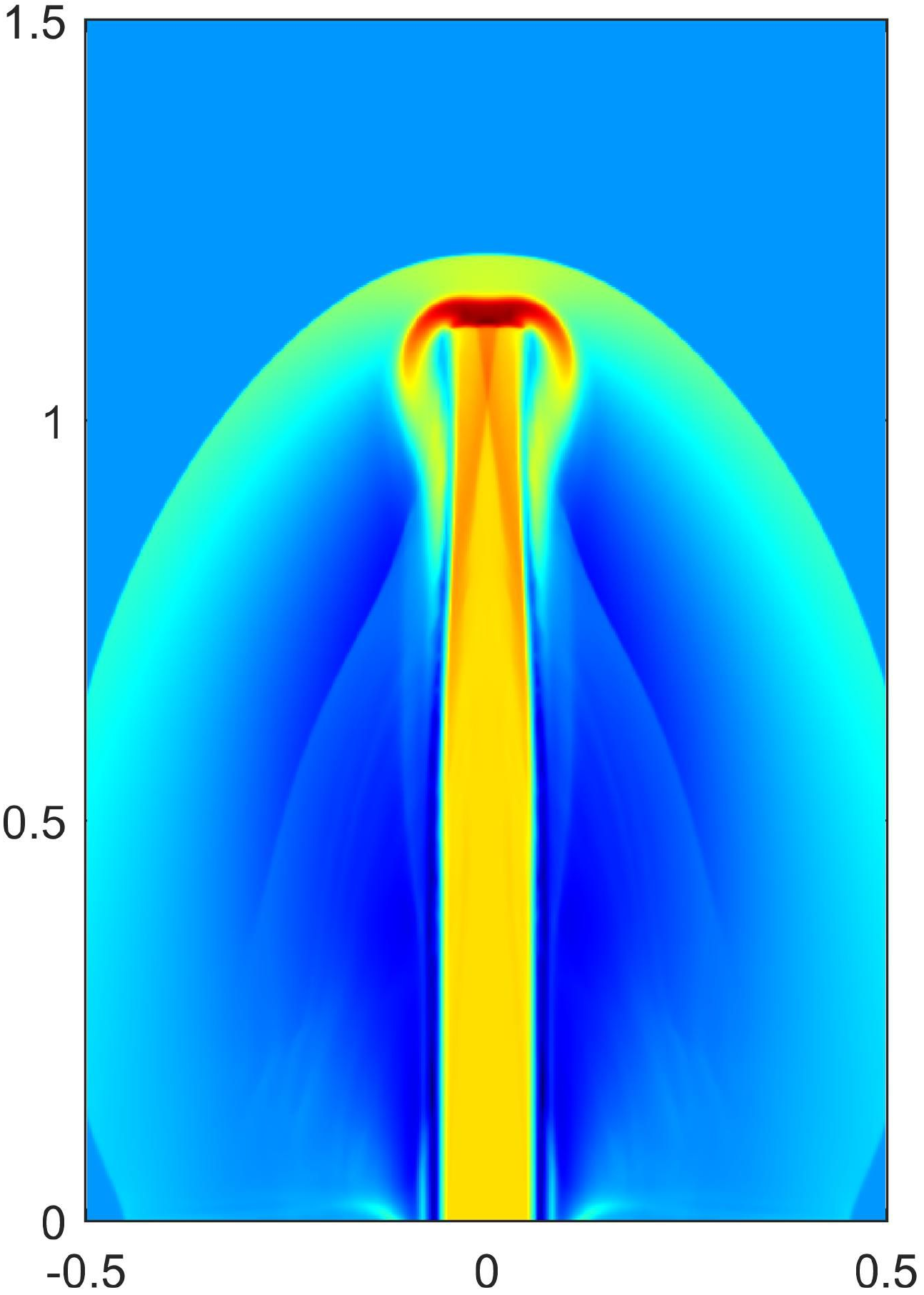}}
	\caption{\small The schlieren images of density logarithm 
		for the Mach 2000 jet problem with $B_a=\sqrt{20000}$. From left to right: $t=0.00025$, $0.0005$ and $0.00075$.} 
	\label{fig:jet4}
\end{figure}

\begin{figure}[htb]
	\centering
	{\includegraphics[width=0.32\textwidth]{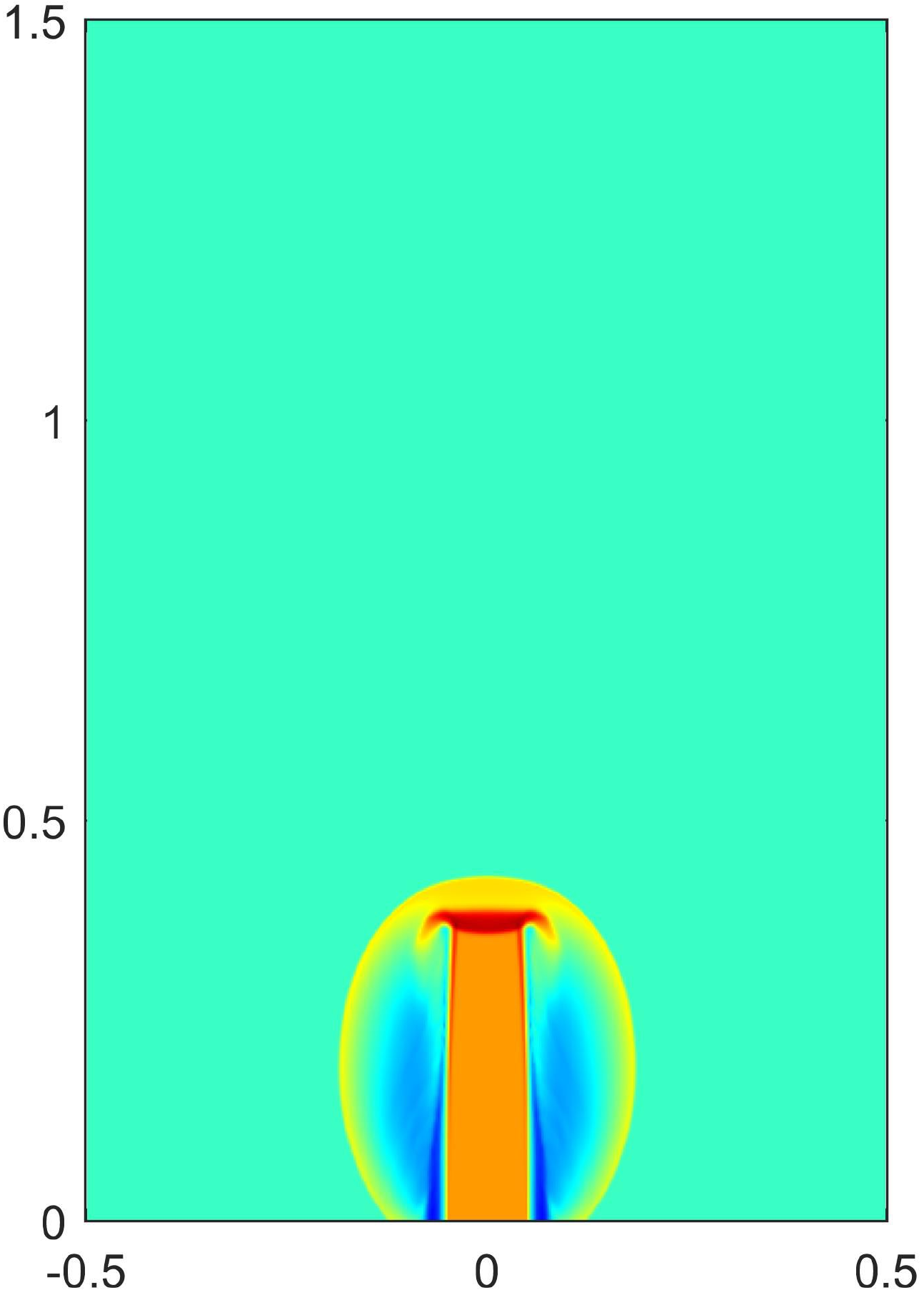}}
	{\includegraphics[width=0.32\textwidth]{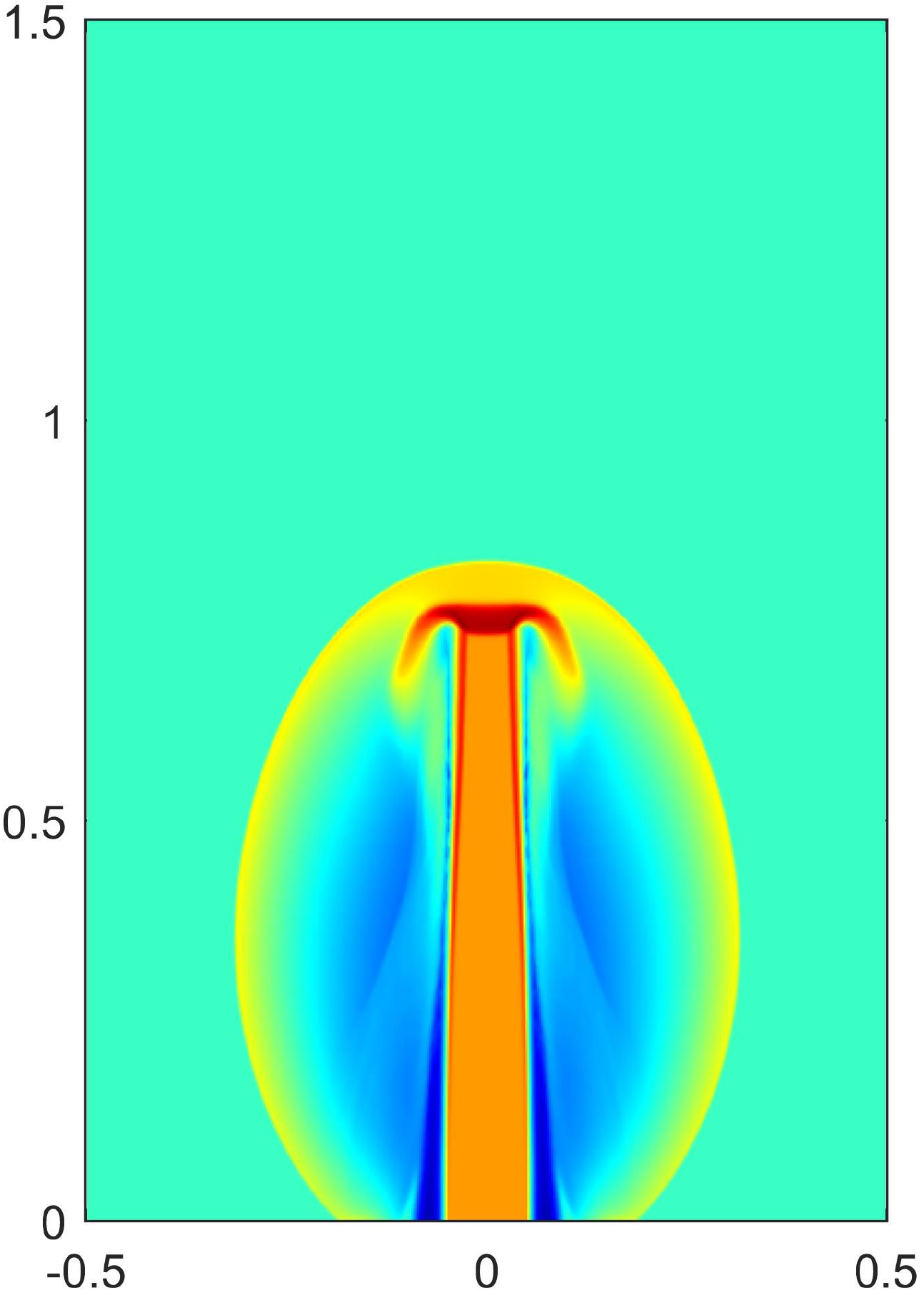}}
	{\includegraphics[width=0.32\textwidth]{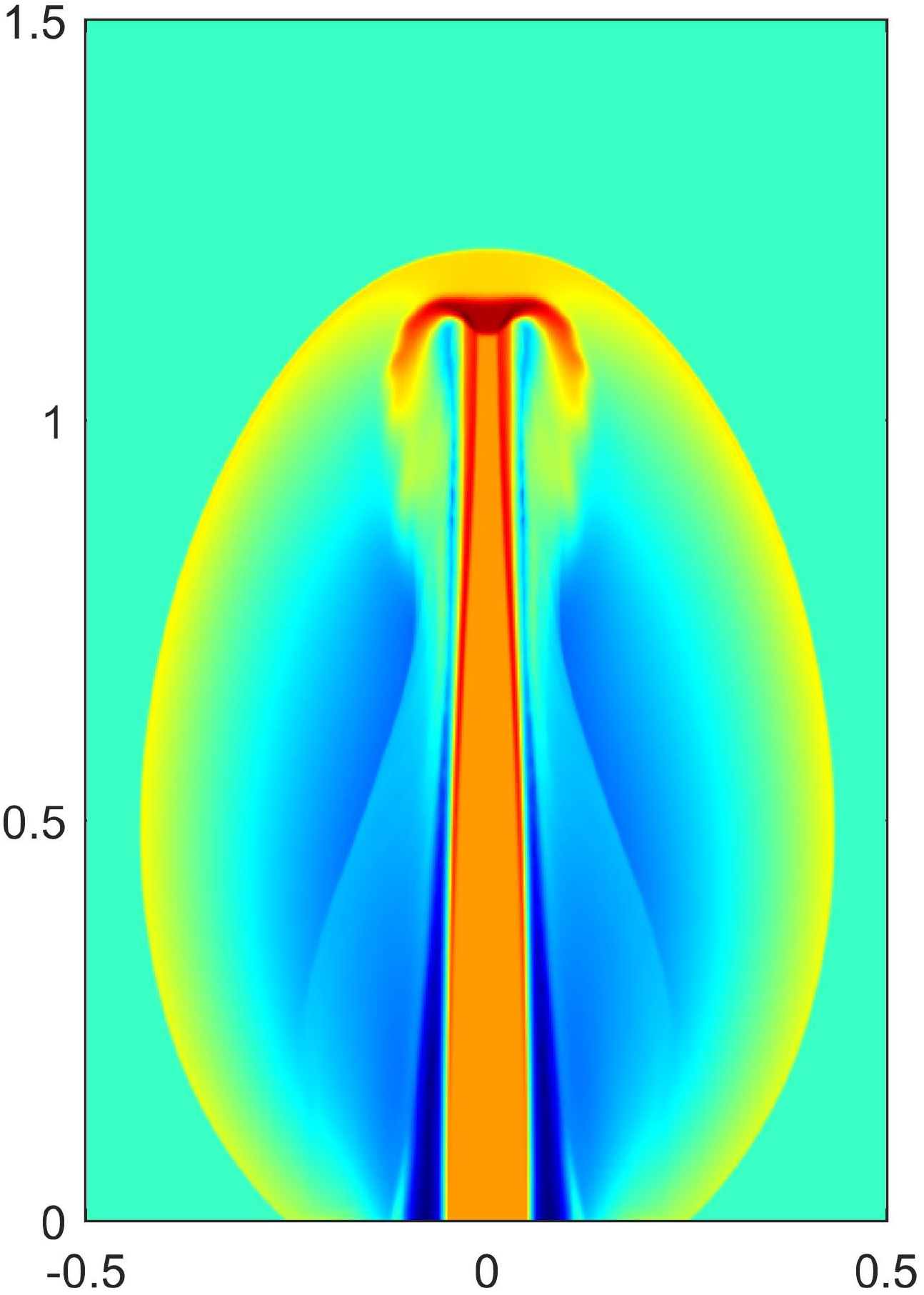}}
	\caption{\small The schlieren images of density logarithm 
		for the Mach 10000 jet problem with $B_a=\sqrt{20000}$. From left to right: $t=0.00005$, $0.0001$ and $0.00015$.} 
	\label{fig:jet5}
\end{figure}

We consider the Mach 800 MHD jets proposed in \cite{Wu2017a,WuShu2018} and extended from the gas dynamical jet of Balsara \cite{Balsara2012} by adding a magnetic field. 
Initially, the domain $[-0.5,0.5]\times[0,1.5]$ is 
full of the static ambient medium with $(\rho,p)=(0.1\gamma,1)$. The adiabatic index $\gamma=1.4$. 
A Mach 800 dense jet is injected in the $y$-direction through the inlet part ($\left|x\right|<0.05$) on the bottom boundary ($ y=0$). 
The fixed inflow condition with $(\rho,p,v_1,v_2,v_3)=(\gamma,1,0,800,0)$
is specified on the nozzle $\{y=0,\left|x\right|<0.05\}$, while the
other boundary conditions are outflow. 
A magnetic field $(0,B_a,0)$ is initialized along the $y$-direction.  
As $B_a$ is set larger, this test becomes more challenging. We set computational domain as  
$[0,0.5]\times[0,1.5]$ with the reflecting boundary
condition specified at $x=0$, and divided it into $200 \times 600$ cells.
We here show our numerical results in two strongly magnetized cases:  
(i) $B_a= \sqrt{2000}$, and the corresponding plasma-beta $\beta_a=10^{-3}$; 
(ii) $B_a= \sqrt{20000}$, and the corresponding plasma-beta $\beta_a=10^{-4}$. 
The schlieren images 
of the numerical solutions for these two cases   
are respectively displayed in 
Figs.~\ref{fig:jet1} and \ref{fig:jet2}   
within the domain $[-0.5,0.5]\times[0,1.5]$. 
Those plots clearly show the time evolution of the jets. 
It is seen that the flow structures in different magnetized cases  
are very different. 
The present method well captures the 
 Mach shock wave at the jet head and
other discontinuities with high resolution. 
The results agree with those in \cite{WuShu2018} computed by 
the PP DG method with a global LF flux. 
In these extreme tests, our PP method exhibits good robustness 
without using any
artificial treatment. 
We also perform the tests with varied Mach numbers, 
and the method also works very robustly. For example, the numerical result for a Mach 2000 jet with $B_a=\sqrt{20000}$ is displayed in Fig.~\ref{fig:jet4}. Interestingly, 
the flow structures are similar to those in Fig.~\ref{fig:jet1} of the Mach 800 jet with a weaker magnetic field $B_a=\sqrt{2000}$. This is probably due to the huge kinetic energy, which becomes dominant and weakens the effect of magnetic field. The dynamics of the Mach 2000 jet evolve much faster than the Mach 800 jet, as expected. 
A higher Mach (Mach 10000) jet with $B_a=\sqrt{20000}$ is further simulated and shown in Fig.~\ref{fig:jet5}. We see that this jet shape is thinner.

In the above simulations, it is necessary to employ the PP limiting procedure to meet the condition \eqref{eq:FVDGsuff}, which is not satisfied automatically. 
To confirm the importance of the proposed penalty term in our PP schemes, 
we have also performed the above tests by dropping the penalty term and keeping 
the PP and WENO limiters turned on. 
The resulting scheme is actually the locally divergence-free, conservative, third-order DG method with PP and WENO limiters. 
We find that this scheme with either the proposed HLL flux or the global LF flux, which is generally not PP in theory, cannot run the above jet tests. 
The failure results from negative numerical pressure produced in the cell averages of the DG solution. 
We observe that, without the proposed penalty term, the code also fails on a refined mesh, and also for 
more strongly magnetized cases. 
 This, again, demonstrates that the proposed penalty term is really crucial for guaranteeing the PP property.

\section{Conclusions}\label{sec:conclusion}

In this paper, we proposed and analyzed   
provably PP high-order DG and finite volume schemes for the ideal MHD on general meshes. 
The unified auxiliary theories were built for rigorous 
PP analysis of numerical schemes with HLL-type flux 
on an arbitrary polytopal mesh. 
A close relation was established between the PP property and the discrete divergence of magnetic field on general meshes. 
We also derived explicit estimates of the wave speeds in the HLL flux 
to ensure the provably PP property. 
In the 1D case, we proved that the standard finite volume and DG methods with the proposed HLL flux are PP, under a condition accessible by a PP limiter. 
In the multidimensional cases, we constructed provably PP high-order DG schemes based on suitable discretization of the modified MHD system \eqref{eq:MHD:GP}.
In addition to the proper wave speeds in the numerical flux and a standard PP limiter, 
we demonstrated that a coupling 
of two divergence-controlling techniques is also crucial 
for achieving the provably PP property.    
The two techniques are 
the locally divergence-free DG element and a properly discretized Godunov--Powell source term, 
which control the divergence error within each cell 
and across the cell interfaces, respectively.  
Our analysis clearly revealed that these two techniques exactly 
contribute the discrete divergence terms 
which are absent in a standard multidimensional DG schemes but 
very important for ensuring the PP property. 
We also proved in Appendix \ref{app:1} the positivity of 
the strong solution 
of the modified MHD system \eqref{eq:MHD}. 
Such a feature, 
not enjoyed by the conservative system \eqref{eq:MHD} (see \cite{WuShu2018}), can serve as a justification for designing provably PP multidimensional schemes 
based on the modified system \eqref{eq:MHD:GP}.
The analysis and findings in this paper provide a clear understanding,  at both discrete and continuous levels, of the relation 
between the PP property and the divergence-free constraint. 
The proposed framework and analysis techniques as well
as the provenly PP schemes can also be useful 
for investigating or designing other PP schemes for
the ideal MHD.

Several numerical tests were conducted 
on 1D mesh and 2D rectangular mesh, to confirm the provenly PP property and to demonstrate the effectiveness of the proposed PP techniques. The implementation of 
our PP DG schemes on unstructured triangular meshes 
is ongoing and will be reported separately in the future.

\appendix

\section{Positivity of strong solutions of the modified MHD system}\label{app:1}In \cite{WuShu2018}, we analytically demonstrated that 
the exact smooth solution of the 
conservative MHD system \eqref{eq:MHD} 
may fail to be PP  
if the divergence-free condition \eqref{eq:2D:BxBy0} is violated.  
Here we would like to show that the strong solutions of 
the modified MHD system \eqref{eq:MHD:GP} 
always retain the positivity of density and pressure even if the divergence-free condition \eqref{eq:2D:BxBy0} is not satisfied. It is reasonable to hope 
that such a claim may also hold for the 
weak entropy solutions of \eqref{eq:MHD:GP}.

Consider the initial-value problem of the system 
\eqref{eq:MHD:GP}, for ${\bf x}\in {\mathbb R}^d$ and $t>0$, with initial data 
\begin{equation}\label{eq:initialA}
(\rho,{\bf v}, p,{\bf B}) ( {\bf x}, 0 ) = 
(\rho_0,{\bf v}_0, p_0,{\bf B}_0) ({\bf x}),
\end{equation}
and the ideal EOS $p=(\gamma-1)\rho e$, where $\gamma>1$. 
Using the method of characteristics, one can show the following result.

\begin{proposition}
	Assume that the initial data \eqref{eq:initialA} 
	are in $C^1({\mathbb R}^d)$ with $\rho_0({\bf x})>0$ and $p_0({\bf x})>0,$ $\forall {\bf x} \in {\mathbb R}^d$. 
	If the initial-value problem of \eqref{eq:MHD:GP} with  \eqref{eq:initialA} has a $C^1$ solution $(\rho,{\bf v}, p,{\bf B}) ( {\bf x}, t )$ for ${\bf x} \in {\mathbb R}^d$ and $0\le t < T$, 
	then the solution satisfies $\rho({\bf x},t)>0$ and 
	$p({\bf x},t)>0$ for all ${\bf x} \in {\mathbb R}^d$ and $0\le t < T$.
\end{proposition}

\begin{proof}
	Let $\frac{D}{Dt}:=\partial_t + {\bf v}({\bf x},t) \nabla \cdot $ be the directional derivative along the direction
	\begin{equation}\label{eq:curve}
	\frac{d {\bf x}}{d t} = {\bf v} ({\bf x},t).  
	\end{equation}
	For any $\left( \bar {\bf x}, \bar t \right) \in \mathbb R^d \times \mathbb R_+$, let ${\bf x}={\bf x}(t;\bar{\bf x},\bar t)$ be the integral curve of \eqref{eq:curve} through the point $\left( \bar {\bf x}, \bar t \right)$. 
	Denote ${\bf x}_0( \bar{\bf x},\bar t ):=  {\bf x}(0;\bar{\bf x},\bar t)$, then,  
	at $t=0$, the curve passes through the point $\left( {\bf x}_0( \bar{\bf x},\bar t ),0 \right)$. 
	Recall that, for smooth solutions, 
	the first equation of the system \eqref{eq:MHD:GP} can be reformulated as 
	\begin{equation}\label{eq:aa12}
	\frac{D \rho}{Dt} = - \rho \nabla \cdot {\bf v}. 
	\end{equation}
	Integrating Eq.~\eqref{eq:aa12} along the curve ${\bf x}={\bf x}(t;\bar{\bf x},\bar t)$ gives 
	$$
	\rho( \bar {\bf x}, \bar t  ) = \rho_0 ( {\bf x}_0( \bar{\bf x},\bar t ) ) \exp \left( - \int_{0}^{\bar t} \nabla \cdot {\bf v}( {\bf x}(t;\bar{\bf x},\bar t), t ) dt \right) > 0. 
	$$
	For smooth solutions, 
	we derive from the modified system \eqref{eq:MHD:GP} the pressure equation 
	\begin{equation}\label{eq:pppp}
	\frac{D p}{Dt} = - \gamma p \nabla \cdot {\bf v}, 
	\end{equation}
	which implies 
	$
	p( \bar {\bf x}, \bar t  ) = p_0 ( {\bf x}_0( \bar{\bf x},\bar t ) ) \exp \left( - \gamma \int_{0}^{\bar t} \nabla \cdot {\bf v}( {\bf x}(t;\bar{\bf x},\bar t), t ) dt \right) > 0. 
	$
\end{proof}

\begin{remark}
	By similar arguments one can show that the above proposition also holds for the modified MHD equations introduced by Janhunen \cite{Janhunen2000}, 
	because the corresponding equations for density and pressure are exactly also 
	\eqref{eq:aa12} and \eqref{eq:pppp}, respectively. This may explain why 
	it is also possible to develop PP schemes based on proper discretization of 
	Janhunen's MHD system, cf.~\cite{Janhunen2000,Bouchut2010,Waagan2009,waagan2011robust}.
\end{remark}

Recall that the pressure equation associated with the conservative system \eqref{eq:MHD} is 
\begin{equation*}
\frac{D p}{Dt} = - \gamma p \nabla \cdot {\bf v} - (\gamma-1)({\bf v}\cdot {\bf B}) \nabla \cdot {\bf B},
\end{equation*}
which, in comparison with \eqref{eq:pppp}, has an additional term proportional to $\nabla \cdot {\bf B}$. 
As shown in \cite{WuShu2018}, due to this term, negative pressure can appear in the exact smooth solution
of the conservative MHD system \eqref{eq:MHD} if $\nabla \cdot {\bf B} \neq 0$.

\section{Review of the positivity-preserving limiter}\label{app:PPlimit}

We employ a simple PP limiter to enforce the condition 
\eqref{eq:1DDG:con2} or \eqref{eq:FVDGsuff} for our 1D or 2D PP schemes. 
The limiter was originally proposed by  
Zhang and Shu \cite{zhang2010,zhang2010b,zhang2011} for scalar conservation laws and the compressible Euler equations. 
It was extended to the ideal MHD case in \cite{cheng}. 
For readers' convenience, we here briefly review this limiter. 
It is worth noting that the PP limiter works only when the cell averages of 
the numerical solutions always stay in ${\mathcal G}$. 
This is rigorously proved for our PP high-order schemes, but does not always hold for the standard multidimensional DG schemes without the proposed penalty term.

We perform the PP limiter separately for each cell. Let $K$ denote a cell, 
and ${\mathbb S}_K$ be the quadrature points involved in the condition \eqref{eq:1DDG:con2} or \eqref{eq:FVDGsuff} in $K$.  Let ${\bf U}_K^n({\bf x})$ be the  
approximate polynomial solution within $K$, and $\bar{\bf U}_K^n$ be the cell average 
which is always preserved in $\mathcal G$ by our PP schemes. If ${\bf U}_K^n({\bf x}) \notin {\mathcal G}$ for some ${\bf x} \in {\mathbb S}_K $, then we seek
 the modified polynomial $\widetilde {\bf U}_K^n({\bf x})$  with the same cell average such that $\widetilde {\bf U}_K^n({\bf x}) \in {\mathcal G}$ for all ${\bf x} \in {\mathbb S}_K $. 
To avoid the effect of the rounding
error, we introduce two sufficiently small positive numbers, $\epsilon_1$ and $\epsilon_2$, 
as the desired lower bounds for density and internal energy, respectively, such that 
$\bar{\bf U}_K^n \in {\mathcal G}_\epsilon= \{ {\bf U}: \rho \ge \epsilon_1,
{\mathcal E}(  {\bf U}  ) \ge \epsilon_2 \}$; e.g., take $\epsilon_1= \min\{10^{-13},\bar \rho_K^n \}$ and $\epsilon_2=\min\{ 10^{-13},{\mathcal E}(\bar{\bf U}_K^n)\}$. 

The PP limiting procedure consists of two steps. First, modify the density to enforce the positivity by 
		$$
	\widehat \rho_K({\bf x}) = \theta_1 ( \rho_K^n( {\bf x}) - \bar \rho_K^n ) + \bar \rho_K^n,\quad \theta_1 = \min\left\{1, \frac{\bar \rho_K^n - \epsilon_1}{ \bar \rho_K^n - \min_{{\bf x}\in {\mathbb S}_K} \rho_K^n({\bf x}) }\right\}.
	$$
Then modify $\widehat{\bf U}_K({\bf x}) := 
	( \widehat \rho_K({\bf x}), {\bf m}_K^n({\bf x}), {\bf B}_K^n({\bf x}), E_K^n({\bf x})  )^\top$ to enforce the positivity of internal energy by 
		$$
	\widetilde {\bf U}_K^n({\bf x}) = \theta_2 ( \widehat {\bf U}_K({\bf x}) -\bar{ \bf  U }_K^n ) + \bar{ \bf  U }_K^n,\quad \theta_2 = \min\left\{1, \frac{ {\mathcal E} ( \bar {\bf U}_K^n ) - \epsilon_2 }{ {\mathcal E} ( \bar {\bf U}_K^n ) - \min_{ {\bf x}\in {\mathbb S}_K} {\mathcal E} \big( \widehat {\bf U}_K({\bf x}) \big) }\right\}.
	$$
It is easy to verify that $\widetilde {\bf U}_K^n({\bf x})$ belongs 
to ${\mathcal G}_\epsilon$ for all ${\bf x} \in {\mathbb S}_K$ and 
has the cell average $\bar{\bf U}_K^n$. Such a limiter can also maintain the approximation accuracy; see  
\cite{zhang2010,zhang2010b,ZHANG2017301}.

\bibliographystyle{siamplain}
\bibliography{references}

\end{document}